\documentclass[11pt, twoside, leqno]{article}

\usepackage{amsfonts}

\usepackage{amssymb}
\usepackage{amsmath}
\usepackage{amsthm}
\usepackage{xcolor}
\usepackage{mathrsfs}

\allowdisplaybreaks

\def\rr{{\mathbb R}}

\def\zz{{\mathbb Z}}
\def\cc{{\mathbb C}}
\def\nn{{\mathbb N}}

\def\cb{{\mathcal B}}

\def\cg{{\mathcal G}}
\def\ch{{\mathcal H}}

\def\cl{{\mathcal L}}
\def\cm{{\mathcal M}}

\def\cs{{\mathcal S}}

\def\cx{{\mathcal X}}
\def\cy{{\mathcal Y}}
\def\sca{{\mathscr A}}
\def\scb{{\mathscr B}}
\def\scc{{\mathscr C}}
\def\scg{{\mathscr G}}
\def\sch{{\mathscr H}}
\def\sci{{\mathscr I}}
\def\scl{{\mathscr L}}

\def\sct{{\mathscr T}}
\def\scx{{\mathscr X}}
\def\scy{{\mathscr Y}}

\def\fz{\infty}
\def\az{\alpha}
\def\bz{\beta}
\def\dz{\delta}

\def\ez{\epsilon}
\def\gz{{\gamma}}

\def\lz{\lambda}

\def\oz{\omega}
\def\boz{\Omega}
\def\tz{\theta}

\def\vz{\varphi}

\def\lf{\left}
\def\r{\right}

\def\hs{\hspace{0.25cm}}
\def\ls{\lesssim}

\def\ov{\overline}
\def\noz{\nonumber}
\def\wz{\widetilde}

\def\st{\subset}

\def\bh{\backslash}

\def\supp{\mathop\mathrm{\,supp\,}}

\def\loc{{\mathop\mathrm{loc\,}}}
\def\diam{\mathop\mathrm{\,diam\,}}
\def\at{{\mathop\mathrm{at}}}

\def\lon{L^1(\mathcal{X})}
\def\ltw{L^2(\mathcal{X})}
\def\llo{L^{\rm log}(\mathcal{X})}

\def\lq{L^q(\mathcal{X})}
\def\li{L^{\infty}(\mathcal{X})}

\def\bmo{\mathop\mathrm{\,{\rm BMO}(\mathcal{X})}}
\def\bmop{\mathop\mathrm{\,{\rm BMO}^{+}(\mathcal{X})}}

\def\hona{H^1_\at (\mathcal{X})}

\def\hfa{H^1_{\at,\,{\rm fin}}(\mathcal{X})}
\def\hlo{H^{\rm log}(\mathcal{X})}
\def\lip{{\mathop\mathrm{\,Lip}}}

\newtheorem{theorem}{Theorem}[section]
\newtheorem{lemma}[theorem]{Lemma}
\newtheorem{corollary}[theorem]{Corollary}
\newtheorem{proposition}[theorem]{Proposition}
\theoremstyle{definition}
\newtheorem{remark}[theorem]{Remark}
\newtheorem{definition}[theorem]{Definition}

\numberwithin{equation}{section}

\begin{document}

\arraycolsep=1pt

\title{\Large\bf Products of Functions in
$\mathop\mathrm{BMO}({\mathcal X})$ and
$H^1_{\rm at}({\mathcal X})$ via Wavelets over
Spaces of Homogeneous Type \footnotetext {\hspace{-0.35cm}
2010 {\it Mathematics Subject Classification}. Primary 42B30;
Secondary 47A07, 42C40, 30L99.
\endgraf {\it Key words and phrases}.
metric measure space of homogeneous type,
product, Hardy space, $\mathop\mathrm{BMO}({\mathcal X})$,
regular wavelet, spline function, bilinear operator, paraproduct.
\endgraf
Dachun Yang is supported by the National
Natural Science Foundation of China
(Grant Nos.~11571039 and 11361020),
the Specialized Research Fund for the Doctoral Program of Higher Education
of China (Grant No. 20120003110003) and the Fundamental Research
Funds for Central Universities of China
(Grant Nos.~2013YB60 and 2014KJJCA10).}}
\author{Xing Fu, Dachun Yang\,\footnote{Corresponding author}
\ \ and Yiyu Liang}
\date{ }
\maketitle

\vspace{-0.8cm}

\begin{center}
\begin{minipage}{13cm}
{\small {\bf Abstract}\quad Let $({\mathcal X},d,\mu)$
be a metric measure space of homogeneous type in the sense of
R. R. Coifman and G. Weiss and $H^1_{\rm at}({\mathcal X})$ be the atomic
Hardy space. Via orthonormal bases of regular wavelets
and spline functions recently constructed by P. Auscher and
T. Hyt\"onen, the authors prove that
the product $f\times g$ of $f\in H^1_{\rm at}({\mathcal X})$
and $g\in\mathop\mathrm{BMO}({\mathcal X})$, viewed as a distribution,
can be written into a sum of two bounded bilinear operators, respectively,
from $H^1_{\rm at}({\mathcal X})\times\mathop\mathrm{BMO}({\mathcal X})$ into
$L^1({\mathcal X})$ and from $H^1_{\rm at}({\mathcal X})
\times\mathop\mathrm{BMO}({\mathcal X})$
into $H^{\log}({\mathcal X})$, which
affirmatively confirms the conjecture suggested by
A. Bonami and F. Bernicot
(This conjecture was presented by L. D. Ky in
[J. Math. Anal. Appl. 425 (2015), 807-817]).}
\end{minipage}
\end{center}

\section{Introduction}\label{s1}

\hskip\parindent Let $H^1(\rr^D)$ and $\mathop\mathrm{BMO}(\rr^D)$
be, respectively, the classical Hardy space and the space of functions
with bounded mean oscillations on $\rr^D$ endowed with the $D$-dimensional
Lebesgue measure.  As is well known,
the pointwise product $fg$ for $f\in H^1(\rr^D)$
and $g\in \mathop\mathrm{BMO}(\rr^D)$ may not be meaningful,
since this pointwise product is not locally integrable on $\rr^D$ in general
(see \cite{bijz} for the details).
Nevertheless, we can view  such a product as a distribution,
which is denoted by $f\times g$ (see \cite{bijz} again).

In 2007, Bonami et al. \cite{bijz}
systematically studied the product of $f\in H^1(\rr^D)$
and $g\in \mathop\mathrm{BMO}(\rr^D)$, which is viewed as a Schwartz
distribution and can be further written as a sum of
an integrable function and a distribution in some adapted
Hardy-Orlicz space. Recently, Bonami, Grellier and Ky \cite{bgk}
essentially improved this result in the following two main aspects.

The first aspect is that the aforementioned Hardy-Orlicz space
can be replaced by a smaller space $H^{\log}(\rr^D)$, which is a
particular case of Musielak-Orlicz-type Hardy spaces
originally introduced by Ky \cite{k14}.
Via the main theorem of Nakai and Yabuta \cite{ny},
Bonami, Grellier and Ky \cite{bgk} further showed that
$H^{\log}(\rr^D)$ is optimal in the sense that it can not
be replaced by a smaller space; see \cite{bgk,k14} for the
details. For more properties on Musielak-Orlicz-type Hardy spaces,
we refer the reader to
\cite{hyy,k14,lhy,lnyz,ly13,ly15,yy}.

Secondly, Bonami, Grellier and Ky \cite{bgk}
answered a question of \cite{bijz}
via showing that $f\times g$ can be written into a sum of
two bilinear bounded operators, respectively, from
$H^1(\rr^D)\times \mathop\mathrm{BMO}(\rr^D)$
into $L^1(\rr^D)$ and from $H^1(\rr^D)\times \mathop\mathrm{BMO}(\rr^D)$
into $H^{\log}(\rr^D)$. As a consequence, they
obtained an optimal endpoint estimate involving the space $H^{\log}(\rr^D)$
for the div-curl lemma, which is related to an implicit conjecture from
\cite{bijz} (see also \cite{bgk,bfg}). Moreover,
the above decomposition of the product plays an essential role in establishing
the bilinear or the subbilinear decompositions, respectively,
for the linear or the sublinear commutators of
singular integrals by Ky \cite{k13}.
For more applications of the above decompositions,
we refer the reader to \cite{ky2,ky-15}.

As is well known, many classical results
of harmonic analysis over Euclidean
spaces can be extended to spaces of homogeneous type
in the sense of Coifman and Weiss \cite{cw71,cw77}, or to
the RD-space introduced by Han, M\"uller and Yang \cite{hmy08}
(see also \cite{hmy06,yz11}).

Recall that a quasi-metric space $(\cx, d)$ equipped
with a nonnegative measure
$\mu$ is called a {\it space of homogeneous type}
in the sense of Coifman and Weiss \cite{cw71,cw77}
if $(\cx, d,\mu)$ satisfies the following {\it measure doubling condition}:
there exists a positive constant $C_{(\cx)}\in[1,\fz)$ such that,
for all balls
$B(x,r):= \{y\in\cx:\,\, d(x, y)< r\}$
with $x\in\cx$ and $r\in(0, \fz)$,
\begin{equation*}
\mu(B(x, 2r))\le C_{(\cx)} \mu(B(x,r)),
\end{equation*}
which further implies that there exists a
positive constant $\wz C_{(\cx)}$ such that,
for all $\lz\in[1,\fz)$,
\begin{equation}\label{a.b}
\mu(B(x, \lz r))\le \wz C_{(\cx)}\lz^{n} \mu(B(x,r)),
\end{equation}
where $n:=\log_2 C_{(\cx)}$. Let
\begin{equation}\label{n0}
n_0:=\inf\{n\in(0,\fz):\ n\ {\rm satisfies}\ (\ref{a.b})\}.
\end{equation}
Observe that $n_0$ measures the
dimension of $\cx$ in some sense, $ n_0\le n$ and
\eqref{a.b} with $n$ replaced by $n_0$ may
not hold true.

It is known that a space of homogeneous type, $(\cx, d,\mu)$,
is called a \emph{metric measure space of homogeneous type}
in the sense of Coifman and Weiss if $d$ is a metric and, moreover, 
a space of homogeneous type, $(\cx,d,\mu)$, is
called an RD-\emph{space} if it satisfies the following additional \emph{reverse
doubling condition} (see \cite{hmy08}): there exist positive constants
$a_0,\ {\widehat C}_{(\cx)}\in(1,\fz)$ such that, for all balls $B(x,r)$
with $x\in\cx$ and $r\in(0, \diam(\cx)/a_0)$,
$$\mu(B(x, a_0r))\ge {\widehat C}_{(\cx)} \mu(B(x,r))$$
(see \cite{yz11} for more equivalent characterizations of RD-spaces).
Here and hereafter,
$$\diam (\cx):=\sup\{d(x,y):\ x,\,y\in\cx\}.$$

Let $(\cx,d,\mu)$ be a space of homogeneous type.
Coifman and Weiss \cite{cw77} introduced the
atomic Hardy space $H^{p,\,q}_\at (\cx,d,\mu)$ for all
$p\in(0,1]$ and $q\in[1,\fz]\cap(p,\fz]$ and showed
that $H^{p,\,q}_\at (\cx,d,\mu)$ is independent
of the choice of $q$, which is hereafter simply denoted by
$H^p_\at(\cx,d,\mu)$, and that its dual space is the
Lipschitz space $\lip_{1/p-1}(\cx,d,\mu)$ when $p\in(0,1)$,
or the space $\mathop\mathrm{BMO}(\cx,d,\mu)$
of functions with bounded mean oscillations when $p=1$.
Coifman and Weiss \cite{cw77} also
introduced the \emph{measure distance $\rho$}
defined by setting, for all $x,\,y\in\cx$,
\begin{equation}\label{a.x}
\rho(x,y):=\inf\lf\{\mu\lf(B_d\r):\
B_d\ \mathrm{is\ a\ ball\ containing}\ x\ {\rm and}\ y\r\},
\end{equation}
where the infimum is taken over all balls in $(\cx,d,\mu)$ containing
$x$ and $y$; see also \cite{ms1}.
It is well known that, although all balls defined by $d$ satisfy the axioms
of the complete system of neighborhoods in $\cx$ [and hence induce a (separated)
topology in $\cx$], the balls $B_d$ are not necessarily
open with respect to the topology
induced by the quasi-metric $d$. However, Mac\'ias and Segovia \cite[Theorem 2]{ms1}
showed that there exists a quasi-metric $\wz{d}$ such that $\wz{d}$
is \emph{equivalent} to $d$, namely, there exists a positive
constant $C$ such that, for all $x,\,y\in\cx$,
$$
C^{-1}d(x,y)\le\wz{d}(x,y)\le Cd(x,y),
$$
and the balls in $(\cx,\wz{d},\mu)$ are open.

We also recall that a quasi-metric measure space $(\cx,\rho,\mu)$
is said to be \emph{normal} in \cite{ms1} if
there exists a fixed positive constant $C_{(\rho)}$
such that, for all $x\in\cx$ and $r\in(0,\fz)$,
$$
C_{(\rho)}^{-1}r\le\mu\lf(B_{\rho}(x,r)\r)\le C_{(\rho)}r.
$$

Assuming that all balls in $(\cx,d,\mu)$ are open,
Coifman and Weiss \cite[p.\,594]{cw77} claimed
that the topology of $\cx$ induced by $d$ coincides with that 
of $\cx$ induced by $\rho$ and
$(\cx,\rho,\mu)$ is a normal space,
which were rigorously proved by
Mac\'ias and Segovia in \cite[Theorem 3]{ms1},
and also that the atomic Hardy space $H^p_\at(\cx,d,\mu)$
associated with $d$ and
the atomic Hardy space $H^p_\at(\cx,\rho,\mu)$
associated with $\rho$ coincide with equivalent quasi-norms
for all $p\in(0,1]$.
Mac\'ias and Segovia \cite[Theorem 2]{ms1} further showed that
there exists a normal quasi-metric
$\wz{\rho}$, which is equivalent to $\rho$, such that 
$\wz{\rho}$ is \emph{$\tz$-H\"older continuous}
with $\tz\in(0,1)$, namely,
there exists a positive constant $C$
such that, for all $x,\,\wz{x},\,y\in\cx$,
$$
\lf|\wz{\rho}(x,y)-\wz{\rho}(\wz{x},y)\r|\le C
\lf[\wz{\rho}(x,\wz{x})\r]^{\tz}
\lf[\wz{\rho}(x,y)+\wz{\rho}(\wz{x},y)\r]^{1-\tz}.
$$
Via certain geometric measure relations
between $(\cx,d,\mu)$ and $(\cx,\rho,\mu)$,
Hu et al. \cite[Theorem 2.1]{hyz} rigorously showed
the claim of Coifman and Weiss \cite[p.\,594]{cw77}
on the coincidence of both $H^p_\at(\cx,d,\mu)$
and $H^p_\at(\cx,\rho,\mu)$, which was also used by Mac\'ias and
Segovia \cite[pp.\,271-272]{ms2}.

When $(\cx,\rho,\mu)$ is a normal quasi-metric measure space,
Coifman and Weiss \cite{cw77} further
established the molecular characterization
for $H^1_\at(\cx,\rho,\mu)$.
When $(\cx,\wz\rho,\mu)$ is a normal quasi-metric
measure space and $\wz{\rho}$ is $\tz$-H\"older
continuous, Mac\'ias and Segovia \cite{ms2}
obtained the grand maximal function characterization
for $H^p_\at(\cx,\wz\rho,\mu)$ with $p\in (\frac 1{1+\tz}, 1]$
via distributions acting on certain spaces of Lipschitz
functions; Han \cite{h94} obtained their
Lusin-area function characterization;
Duong and Yan \cite{dy03} then characterized these
atomic Hardy spaces in terms of Lusin-area functions
associated with some Poisson semigroups;
Li \cite{l98} also obtained a characterization of
$H^p_\at(\cx,\wz\rho,\mu)$ in terms of the grand maximal function
defined via test functions introduced in \cite{hs94}.

Over RD-spaces $(\cx,d,\mu)$ with $d$ being a metric,
for $p\in (\frac {n_0}{n_0+1},1]$
with $n_0$ as in \eqref{n0}, Han et al. \cite{hmy06}
developed a Littlewood-Paley theory
for atomic Hardy spaces $H^p_\at(\cx,d,\mu)$;
Grafakos et al. \cite{gly1} established their characterizations
in terms of various maximal functions.
Moreover, it was shown in \cite{hmy08} that these Hardy spaces
coincide with Triebel-Lizorkin spaces on $(\cx,d,\mu)$.
Some basic tools, including spaces of test functions,
approximations
of the identity and various Calder\'on reproducing formulas on RD-spaces,
were well developed in \cite{hmy06,hmy08},
in order to develop a real-variable theory of Hardy spaces or,
more generally, Besov spaces and Triebel-Lizorkin
spaces on RD-spaces.
From then on, these basic tools play important roles in
harmonic analysis on RD-spaces (see, for example,
\cite{glmy, gly, hmy06, hmy08, kyz10, kyz11, yz08, yz11}).

Let $(\cx, d,\mu)$ be an RD-space. The problem about the product of
$f\in H^1_\at(\cx,d,\mu)$
and $g\in \mathop\mathrm{BMO}(\cx,d,\mu)$ was first studied by
Feuto \cite{f09}. In \cite{f09}, Feuto showed that
the product of $f\in H^1_\at(\cx,d,\mu)$ and
$g\in \mathop\mathrm{BMO}(\cx,d,\mu)$, viewed as a
distribution, can be written as a sum of
an integrable function and a distribution in some adapted
Hardy-Orlicz space. Recently, Ky \cite{ky} improved the above
result via showing that the product $g\times f$ can
be written into a sum of two
linear operators and via replacing the Hardy-Orlicz space
by $H^{\log}(\cx,d,\mu)$ which
is a smaller space than the aforementioned
Hardy-Orlicz space and is known to be optimal
even when $\cx=\rr^D$ endowed with the $D$-dimensional Lebesgue measure.
A. Bonami and F. Bernicot further \emph{conjectured} that
$g\times f$ can be written into a sum of two
\emph{bilinear} operators, which was presented by Ky in
\cite[p.\,809,\ Conjecture]{ky}.

Recently, Auscher and Hyt\"onen \cite{ah13}
built an orthonormal basis
of H\"older continuous wavelets with exponential decay
via developing randomized dyadic structures and properties of
spline functions over general spaces of homogeneous type.
Fu and Yang \cite{fy1} further obtained an unconditional basis of
$H^1_{\rm at}({\mathcal X})$ and several
equivalent characterizations of $H^1_{\rm at}({\mathcal X})$
in terms of wavelets.
Motivated by \cite{ah13}, \cite{bgk} and \cite{fy1}, in this article,
we give an affirmative answer to the aforementioned conjecture
of Bonami and Bernicot on
a metric measure space $(\cx,d,\mu)$ of homogeneous type.

We point out that the main result of this article can be
used to study the end-point boundedness of commutators generated by
linear or sublinear operators (see \cite{k13} for the Euclidean case).
More applications are also possible
(see \cite{ky2,ky-15} for the Euclidean case). But
we will not consider these problems in this article
due to the length of the article.

Throughout this article, for the presentation simplicity,
we \emph{always assume} that $(\cx,d,\mu)$
is a metric measure of homogeneous type, $\diam (\cx)=\fz$
and $(\cx,d,\mu)$ is non-atomic, namely, $\mu(\{x\})=0$ for any
$x\in\cx$. It is known that, if $\diam (\cx)=\fz,$
then $\mu(\cx)=\fz$ (see, for example, \cite[Lemma 8.1]{ah13}).

To state the main result of this article, we first recall
the notion of the space of test functions on $\cx$,
whose following versions were
originally introduced  by Han, M\"uller and Yang
\cite[Definition 2.2]{hmy06} (see also \cite[Definition 2.8]{hmy08}).

\begin{definition}\label{da.d}
Let $x_1\in\cx$, $r\in(0,\fz)$,
$\bz\in(0,1]$ and $\gz\in(0,\fz)$.
A function $f$ on $\cx$ is said to belong to the
{\it space of test functions, $\cg(x_1,r,\bz,\gz)$}, if
there exists a non-negative constant $\wz C$ such that
\begin{itemize}

\item[(T1)] $|f(x)|\le \wz{C}\frac{1}{V_r(x_1)+V(x_1,x)}
[\frac{r}{r+d(x_1,x)}]^\gz$ for all $x\in\cx$;

\item[(T2)] $|f(x)-f(y)|\le \wz{C}[\frac{d(x,y)}{r+d(x_1,x)}]^\bz
\frac{1}{V_r(x_1)+V(x_1,x)}[\frac{r}{r+d(x_1,x)}]^\gz$
for all $x,\,y\in\cx$ satisfying $d(x,y)\le[r+d(x_1,x)]/2$,
\end{itemize}
where $V_r(x_1):=\mu(B(x_1,r))$ and $V(x_1,x):=\mu(B(x_1,d(x_1,x)))$.
Moreover, for $f\in\cg(x_1,r,\bz,\gz)$, its norm is defined by setting
$$
\|f\|_{\cg(x_1,\,r,\,\bz,\,\gz)}:=\inf\lf\{\wz{C}:\ {\wz C}\
{\rm satisfies\ (T1)\ and\ (T2)}\r\}.
$$
\end{definition}

Fix $x_1\in\cx$. It is obvious that $\cg(x_1,1,\bz,\gz)$
is a Banach space. For the notational simplicity, we write
$\cg(\bz,\gz)$ instead of $\cg(x_1,1,\bz,\gz)$

For any given $\ez\in(0,1]$, let $\cg^\ez_0(\bz,\gz)$
be the completion of the set $\cg(\ez,\,\ez)$
in $\cg(\bz,\gz)$ when $\bz,\gz\in(0,\,\ez]$.
Moreover, if $f\in\cg^\ez_0(\bz,\gz)$, we then let
$\|f\|_{\cg^\ez_0(\bz,\gz)}:=\|f\|_{\cg(\bz,\gz)}$.
Recall that the \emph{dual space}
$(\cg^\ez_0(\bz,\gz))^*$ is defined
to be the set of all continuous linear functionals $\cl$ from
$\cg^\ez_0(\bz,\gz)$ to $\cc$ and
endowed with the weak-$*$ topology.

We remark that, for any $x\in\cx$ and $r\in(0,\fz)$,
$\cg(x,r,\bz,\gz)=\cg(x_1,1,\bz,\gz)$ with equivalent norms and
the equivalent positive constants depending on $x$ and $r$.

The following notion of the space $\bmo$ is from \cite{cw77}.

\begin{definition}\label{da.e}
The space $\bmo$ is defined to be
the class of all functions, $b\in L^1_{\loc}(\cx)$,
satisfying
$$
\|b\|_{\bmo}:=\sup_{B}\frac1{\mu(B)}\int_{B}|b(x)-m_{B}(b)|\,d\mu(x)<\fz,
$$
where the infimum is taken over all balls $B$ in $\cx$
and $m_B(b):=[\mu(B)]^{-1}\int_B b\,d\mu$.
\end{definition}

Now we recall the following notion of Hardy spaces $\hona$,
which was introduced in \cite{cw77}.

\begin{definition}\label{dc.k}
Let $q\in(1, \fz]$. A function $a$ on
$\cx$ is called  a {\it $(1, q)$-atom} if

(i) $\supp(a)\subset B$ for some ball $B\subset\cx$;

(ii) $\|a\|_{L^q(\cx)}\le [\mu(B)]^{1/q-1}$;

(iii) $\int_\cx a(x)\, d\mu(x)=0$.

A function $f\in L^1(\cx)$ is said to be in the {\it Hardy space
$H_{\rm at}^{1,\,q}(\cx)$} if there exist
$(1,q)$-atoms $\{a_j\}_{j=1}^\fz$ and numbers
$\{\lz_j\}_{j=1}^\fz\subset\cc$ such that
\begin{equation}\label{c.z}
f=\sum_{j=1}^\fz\lz_j a_j,
\end{equation}
which converges in $L^1(\cx)$, and
$$\sum_{j=1}^\fz|\lz_j|<\fz.$$
Moreover, the norm of $f$ in $H_{\rm at}^{1,\,q}(\cx)$
is defined by setting
$$
\|f\|_{H_{\rm at}^{1,\,q}(\cx)}
:=\inf\lf\{\sum_{j\in\nn}|\lz_j|\r\},
$$
where the infimum is taken over all possible decompositions of
$f$ as in \eqref{c.z}.
\end{definition}

Coifman and Weiss \cite{cw77}
proved that $H_{\rm at}^{1,\,q}(\cx)$ and
$H^{1,\,\fz}_{\rm at}(\cx)$ coincide with equivalent norms for all
different $q\in(1, \fz)$.
Thus, from now on, we denote $H^{1,\,q}_{\rm at}(\cx)$
simply by $H^1_{\rm at}(\cx)$.

\begin{remark}\label{rc.i}
It was shown in \cite{cw77} that $\hona$ is a Banach space
which is the predual of $\mathop\mathrm{BMO}(\cx)$.
\end{remark}

We also need to recall
some notions and results from \cite{ky}.

Let $\llo$ denote the \emph{Musielak-Orlicz-type space}
of $\mu$-measurable functions $f$ such that
$$
\int_{\cx}\frac{|f(x)|}{\log(e+|f(x)|)+\log(e+d(x_0,x))}
\,d\mu(x)<\fz;
$$
see \cite{ky}. For all $f\in\llo$, the norm of $f$ is defined by setting
$$
\|f\|_{\llo}:=\inf\lf\{\lz\in(0,\fz):\
\int_{\cx}\frac{|f(x)|/\lz}{\log(e+|f(x)|/\lz)+\log(e+d(x_0,x))}
\,d\mu(x)\le1\r\}.
$$

\begin{remark}\label{rf.x}
It is easy to see that $\lon\st\llo$ and, for all $f\in\lon$,
$$
\|f\|_{\llo}\le\|f\|_{\lon}.
$$
\end{remark}

Let $\ez\in(0,1]$, $\bz,\,\gz\in(0,\ez]$ and $f\in(\cg^\ez_0(\bz,\gz))^*$.
The \emph{grand maximal function}
$\cm(f)$ is defined by setting, for all $x\in\cx$,
\begin{equation}\label{a.u}
\cm(f)(x):=\sup\lf\{|\langle f,h\rangle|:\
h\in\cg^\ez_0(\bz,\gz),\ \|h\|_{\cg(x,\,r,\,\bz,\,\gz)}\le1
\ {\rm for\ some\ }r\in(0,\fz)\r\}.
\end{equation}

The following notion of Musielak-Orlicz-type
Hardy spaces is from \cite{ky}.

\begin{definition}\label{df.w}
Let $\ez\in(0,1]$ and $\bz,\,\gz\in(0,\ez]$.
The \emph{Hardy space $\hlo$} is defined by setting
$$
\hlo:=\lf\{f\in(\cg^\ez_0(\bz,\gz))^*:\
\|f\|_{\hlo}:=\|\cm(f)\|_{\llo}<\fz\r\}.
$$
\end{definition}

We also need to illustrate the
meaning of the product $f\times g$ for every $f\in\hona$ and
$g\in\bmo$ (see \cite{ky}).
For any $h\in\cg^\ez_0(\bz,\gz)$, let
$$
\langle f\times g, h\rangle:=\langle gh,f\rangle
:=\int_\cx ghf\,d\mu.
$$
From \cite[Proposition 3.1]{ky} (see also Lemma \ref{lf.t} below),
it follows that
$gh\in\bmo$ and hence the above definition is well defined
in the sense of the duality between $\hona$ and $\bmo$.

Now we state the main result of this article as follows.

\begin{theorem}\label{ta.a}
Let $(\cx,d,\mu)$ be a metric measure space of homogeneous type. Then
there exist two bounded bilinear operators $\scl:\ \hona\times\bmo
\to\lon$ and $\sch:\ \hona\times\bmo\to \hlo$
such that, for all $f\in\hona$ and $g\in\bmo$,
$$
f\times g=\scl(f,g)+\sch(f,g)\quad {\rm in}
\quad \lf(\cg^{\ez}_0(\bz,\gz)\r)^*,
$$
where $\ez\in(0,1]$ and $\bz,\,\gz\in(0,\ez]$.
\end{theorem}

We show Theorem \ref{ta.a} by borrowing some ideas from \cite{bgk}.
Indeed, for all $f,\,g\in\ltw$ with finite wavelet decompositions,
we write the pointwise product $fg$ into a sum in $\ltw$ of three
bilinear operators, $\Pi_1(f,g)$, $\Pi_2(f,g)$ and $\Pi_3(f,g)$,
which are called \emph{paraproducts}
(see Lemma \ref{le.a} below).
We then investigate their boundedness separately.

For $\Pi_3$, from the orthonormal basis of regular wavelets,
it follows easily that $\Pi_3$ is bounded from $\ltw\times\ltw$
into $\lon$ (see Lemma \ref{le.b} below).
We then give the meaning of $\Pi_3(a,g)$
for any $(1,2)$-atom $a$ and $g\in\bmo$.
This is essentially different from the proof of \cite[Theorem 5.2]{bgk}
on Euclidean spaces via compactly supported wavelets, where $\Pi_3(a,g)$
can be written into just one part which
is estimated by the aforementioned boundedness of $\Pi_3$.
However, due to the lack of the compact supports of wavelets
constructed by Auscher and Hyt\"onen \cite{ah13},
we need to decompose $\Pi_3(a,g)$ into three parts via
the wavelet characterizations of $\bmo$ from
Auscher and Hyt\"onen \cite[Theorem 11.4]{ah13}
(see the proof of Theorems \ref{tf.b} below).
On the estimate of the first part, we use the boundedness
of $\Pi_3(a,g)$ from $\ltw\times\ltw$ into $\lon$.
Other two parts are estimated directly via the exponential
decay of the wavelets and a related very useful
technical lemma (namely, \cite[Lemma 6.4]{ah13} which is re-stated
as Lemma \ref{lb.x} below).
Finally, in Theorem \ref{tf.b} below, we extend
$\Pi_3$ to a bounded bilinear operator
from $\hona\times\bmo$ into $\lon$ by establishing a criterion
on the boundedness of a sublinear operator from
$\hona$ into a quasi-Banach space (see Theorem \ref{td.f}
below), where the equivalence of norms on finite linear combinations
of atoms obtained by Mauceri and Meda \cite{mm11} (see
Theorem \ref{tf.z} below) plays a crucial role.

For $\Pi_1$, to show that $\Pi_1$ is bounded
from $\ltw\times\ltw$ into $\hona$ (see Lemma \ref{le.c} below),
we first write $\Pi_1(f,g)$ into a multiple sum
via the spline expansion of $f$ and
the wavelet expansion of $g$. A key point here is,
via smartly relabeling the index of summations of this sum
[see \eqref{e.g} below], we exchange the order
of summations of this sum so that $\Pi_1(f,g)$
becomes a sum of a sequence of Calder\'on-Zygmund
operators acting on some functions which are proved to belong to $\hona$
by the equivalent characterizations of $\hona$ in terms
of wavelets obtained in \cite{fy1} (see also Theorem \ref{tc.d} below), which 
play key roles in this step. Thus, the desired boundedness of $\Pi_1$ follows from
the  boundedness in $\hona$ of Calder\'on-Zygmund
operators and a technical lemma (see Lemma \ref{lc.i} below).
Similar to $\Pi_3$, we give the meaning of
$\Pi_1(a,g)$ by decomposing it into three parts via
the wavelet characterizations of $\bmo$
(see the proof of Theorem \ref{tf.c} below),
which is much more sophisticated
than the Euclidean case due to the lack of the compact supports of wavelets.
For the estimate of the first part, we use the boundedness
of $\Pi_1(a,g)$ from $\ltw\times\ltw$ into $\hona$.
Other two parts are estimated directly via the
vanishing moments of the wavelets, the properties of spline functions,
the orthogonality, the exponential
decay of the wavelets and \cite[Lemma 6.4]{ah13}.
Finally, we extend $\Pi_1$ to a bounded bilinear operator
from $\hona\times\bmo$ into $\hona$ via a way similar to that
used for $\Pi_3$ (see Theorem \ref{tf.c} below).

Now we turn to $\Pi_2$. The boundedness of $\Pi_2$
from $\ltw\times\ltw$ into $\hona$ follows from the boundedness of
$\Pi_1$, since
$\Pi_2(f,g)=\Pi_1(g,f)$
for all $f,\,g\in\ltw$.
To define $\Pi_2(a,g)$ for any
$(1,2)$-atom $a$ and $g\in\bmo$, we decomposing it into three parts
(see the proof of Theorem \ref{tf.d} below), which are different
from those of $\Pi_1$ and $\Pi_3$.
For the estimate of the first part, we use the boundedness
of $\Pi_1(a,g)$ from $\ltw\times\ltw$ into $\hona$ and the fact
$\hona\st\hlo$ (see Corollary \ref{cf.o} below).
The second part is estimated directly via the
vanishing moments of the wavelets, the properties of spline functions,
the orthogonality, the exponential
decay of the wavelets, \cite[Lemma 6.4]{ah13}
and the fact that $\hona\st\hlo$.
For the third part, we need some useful properties
established in \cite{ky} (see also Lemmas \ref{lf.v}
and \ref{lf.u} below, observing these properties are
obtained without resorting to the reverse
doubling property of $\cx$) and an observation that
$\Pi_2(a,1)=a$ [see \eqref{5.10y} below]
via the properties of spline functions.
The last observation much simplifies the estimates for $\Pi_2(a,g)$
(see some similar estimates in \cite[Lemma 5.1]{bgk} for the Euclidean case,
which strongly depend on the compact support property of wavelets).
Finally, we extend $\Pi_2$ to a bounded bilinear operator
from $\hona\times\bmo$ into $\hlo$ by a way similar to that
used for $\Pi_3$ (see Theorem \ref{tf.d} below).

By these boundedness of $\Pi_1$, $\Pi_2$ and $\Pi_3$,
together with the decompositions for products
in $\ltw$, we conclude that, for all $(1,2)$-atoms $a$ and $g\in\bmo$,
$$
a\times g=\Pi_1(a,g)+\Pi_2(a,g)+\Pi_3(a,g)\quad {\rm in}
\quad \lf(\cg^{\ez}_0(\bz,\gz)\r)^*,
$$
which is also different from the result in \cite[Theorem 1.1]{bgk}
with compactly supported wavelets, where the above equality
was proved to hold true in $\ltw$ with $g$ replaced by $g\eta$, where
$\eta$ is a smooth function with compact support.
This, combined with the boundedness of $\Pi_1$, $\Pi_2$ and
$\Pi_3$ and a standard density argument, then
finishes the proof of Theorem \ref{ta.a},
where the wavelet characterizations of $\bmo$
and the fact $\Pi_2(a,1)=a$ again
play crucial roles.

\begin{remark}\label{ra.b}
(i) When $(\cx,d,\mu):=(\rr^D,|\cdot|,dx)$ is the Euclidean
space associated with the $D$-dimensional Lebesgue measure,
Theorem \ref{ta.a} then coincides with \cite[Theorem 1.1]{bgk}.

(ii) As in \cite{k13} on the Euclidean space,
the main result of this article can be applied
to investigate the end-point boundedness of commutators generated by
linear or sublinear operators, since these
commutators can be controlled by two bilinear or subbilinear
bounded operators, which are constructed
via $\Pi_1$, $\Pi_2$ and $\Pi_3$, the details
being presented in another article.
More applications to function spaces associated with operators on spaces
of homogeneous type are also possible
(see \cite{ky2,ky-15} for the Euclidean case).
Also, in a forthcoming paper, we will develop a
complete real-variable theory for Musielak-Orlicz Hardy
spaces, which include $\hlo$ as a special case, on spaces of homogeneous type.

(iii) In what follows, if the metric $d$ is replaced by a
quasi-metric $\rho$, by arguments essentially the same
as those used in the case of $d$,
we conclude that all the results obtained in this article
\emph{remain} valid, since most of the tools we need are from
\cite{ah13,ky} (see also \cite{ah15}),
which were established in the context of
spaces of homogeneous type.
Some minor modifications are needed when dealing with
the inclusion relations between two balls, where
the quasi-triangle constant is involved,
which only alter the corresponding results
by additive positive constants via \eqref{a.b}.
\end{remark}

The organization of this paper is as follows.

In Section \ref{s2}, we recall
some preliminary notions and the equivalent characterizations of $\hona$ in terms of wavelets from \cite{fy1}.

In Section \ref{s4}, we establish the $\ltw$ estimates
for the product of two functions in $\ltw$
and, in Section \ref{s5}, we finish the
proof of Theorem \ref{ta.a}.

Finally, we make some conventions on notation.
Throughout the whole paper, $C$ stands for a {\it positive constant} which
is independent of the main parameters, but it may vary from line to
line. Moreover, we use $C_{(\rho,\,\az,\,\ldots)}$
to denote a positive constant depending
on the parameters $\rho,\,\az,\,\ldots$.
Usually, for a ball $B$, we use $c_{B}$ and $r_{B}$, respectively, to denote
its center and radius. Moreover, for any
$x,\,y\in\cx$, $r,\,\rho\in(0,\fz)$ and ball $B:=B(x,r)$,
$$
\rho B:=B(x,\rho r), \quad V(x,r):=\mu(B(x,r))=:V_r(x),
\quad V(x,y):=\mu(B(x,d(x,y))).
$$
If, for two real functions $f$ and $g$, $f\le Cg$, we then write $f\ls g$;
if $f\ls g\ls f$, we then write $f\sim g$.
From \eqref{a.b}, it easily follows that, for any $x,\,y\in\cx$ and
$r\in(0,\fz)$, if $d(x,y)\le r$, then
\begin{equation}\label{a.e}
V(x,r)\sim V(y,r)
\end{equation}
with the equivalent positive constants independent of $x$, $y$ and $r$.
For any subset $E$ of $\cx$, we use
$\chi_E$ to denote its {\it characteristic function}.
Furthermore, $\langle\cdot,\cdot\rangle$ and $(\cdot,\cdot)$
represent the duality relation and the $\ltw$ inner product,
respectively.

\section{Preliminaries}\label{s2}

\hskip\parindent In this section, we first recall
some preliminary notions and some
lower bounds for regular wavelets and then introduce the
equivalent characterizations of $\hona$ in terms of wavelets from \cite{fy1}.

The following notion of the geometrically doubling
is well known in analysis on metric spaces,
for example, it can be found in
Coifman and Weiss \cite[pp.\,66-67]{cw71}.

\begin{definition}\label{db.b}
A metric space $(\cx,d)$ is said to be \emph{geometrically doubling} if there
exists some $N_0\in \nn$ such that, for any ball
$B(x,r)\st \cx$ with $x\in\cx$ and $r\in(0,\fz)$,
there exists a finite ball covering $\{B(x_i,r/2)\}_i$ of
$B(x,r)$ such that the cardinality of this covering is at most $N_0$,
where, for all $i$, $x_i\in\cx$.
\end{definition}

\begin{remark}\label{rb.l}
Let $(\cx,d)$ be a geometrically doubling metric space.
In \cite{h10}, Hyt\"onen showed that the geometrically doubling condition
is equivalent to each of following statements:
\vspace{-0.25cm}
\begin{itemize}
  \item[\rm(i)] For any $\ez\in (0,1)$ and any ball $B(x,r)\st \cx$
with $x\in\cx$ and $r\in(0,\fz)$,
there exists a finite ball covering $\{B(x_i,\ez r)\}_i$,
with $x_i\in\cx$ for all $i$, of
$B(x,r)$ such that the cardinality of this covering
is at most $N_0\ez^{-G_0}$,
here and hereafter, $N_0$ is as in Definition \ref{db.b} and
$G_0:=\log_2N_0$.
\vspace{-0.25cm}
  \item[\rm(ii)] For every $\ez\in (0,1)$, any ball $B(x,r)\st \cx$
with $x\in\cx$ and $r\in(0,\fz)$ contains
at most $N_0\ez^{-G_0}$ centers of disjoint balls $\{B(x_i,\ez r)\}_i$
with $x_i\in\cx$ for all $i$.
\vspace{-0.25cm}
  \item[\rm(iii)] There exists $M\in \nn$ such that any ball $B(x,r)\st \cx$
with $x\in\cx$ and $r\in(0,\fz)$ contains at most $M$ centers
$\{x_i\}_i\st\cx$ of
  disjoint balls $\{B(x_i, r/4)\}_{i=1}^M$.
  \end{itemize}
\end{remark}

It was proved by Coifman and Weiss in \cite[pp.\,66-68]{cw71}
that metric measure spaces of
homogeneous type are geometrically doubling.

In order to introduce the orthonormal basis of regular wavelets from \cite{ah13},
we first need to recall some notions and notation from \cite{ah13};
see also \cite{fy1}.
For every $k\in\zz$, a set of \emph{reference dyadic points},
$\{x^k_\az\}_{\az\in\sca_k}$, here and hereafter,
\begin{equation}\label{b.v}
\sca_k\ {\rm denotes\ some\ countable\ index\ set\ for\ each}\ k\in\zz,
\end{equation}
is chosen as follows [the Zorn lemma
(see \cite[Theroem I.2]{rs80})
is needed since we consider the maximality]. For
$k=0$, let
$\scx^0:=\{x^0_{\az}\}_{\az\in\sca_0}$ be
a maximal collection of $1$-separated points. Inductively,
for any $k\in\nn$, let
\begin{equation}\label{2.1x}
\scx^k:=\{x^k_\az\}_{\az\in\sca_k}\supset\scx^{k-1}\quad
{\rm and}\quad \scx^{-k}:=\{x^{-k}_\az\}_{\az\in\sca_k}\st\scx^{-(k-1)}
\end{equation}
be maximal $\dz^k$-separated and $\dz^{-k}$-separated collections
in $\cx$ and in $\scx^{-(k-1)}$, respectively.
Fix $\dz$ a small positive parameter, for example, it
suffices to take $\dz\le\frac{1}{1000}$.
From \cite[Lemma 2.1]{ah13}, it follows that
\begin{equation}\label{b.z}
d\lf(x^k_\az,x^k_\bz\r)\ge\dz^k\ {\rm for\ all\ }\az,\,\bz\in\sca_k\
{\rm and}\ \az\neq\bz, \quad
d\lf(x,\scx\r):=\inf_{\az\in\sca_k}d\lf(x,x^k_{\az}\r)<2\dz^k.
\end{equation}

Obviously, the dyadic reference points
$\{x^k_\az\}_{k\in\zz,\,\az\in\sca_k}$ satisfy \cite[(2.3) and (2.4)]{hk12}
with $A_0=1$, $c_0=1$ and $C_0=2$, which further induces a
dyadic system of dyadic cubes over geometrically doubling
metric spaces as in \cite[Theorem 2.2]{hk12}. We restate it
in the following theorem, which is applied to the construction
of the orthonormal basis of regular wavelets
as in \cite{ah13}.

\begin{theorem}\label{tb.c}
Let $(\cx,d)$ be a geometrically doubling metric space.
Then there exist families of sets,
$\mathring{Q}^k_{\az}\st Q^k_{\az}\st\ov{Q}^k_{\az}$
(called, respectively, \emph{open, half-open and
closed dyadic cubes}) such that:

\begin{itemize}
\item[\rm(i)] $\mathring{Q}^k_{\az}$ and $\ov{Q}^k_{\az}$
denote, respectively, the interior and the closure of $Q^k_{\az}$;

\item[\rm(ii)] if $\ell\in\zz\cap[k,\fz)$ and $\az,\,\bz\in\sca_k$,
 then either $Q^\ell_{\bz}\st Q^k_{\az}$ or $Q^k_{\az}\cap Q^\ell_{\bz}=\emptyset$;

\item[\rm(iii)] for any $k\in\zz$,
\begin{equation*}
\cx=\bigcup_{\az\in\sca_k}Q^k_{\az}\quad (disjoint\ union);
\end{equation*}

\item[\rm(iv)] for any $k\in\zz$ and $\az\in\sca_k$ with $\sca_k$ as in \eqref{b.v},
\begin{equation*}
B\lf(x^k_{\az},\frac13\dz^k\r)\st Q^k_{\az}
\st B\lf(x^k_{\az},4\dz^k\r)=:B\lf(Q^k_{\az}\r);
\end{equation*}

\item[\rm(v)] if $k\in\zz,\ \ell\in\zz\cap[k,\fz),\ \az,\,\bz\in\sca_k$ and
$Q^\ell_{\bz}\st Q^k_{\az}$, then $B(Q^\ell_{\bz})\st B(Q^k_{\az}).$

\end{itemize}
The open and closed cubes $\mathring{Q}^k_{\az}$ and $\ov{Q}^k_{\az}$,
with $(k,\az)\in\sca$, here and hereafter,
\begin{equation}\label{b.s}
\sca:=\{(k,\az):\ k\in\zz,\ \az\in\sca_k\},
\end{equation}
depend only on the points $x^\ell_\bz$ for
$\bz\in\sca_{\ell}$ and $\ell\in\zz\cap[k,\fz)$.
The half-open cubes $Q^k_{\az}$,
with $(k,\az)\in\sca$, depend on $x^\ell_\bz$ for $\bz\in\sca_{\ell}$
and $\ell\in\zz\cap[\min\{k,k_0\},\fz)$,
where $k_0\in\zz$ is a preassigned number
entering the construction.
\end{theorem}

\begin{remark}\label{rb.d}
(i) In what follows, let $\le$ be the
\emph{partial order} for dyadic points as in \cite[Lemma 2.10]{hk12}.
It was shown in \cite[Lemma 2.10]{hk12} with $C_0=2$
that, if
$k\in\zz$, $\az\in\sca_k$ with $\sca_k$ as in \eqref{b.v},
$\bz\in\sca_{k+1}$ and
$(k+1,\bz)\le(k,\az)$, then
$d(x_\bz^{k+1},x_\az^k)<2\dz^k.$

(ii) For any $(k,\az)\in\sca$, let
\begin{equation}\label{b.k}
L(k,\az):=\lf\{\bz\in\sca_{k+1}:\ (k+1,\bz)\le(k,\az)\r\}.
\end{equation}
By the proof of \cite[Theorem 2.2]{hk12}
and the geometrically doubling property, we have the following
conclusions: $1\le\#L(k,\az)\le {\wz N}_0$ and
$
Q^k_{\az}=\bigcup_{\bz\in L(k,\,\az)}Q^{k+1}_{\bz}
$,
where ${\wz N}_0\in\nn$ is independent of $k$ and $\az$.
Here and hereafter, for any finite set $\mathfrak{C}$,
$\#\mathfrak{C}$ denotes its \emph{cardinality}.
\end{remark}

The following useful estimate about the $1$-separated set
is from \cite[Lemma 6.4]{ah13}.

\begin{lemma}\label{lb.x}
Let $\Xi$ be a $1$-separated set in a geometrically doubling
metric space $(\cx,d)$ with positive constant $N_0$.
Then, for all $\ez\in(0,\fz)$, there exists a positive constant
$C_{(\ez,\,N_0)}$, depending on $\ez$ and $N_0$, such that
$$
\sup_{a\in\cx}e^{\ez d(a,\,\Xi)/2}\sum_{b\in\Xi}
e^{-\ez d(a,\,b)}\le C_{(\ez,\,N_0)},
$$
here and hereafter, for any set $\Xi\st\cx$ and $x\in\cx$,
$d(x,\Xi):=\inf_{a\in\Xi}d(x,a)$.
\end{lemma}

We now recall more notions and
notation from \cite{ah13}. Let
$(\Omega,\mathscr{F},\mathbb{P}_{\omega})$ be
the \emph{natural probability measure space}
with the same notation as in \cite{ah13}, where $\mathscr{F}$ is
defined as the smallest $\sigma$-algebra containing the set
$$\lf\{\prod_{k\in\zz}A_k:\ A_k\st\Omega_k:={\{0,1,\ldots,L\}
\times\{1,\ldots,M\}}\ {\rm and\ only\ finite\ many}\
A_k\neq \Omega_k\r\},
$$
where $L$ and $M$ are defined as in \cite{ah13}.
For every $(k,\az)\in\sca$ with $\sca$ as in \eqref{b.s},
the \emph{spline function} is defined by setting
$$s^k_\az(x):=\mathbb{P}_\oz\lf(
\lf\{\oz\in\boz:\ x\in\overline{Q}^k_\az(\oz)\r\}\r),
\quad x\in\cx.$$

The following conclusions of splines are taken from \cite[Theorem 3.1]{ah13}.

\begin{proposition}\label{pb.x}
The spline functions have the following properties:
\begin{enumerate}
\item[\rm (i)] for all $(k,\az)\in\sca$ and $x\in\cx$,
\begin{equation}\label{b.x}
\chi_{B(x^k_\az,\,\frac18\dz^k)}(x)
\le s^k_\az(x)\le\chi_{B(x^k_\az,\,8\dz^k)}(x);
\end{equation}

\item[\rm (ii)] for all $k\in\zz$, $\az,\,\bz\in\sca_k$,
with $\sca_k$ as in \eqref{b.v}, and $x\in\cx$,
\begin{equation}\label{b.y}
s^k_\az(x^k_\bz)=\dz_{\az\bz},\quad
\sum_{\az\in\sca_k}s^k_\az(x)=1\quad {\rm and}\quad
s^k_\az(x)=\sum_{\bz\in\sct_{k+1}}p^k_{\az\bz}s^{k+1}_\bz(x),
\end{equation}
where,  for each $k\in\zz$, $\sct_{k+1}\st\sca_{k+1}$
denotes some countable index set
$$
\dz_{\az\bz}:=\left\{
\begin{array}{cc}
1,\ \ &\ \ {\rm if}\ \ \az=\bz,\\
0,\ \ &\ \ {\rm if}\ \ \az\neq\bz,
\end{array}
\right.
$$
and $\{p^k_{\az\bz}\}_{\bz\in\sct_{k+1}}$ is
a finite nonzero set of nonnegative numbers with
$p^k_{\az\bz}\le1$ for all $\bz\in\sct_{k+1}$;

\item[\rm (iii)] there exist positive constants $\eta\in(0,1]$ and $C$,
independent of $k$ and $\az$, such that,
for all $(k,\az)\in\sca$ and $x,\,y\in\cx$,
\begin{equation}\label{b.1}
\lf|s^k_\az(x)-s^k_\az(y)\r|\le C\lf[\frac{d(x,y)}{\dz^k}\r]^\eta.
\end{equation}
\end{enumerate}
\end{proposition}

By \cite[Theorem 5.1]{ah13}, we know that there exists a linear,
bounded uniformly on $k\in\zz$, and injective map $U_k:\ell^2(\sca_k)
\to\ltw$ with closed range, defined by
$$
U_k\lz:=\sum_{\az\in\sca_k}\frac{\lz_{\az}}{\sqrt{\mu^k_\az}}s^k_\az,
\quad \lz:=\lf\{\lz^k_\az\r\}_{\az\in\sca_k}\in\ell^2(\sca_k),
$$
here and hereafter, $\mu^k_\az:=\mu(B(x^k_\az,\dz^k))=:V(x^k_\az,\dz^k)$
for all $(k,\az)\in\sca$, $\ell^2(\sca_k)$ denotes the space of all
sequences $\lz:=\{\lz^k_\az\}_{\az\in\sca_k}\subset\cc$ such that
$$
\|\lz\|_{\ell^2(\sca_k)}:=\lf\{\sum_{\az\in\sca_k}
\lf|\lz^k_\az\r|^2\r\}^{1/2}<\fz.
$$

Observe that, if $k\in\zz$, $\lz,\,\wz{\lz}\in\ell^2(\sca_k)$,
$f=U_k\lz$ and $\wz{f}=U_k\wz{\lz}$, then
$$
\lf(f,\wz{f}\r)_{\ltw}=\lf(M_k\lz,\wz{\lz}\r)_{\ell^2(\sca_k)}
$$
with $M_k$ being the infinite matrix which has entries
$M_k(\az,\bz)=\frac{(s^k_\az,s^k_\bz)_{\ltw}}
{\sqrt{\mu^k_\az\mu^k_\bz}}$ for $\az,\,\bz\in\sca_k$. Let $U^*_k$
be the \emph{adjoint operator} of $U_k$ for all $k\in\zz$.
Thus, for each $k\in\zz$, $M_k=U_k^*U_k$ is bounded,
invertible, positive and self-adjoint on $\ell^2(\sca_k)$.
Let $V_k:=U_k(\ell^2(\sca_k))$ for all $k\in\zz$.
The following result from \cite{ah13} (with $\mu^k_\az$ replaced
by $\nu^k_\az:=\int_\cx s^k_\az\,d\mu$)
shows that $\{V_k\}_{k\in\zz}$ is a \emph{multiresolution analysis}
(for short, MRA) of $\ltw$.

\begin{theorem}\label{tb.n}
Suppose that $(\cx,d,\mu)$ is a metric measure space of homogeneous type.
Let $k\in\zz$ and $V_k$ be the closed linear span of
$\{s^k_\az\}_{\az\in\sca_k}$. Then
$V_k\st V_{k+1}$ and
\begin{equation}\label{b.p}
\overline{\bigcup_{k\in\zz}V_k}=\ltw,\quad
\bigcap_{k\in\zz}V_k=\{0\}.
\end{equation}
Moreover, the functions $\{s^k_\az/\sqrt{\nu^k_\az}\}_{\az\in\sca_k}$
form a Riesz basis of $V_k$: for all sequences of complex numbers
$\{\lz_\az^k\}_{\az\in\sca_k}$,
$$
\lf\|\sum_{\az\in\sca_k}\lz^k_\az s^k_\az\r\|_{\ltw}
\sim\lf[\sum_{\az\in\sca_k}\lf|\lz_\az^k\r|^2\nu^k_\az\r]^{1/2}
$$
with equivalent positive constants independent of $k$ and
$\{\lz_\az^k\}_{\az\in\sca_k}$.
\end{theorem}

Now we are ready to introduce the following
notable orthonormal basis of regular wavelets constructed by
Auscher and Hyt\"onen \cite{ah13}. In what follows,
for all $k\in\zz$, let
\begin{equation}\label{b.w}
\scg_k:=\sca_{k+1}\bh\sca_k
\end{equation}
with $\sca_k$ as in \eqref{b.v}. The following theorem is just \cite[Theorem 7.1]{ah13}.

\begin{theorem}\label{tb.a}
Let $(\cx,d,\mu)$ be a metric measure space of homogeneous type.
Then there exists an
orthonormal basis $\{\psi_{\bz}^k\}_{k\in\zz,\,
\bz\in\scg_k}$
of $\ltw$ and positive constants $\eta\in(0,1]$
as in \eqref{b.1}, $\nu$ and $C_{(\eta)}$,
independent of $k$, $\az$ and $\bz$, such that
\begin{equation}\label{b.a}
\lf|\psi_{\bz}^k(x)\r|\le\frac{C_{(\eta)}}
{\sqrt{V(x_{\bz}^{k+1},\dz^k)}}
e^{-\nu\dz^{-k}d(x_{\bz}^{k+1},\,x)}\ \ {\rm for\ all}\ x\in\cx,
\end{equation}
\begin{equation}\label{b.b}
\lf|\psi_{\bz}^k(x)-\psi_{\bz}^k(y)\r|
\le\frac{C_{(\eta)}}{\sqrt{V(x_{\bz}^{k+1},\dz^k)}}
\lf[\frac{d(x,y)}{\dz^k}\r]^{\eta}
e^{-\nu\dz^{-k}d(x_{\bz}^{k+1},\,x)}
\end{equation}
for all $x,\,y\in\cx$ satisfying $d(x,y)\le\dz^k$, and
\begin{equation}\label{b.c}
\int_{\cx}\psi_{\bz}^k(x)\,d\mu(x)=0.
\end{equation}
\end{theorem}

In order to introduce the lower bounds of wavelets,
we need to recall a slight difference on the notation
of the orthonormal basis $\{\psi_{\bz}^k\}_{k\in\zz,\,\bz\in\scg_k}$
from \cite{fy1} that
$$
\lf\{\psi_{\az,\,\bz}^k\r\}_{(k,\,\az)\in\wz{\sca},\,
\bz\in\wz{L}(k,\,\az)}:=\lf\{\psi_\bz^k\r\}_{k\in\zz,\,
\bz\in\scg_k},
$$
where,
\begin{equation}\label{2.19x1}
\wz{\sca}:=\{(k,\az)\in\sca:\ \#L(k,\az)>1\}
\end{equation}
with $\sca$ and $L(k,\az)$, respectively, as
in \eqref{b.s} and \eqref{b.k},
and, for all $(k,\az)\in\wz{\sca}$,
\begin{equation}\label{2.19x2}
\wz{L}(k,\az):=L(k,\az)\bh\{\az\},
\end{equation}
via the fact that, for any $k\in\zz$,
$$
\sca_{k+1}\bh\sca_k=\bigcup_{\{\az\in\sca_k:\ \#L(k,\,\az)>1\}}
\wz{L}(k,\az).
$$

Now we introduce the lower bounds of $\psi^k_{\az,\,\bz}$ from \cite{fy1}
which is important to the succeeding context.

\begin{theorem}\label{tb.m}
Let $(\cx,d,\mu)$ be a metric measure space of homogeneous type.
Then there exist positive constants $\ez_0$ and $C$,
independent of $k$, $\az$ and $\bz$, such
that, for all $(k,\az)\in\wz{\sca}$ with $\wz{\sca}$ as in
\eqref{2.19x1},
$\bz\in\wz{L}(k,\az)$ with $\wz{L}(k,\az)$ as in \eqref{2.19x2},
and $x\in B(y^k_\bz,\ez_0\dz^k)\st Q^k_\az$,
$$
\lf|\psi^k_{\az,\,\bz}(x)\r|\ge C\frac1{\sqrt{\mu(Q^k_\az)}}.
$$
\end{theorem}

We also recall the molecular characterization of $\hona$ from
\cite{hyz}, which plays important roles in the succeeding content, since it
partially compensates the defect of the
regular wavelets without bounded supports.

The following notions of $(1,q,\eta)$-molecules are from \cite{hyz}.

\begin{definition}\label{dc.z}
Let $q\in(1, \fz]$ and $\{\eta_k\}_{k\in\nn}\st[0,\fz)$ satisfy
\begin{equation}\label{c.w}
\sum_{k\in\nn}k\eta_k<\fz.
\end{equation}

A function $m\in\lq$ is called a \emph{$(1,q,\eta)$-molecule centered
at a ball $B:=B(x_0,r)$}, for some $x_0\in\cx$ and $r\in(0,\fz)$, if

(M1) $\|m\|_{\lq}\le[\mu(B)]^{1/q-1}$;

(M2) for all $k\in\nn$,
$$
\lf\|m\chi_{B(x_0,2^kr)\bh B(x_0,2^{k-1}r)}\r\|_{\lq}
\le\eta_k2^{k(1/q-1)}[\mu(B)]^{1/q-1};
$$

(M3) $\int_\cx m(x)\,d\mu(x)=0$.
\end{definition}

Then the following molecular characterization of the space $\hona$
is a slight variant of \cite[Theorem 2.2]{hyz} which is originally related
to the quasi-metric $\rho$ as in \eqref{a.x}
and is obviously true with $\rho$ replaced by $d$.

\begin{theorem}\label{tc.y}
Suppose that $(\cx,d,\mu)$ is a metric measure space of homogeneous type.
Let $q\in(1, \fz]$ and $\eta=\{\eta_k\}_{k\in\nn}\st[0,\fz)$
satisfy \eqref{c.w}.
Then there exists a positive constant
$C$ such that, for any $(1,q,\eta)$-molecule $m$, $m\in\hona$
and
$$\|m\|_{\hona}\le C.$$
Moreover, $f\in\hona$ if and only if there exist
$(1,q,\eta)$-molecules $\{m_j\}_{j\in\nn}$ and numbers
$\{\lz_j\}_{j\in\nn}\st\cc$ such that
$$
f=\sum_{j\in\nn}\lz_jm_j,
$$
which converges in $\lon$. Furthermore,
$$
\|f\|_{\hona}\sim\inf\lf\{\sum_{j\in\nn}|\lz_j|\r\},
$$
where the infimum is taken over all the decompositions of $f$
as above and the equivalent positive constants
are independent of $f$.
\end{theorem}

Now we introduce the notion of Calder\'on-Zygmund
operators from \cite{cw71}; see also \cite{ah13}.

\begin{definition}\label{d4.6}
A function $K\in L_\loc^1(\{\cx\times
\cx\}\bh\{(x,x):x\in\cx\})$ is called a \emph{Calder\'on-Zygmund
kernel} if there exists a positive constant $C_{(K)}$,
depending on $K$, such that
\begin{enumerate}
\item[(i)] for all $x,\,y\in\cx$ with $x\ne y$,
\begin{equation}\label{d.a}
|K(x,y)|\le C_{(K)}\frac{1}{V(x,y)};
\end{equation}

\item[(ii)] there exist positive constants
$s\in (0,1]$ and $c_{(K)}\in(0,1)$, depending on $K$,
such that
\begin{enumerate}
\item[${\rm (ii)}_1$] for all $x,\,\wz x,\,y\in\cx$
with $d(x,y)\ge c_{(K)}d(x,\wz{x})>0$,
\begin{equation}\label{d.b}
|K(x,y)-K(\wz{x},y)|\le C_{(K)}
\lf[\frac{d(x,\wz{x})}{d(x,y)}\r]^s
\frac1{V(x,y)};
\end{equation}

\item[${\rm (ii)}_2$] for all $x,\,\wz x,\,y\in\cx$
with $d(x,y)\ge c_{(K)}d(y,\wz{y})>0$,
\begin{equation}\label{d.g}
\lf|K(x,y)-K(x,\wz{y})\r|\le C_{(K)}
\lf[\frac{d(y,\wz{y})}{d(x,y)}\r]^s
\frac1{V(x,y)}.
\end{equation}
\end{enumerate}
\end{enumerate}

Let $T:\ C^s_b(\cx)\to (C^s_b(\cx))^*$ be a linear continuous
operator. The operator $T$ is
called  a \emph{Calder\'on-Zygmund operator}
with kernel $K$ satisfying \eqref{d.a}, \eqref{d.b}
and \eqref{d.g} if, for
all $f\in C^s_b(\cx)$,
\begin{equation}\label{d.c}
Tf(x):=\int_{\cx}K(x,y)f(y)\,d\mu(y),\quad x\not\in\supp(f).
\end{equation}
\end{definition}

Then we recall some results from
\cite[Proposition 3.1]{yz08}
(see also \cite[Theorem 4.2]{hyz}) about the
boundedness of Calder\'on-Zygmund operators. In what follows
$T^*1=0$ means that, for all $(1,2)$-atom $a$,
$\int_\cx Ta(x)\,d\mu(x)=0$. By some careful
examinations, we see that this result remains valid over
the metric measure space of homogeneous type
without resorting to the reverse doubling condition,
the details being omitted.

\begin{theorem}\label{te.g}
Let $(\cx,d,\mu)$ be a metric measure space of homogeneous type.
Suppose that $T$ is a Calder\'on-Zygmund operator as in
\eqref{d.c} which is bounded on $\ltw$.
\begin{enumerate}
\item[\rm (i)] Then there exists a positive constant
$C$, depending only on $\|T\|_{\cl(\ltw)}$, $s$, $C_{(K)}$,
$c_{(K)}$ and $\wz{C}_{(\cx)}$,
 such that, for all $f\in \hona$, $Tf\in \lon$ and
$\|Tf\|_{\lon}\le C\|f\|_{\hona}$.

\item[\rm (ii)] If further assuming that $T^*1=0$, then
there exists a positive constant $\wz{C}$,
depending only on $\|T\|_{\cl(\ltw)}$, $s$, $C_{(K)}$,
$c_{(K)}$ and $\wz{C}_{(\cx)}$,
such that, for all $f\in \hona$, $Tf\in \hona$ and
$\|Tf\|_{\hona}\le C\|f\|_{\hona}$.
\end{enumerate}
\end{theorem}

Now we recall several equivalent characterizations
for $\hona$ via wavelets from \cite{fy1}.
To this end, we need more notation.
We point out that, for any $(k,\az,\bz)\in\sci$, where
\begin{equation}\label{b.2}
\mathscr{I}:=\lf\{(k,\az,\bz):\ (k,\az)\in\wz{\sca},\,
\bz\in\wz{L}(k,\az)\r\}
\end{equation}
with $\wz{\sca}$ and $\wz{L}(k,\az)$ respectively
as in \eqref{2.19x1} and \eqref{2.19x2},
we have $\psi^k_{\az,\,\bz}\in\li$ and hence
$\langle f,\psi^k_{\az,\,\bz}\rangle$ is well defined for
any $f\in\lon$ in the sense of duality between $\lon$ and $\li$.

\begin{theorem}\label{tc.d}
Let $(\cx,d,\mu)$ be a metric measure space of homogeneous type.
Suppose that $f\in\lon$ and
$$
f=\sum_{(k,\,\az,\,\bz)\in\sci}
\langle f,\psi^k_{\az,\,\bz}\rangle\psi^k_{\az,\,\bz}
\quad {\rm in}\quad \lon.
$$
Then the following statements are mutually equivalent:
\begin{enumerate}
\item[{\rm (i)}] $f\in \hona$;

\item[{\rm (ii)}] $\sum_{(k,\,\az,\,\bz)\in\sci}
\langle f,\psi^k_{\az,\,\bz}\rangle\psi^k_{\az,\,\bz}$
converges unconditionally in $L^1(\cx)$;

\item[{\rm (iii)}] $\|f\|_{\rm (iii)}
:=\|\{\sum_{(k,\,\az,\,\bz)\in\sci}
|\langle f,\psi^k_{\az,\,\bz}\rangle|^2
|\psi^k_{\az,\,\bz}|^2\}^{1/2}\|_{L^1(\cx)}<\fz;
$

\item[{\rm (iv)}] $\|f\|_{\rm (iv)}:=
\|\{\sum_{(k,\,\az,\,\bz)\in\sci}
|\langle f,\psi^k_{\az,\,\bz}\rangle|^2
\frac{\chi_{Q_\az^k}}{\mu(Q_\az^k)}\}^{1/2}\|_{L^1(\cx)}<\fz;
$

\item[{\rm (v)}] $\|f\|_{\rm (v)}:=
\|\{\sum_{(k,\,\az,\,\bz)\in\sci}
|\langle f,\psi^k_{\az,\,\bz}\rangle|^2
[R^k_{\az,\,\bz}]^2\}^{1/2}\|_{L^1(\cx)}<\fz$,
\end{enumerate}
here and hereafter,
\begin{equation}\label{c.d}
R^k_{\az,\,\bz}:=\frac{\chi_{W^k_{\az,\,\bz}}}
{\sqrt{\mu(Q^k_\az)}},
\end{equation}
and
$
W^k_{\az,\,\bz}:=B(y^k_{\bz},\ez_0\dz^k)\st Q^k_\az
$
as in Theorem \ref{tb.m}.

Moreover, $\|\cdot\|_{\rm (iii)}$, $\|\cdot\|_{\rm (iv)}$
and $\|\cdot\|_{\rm (v)}$ give norms on $\hona$, which are
equivalent to $\|\cdot\|_{\hona}$, respectively.
\end{theorem}

The following lemma is taken from \cite[Lemma 3.18]{fy1}.

\begin{lemma}\label{lc.i}
Let $(\cx,d,\mu)$ be a metric measure space of homogeneous type.
For any family of numbers,
$\{a(j,\az,\bz)\}_{(j,\,\az,\,\bz)\in\sci}\st\cc$
with $\sci$ as in \eqref{b.2}, let $\cs$ be any finite
subset of $\sci$ and
$$
\vz_{\cs}(x):=\lf\{\sum_{(j,\,\az,\,\bz)\in\cs}
|a(j,\az,\bz)|^2
\lf[R^j_{\az,\,\bz}(x)\r]^2\r\}^{1/2},\quad x\in\cx,
$$
where $R^j_{\az,\,\bz}$ is as in \eqref{c.d}.
Suppose that $\vz_{\cs}\in\lon$.
Then the function
$$
\sum_{(j,\,\az,\,\bz)\in\cs}a(j,\az,\bz)
\psi^j_{\az,\,\bz}\in \hona
$$
and there exists a positive constant $C$, independent of $\cs$,
such that
$$
 \lf\|\sum_{(j,\,\az,\,\bz)\in\cs}a(j,\az,\bz)
\psi^j_{\az,\,\bz}\r\|_{\hona}\le C\|\vz_{\cs}\|_{\lon}.
$$
\end{lemma}

\section{$\ltw$ Estimates for Products of Two Functions}\label{s4}

\hskip\parindent In this section, we obtain $\ltw$ estimates for
the product of two functions in $\ltw$.

If two functions in $\ltw$ both have finite wavelet decompositions,
we have the following conclusion.

\begin{lemma}\label{le.a}
Suppose that $(\cx,d,\mu)$ is a metric measure space of homogeneous type.
Let $f,\,g\in\ltw$, $\{V_k\}_{k\in\zz}$ be an MRA of $\ltw$
as in Theorem \ref{tb.n}, $W_k$ be the orthogonal [in $\ltw$]
complement of $V_k$ in $V_{k+1}$ and $P_k$
and $Q_k$ be the projection operators from $\ltw$ onto, respectively, $V_k$
and $W_k$. Suppose that $f$ and $g$ both have finite
wavelet decompositions, namely,
there exist $M_1,\,M_2\in\nn$ such that
\begin{equation}\label{4.1x}
f=\sum_{k=-M_1}^{M_1}\sum_{\bz\in\scg_k}\lf(f,\psi^k_\bz\r)\psi^k_\bz \quad
{\rm and}\quad
g=\sum_{k=-M_2}^{M_2}\sum_{\bz\in\scg_k}\lf(g,\psi^k_\bz\r)\psi^k_\bz,
\end{equation}
where $\scg_k$ for $k\in\zz$ is as in \eqref{b.w}.
Then
\begin{eqnarray}\label{e.z}
fg&&=\sum_{k\in\zz}(P_kf)(Q_kg)+\sum_{k\in\zz}(Q_kf)(P_kg)
+\sum_{k\in\zz}(Q_kf)(Q_kg)\\
&&\noz=:\Pi_1(f,g)+\Pi_2(f,g)+\Pi_3(f,g)
\end{eqnarray}
in $\ltw$.
\end{lemma}

\begin{proof}
We first claim that, for any $\ell\in\zz$,
\begin{equation}\label{e.d}
f=\sum_{k\in\zz}Q_kf=P_{\ell}f+\sum_{k=\ell}^\fz Q_kf
\end{equation}
holds true in $\ltw$.
Indeed, let $f\in\ltw$. By \eqref{b.p}, we see that,
for any given $\ez\in(0,\fz)$, there exist $k_0\in\zz$ and $f_0\in V_{k_0}$
such that $\|f-f_0\|_{\ltw}<\ez$.
Moreover, by the fact that $P_{k_0}$ is the projection operator
from $\ltw$ onto $V_{k_0}$, we know that
$$
\lf\|f-P_{k_0}f\r\|_{\ltw}\le\lf\|f-f_0\r\|_{\ltw}<\ez,
$$
which, combined with the fact that, for all $k>k_0$,
$P_{k_0}f\in V_{k_0}\st V_k$, implies that
$$
\lf\|f-P_{k}f\r\|_{\ltw}\le\lf\|f-P_{k_0}f\r\|_{\ltw}<\ez.
$$
Thus,
\begin{equation}\label{e.b}
f=\lim_{k\to\fz}P_kf \quad {\rm in}\quad \ltw.
\end{equation}

On the other hand, we show that
\begin{equation}\label{e.a}
\lim_{k\to\fz}P_{-k}f=0 \quad {\rm in}\quad \ltw.
\end{equation}
For any $k\in\zz$, from $P_kf=P_{k-1}f+Q_{k-1}f$ and
$P_{k-1}f\bot Q_{k-1}f$, it follows that
$$
\lf\|P_{k}f\r\|^2_{\ltw}=\lf\|P_{k-1}f\r\|^2_{\ltw}
+\lf\|Q_{k-1}f\r\|^2_{\ltw}
\ge\lf\|P_{k-1}f\r\|^2_{\ltw}.
$$
Therefore, $\{\|P_{-k}f\|_{\ltw}\}_{k\in\nn}$ is
decreasing and bounded below.
Thus,
$$
\lim_{k\to\fz}\|P_{-k}f\|_{\ltw}
$$
exists. Meanwhile, by the Banach-Alaoglu theorem (see,
for example, \cite[Theorem IV.\,21]{rs80}),
we conclude that the unit ball in $\ltw$ is compact,
and hence there exists $f_1\in\ltw$ and a subsequence
$\{P_{-n_k}f\}_{k\in\nn}$ of $\{P_{-k}f\}_{k\in\nn}$ such that
$$
\lim_{k\to\fz}P_{-n_k}f=f_1 \quad {\rm in}\quad \ltw
$$
and $\|f_1\|_{\ltw}=\lim_{k\to\fz}\|P_{-n_k}f\|_{\ltw}$.
Moreover, since, for any $m\in\nn$ and $k\ge m$,
$P_{-k}f\in V_{-m}$ and $V_{-m}$ is closed, we have
$f_1\in V_{-m}$. This, together with \eqref{b.p},
implies that $f_1\in\bigcap_{m\in\zz}V_m=\{0\}$,
and hence
$$
\lim_{k\to\fz}P_{-n_k}f=0 \quad {\rm in}\quad \ltw.
$$
From this and the fact that $\lim_{k\to\fz}\|P_{-k}f\|_{\ltw}$
exists, we deduce that
$$
\lim_{k\to\fz}P_{-k}f=0 \quad {\rm in}\quad \ltw,
$$
which shows \eqref{e.a}.

Furthermore, from \eqref{e.b}, $P_{k+1}=P_k+Q_k$
and \eqref{e.a}, we deduce that
\begin{eqnarray*}
f&&=\lim_{k\to\fz}P_kf=\lim_{m\to\fz}\lf[\sum_{k=-m}^m
\lf(P_{k+1}f-P_kf\r)+P_{-m}f\r]\\
&&=\lim_{m\to\fz}\lf[\sum_{k=-m}^m Q_kf+P_{-m}f\r]
=\lim_{m\to\fz}\lf[\sum_{k=-m}^{\ell-1} Q_kf+P_{-m}f
+\sum_{k=\ell}^m Q_kf\r]\\
&&=\lim_{m\to\fz}\lf[P_{\ell}f+\sum_{k=\ell}^m Q_kf\r]
=P_{\ell}f+\sum_{k=\ell}^\fz Q_kf
\end{eqnarray*}
holds true in $\ltw$,
which completes the proof of the above claim \eqref{e.d}.

By the finite wavelet decomposition of $f$,
we have $Q_kf=0$ for all $k\not\in\{-M_1,\ldots,M_1\}$,
which, together with $P_k=P_{k-1}+Q_{k-1}$ and \eqref{e.a}, implies that
\begin{eqnarray}\label{4.5x1}
f&&=\sum_{k=-M_1}^{M_1}\sum_{\bz\in\scg_k}\lf(f,\psi^k_\bz\r)\psi^k_\bz
=\sum_{k=-M_1}^{M_1}Q_kf=P_{M_1+1}f-P_{-M_1}f\\
&&\noz=P_{M_1+1}f-\lim_{k\to\fz}P_{-k}f
=P_{M_1+1}f \quad {\rm in}\quad \ltw.
\end{eqnarray}
Similarly, $g=P_{M_2+1}g$ in $\ltw$. Let $M_0:=\max\{M_1+1,M_2+1\}$.
Then, by the finite wavelet decompositions of $f$ and $g$ again,
we see that, for all $k\in\nn$, $P_{M_0+k}f=P_{M_0}f$
and $P_{M_0+k}g=P_{M_0}g$ and, for all $k\in\nn\bh\{1,\ldots,M_0-1\}$,
$P_{-k}f=0=P_{-k}g$. These facts, combined with \eqref{e.b}, \eqref{e.a}
and $P_{k+1}=P_k+Q_k$ for all $k\in\zz$, imply that
\begin{eqnarray}\label{4.5x2}
fg&&=\lim_{M\to\fz}P_M fP_M g
=\lim_{M\to\fz}\lf[P_M fP_M g-P_{-M}fP_{-M}g\r]\\
&&\noz=\lim_{M\to\fz}\sum_{k=-M}^{M-1}
\lf(P_{k+1}fP_{k+1}g-P_k fP_k g\r)\\
&&\noz=\sum_{k\in\zz}\lf(P_{k+1}fP_{k+1}g-P_k fP_k g\r)\\
&&\noz=\sum_{k\in\zz}\lf[\lf(P_kf+Q_kf\r)\lf(P_kg+Q_kg\r)-P_k fP_k g\r]\\
&&\noz=\sum_{k\in\zz}(P_kf)(Q_kg)+\sum_{k\in\zz}(Q_kf)(P_kg)
+\sum_{k\in\zz}(Q_kf)(Q_kg),
\end{eqnarray}
which completes the proof of Lemma \ref{le.a}.
\end{proof}

\begin{remark}\label{r4.1x}
(i) We point out that $f$ and $g$ as in \eqref{4.1x} have finite
wavelet decompositions only in some sense, since $\scg_k$ for any
$k\in\zz$ may have an infinite cardinality.

(ii) The finite wavelet decompositions of $f$ and $g$ play
a key role in the proofs of \eqref{4.5x1} and \eqref{4.5x2}
appeared in the proof of Lemma \ref{le.a}.
\end{remark}

The proof of the boundedness of $\Pi_3$ from $\ltw\times\ltw$
into $\lon$ is parallel to that of \cite[Lemma 4.1]{bgk},
the details being omitted.

\begin{lemma}\label{le.b}
Let $(\cx,d,\mu)$ be a metric measure space of homogeneous type. Then
the bilinear operator $\Pi_3$ in \eqref{e.z} is bounded from $\ltw\times\ltw$
into $\lon$.
\end{lemma}

Before we deal with $\Pi_1$ in \eqref{e.z}, we need to establish some
important estimates of some integral operators defined as follows.
Let
\begin{equation}\label{4.4.x1}
\scc:=\{(k,\bz):\ k\in\zz,\ \bz\in\scg_k\},
\end{equation}
where $\scg_k$ is as in \eqref{b.w}.
Choosing a fixed collection
\begin{equation}\label{4.5x}
\lf\{\scc_N:\ N\in\nn,\ \scc_N\st\scc\ {\rm and\ \scc_N\ is\ finite}\r\}
\end{equation}
satisfies $\scc_N\uparrow\scc$, namely, for any $N\in\nn$,
$\scc_N\st\scc_{N+1}$ and $\scc=\bigcup_{N\in\nn}\scc_N$.
Fixed $k\in\zz$ and $N\in\nn$, for any $j\in\zz$,
$(j,\bz)\in\scc$ and
\begin{equation}\label{4.5y}
\sca^k_{j,\,\bz}:=\lf\{\az\in\sca_j:\
2^k\dz^{j+1}\le d(x^j_{\az},y^j_\bz)<2^{k+1}\dz^{j+1}\r\},
\end{equation}
where $x^j_\az$ is as in \eqref{2.1x} with $k$ replaced by $j$,
and $y^j_\bz:=x^{j+1}_\bz$ for $\bz\in\scg_j$.
By the geometrically doubling condition and
Remark \ref{rb.l}(ii),
we see that, for all $j,\,k\in\zz$ and $\bz\in\scg_j$,
\begin{equation}\label{4.5z}
M^k_{j,\,\bz}:=\#\sca^k_{j,\,\bz}\le N_02^{(k+1)G_0}=:m_k,
\end{equation}
where $G_0$ and $N_0$ are the same as in Remark \ref{rb.l}(ii).

Now we relabel the set $\sca^k_{j,\,\bz}$ as
$\sca^k_{j,\,\bz}=:\{\az^i_{j,\,\bz}\}_{i=1}^{M^k_{j,\,\bz}}$.
If $M^k_{j,\,\bz}<m_k$, then we further enlarge $\sca^k_{j,\,\bz}$ to
$\{\az^i_{j,\,\bz}\}_{i=1}^{m_k}$ with
$s^j_{\az^i_{j,\,\bz}}:=0$ for any $i\in\nn\cap(M^k_{j,\,\bz},m_k]$.
If $M^k_{j,\,\bz}=m_k$, the set $\sca^k_{j,\,\bz}$ remains unchanged.
Let $\az:=\az^i_{j,\,\bz}\in\sca^k_{j,\,\bz}$,
\begin{equation}\label{4.4.x}
\wz{\psi}^{k,\,i}_{j,\,\bz}:=e^{\nu\dz2^{k-2}}s^j_{\az}\psi^j_\bz
\quad {\rm and}\quad
U^N_{k,\,i}g:=\sum_{(j,\,\bz)\in\scc_N}\lf(g,\psi^j_\bz\r)
\wz{\psi}^{k,\,i}_{j,\,\bz} \quad {\rm for\ all\ }g\in\ltw,
\end{equation}
where $\psi^j_\bz$ is as in Theorem \ref{tb.a} with $k$ and $\az$
replaced, respectively, by $j$ and $\bz$. We also define
\begin{equation}\label{3.5x}
\scy^k:=\scx^{k+1}\bh\scx^k
\end{equation}
with $\scx^k$ as in \eqref{2.1x}.
Then it is obvious that $U^N_{k,\,i}g\in\ltw$
for all $g\in\ltw$, noticing $\scc_N$ is finite. Moreover,
we have the following result.

\begin{proposition}\label{pe.e}
Suppose that $(\cx,d,\mu)$ is a metric measure space of homogeneous type.
Let $U^N_{k,\,i}$ be defined as in \eqref{4.4.x} for
$N\in\nn$, $k\in\zz$ and $i\in\{1,\ldots, m_k\}$
with $m_k$ as in \eqref{4.5z}. Then there exists a
positive constant $C$, independent of $N$, $k$ and $i$, such that,
for all $g,\,h\in\ltw$,
\begin{eqnarray}\label{e.h}
\lf|\lf(U^N_{k,\,i}g,h\r)\r|
\le C\lf[\sum_{(j,\,\bz)\in\scc_N}
\lf|\lf(g,\psi^j_\bz\r)  \r|^2\r]^{1/2}
\|h\|_{\ltw}\le C\|g\|_{\ltw}\|h\|_{\ltw}.
\end{eqnarray}
\end{proposition}

\begin{proof}
Let $k\in\zz$, $i\in\{1,\ldots,m_k\}$,
with $m_k$ as in \eqref{4.5z}, and
$U^N_{k,\,i}$ be as in \eqref{4.4.x}.
By the H\"older inequality and Theorem \ref{tb.a},
we see that, for all $g,\,h\in\ltw$,
\begin{eqnarray*}
\lf|\lf(U^N_{k,\,i}g,h\r)\r|&&
\le\sum_{(j,\,\bz)\in\scc_N}\lf|\lf(g,\psi^j_\bz\r)  \r|
\lf|\lf(\wz{\psi}^{k,\,i}_{j,\,\bz},h\r)\r|\\
&&\le\lf[\sum_{(j,\,\bz)\in\scc_N}
\lf|\lf(g,\psi^j_\bz\r)  \r|^2\r]^{1/2}
\lf[\sum_{(j,\,\bz)\in\scc_N}
\lf|\lf(\wz{\psi}^{k,\,i}_{j,\,\bz},h\r)\r|^2\r]^{1/2}\\
&&\le\|g\|_{\ltw}\lf[\sum_{j\in\zz}\sum_{\bz\in\scg_j}
\lf|\lf(\wz{\psi}^{k,\,i}_{j,\,\bz},h\r)\r|^2\r]^{1/2}.
\end{eqnarray*}
Thus, to show \eqref{e.h}, it suffices to prove that
\begin{equation}\label{e.i}
{\rm I}:=\lf[\sum_{j\in\zz}\sum_{\bz\in\scg_j}
\lf|\lf(\wz{\psi}^{k,\,i}_{j,\,\bz},h\r)\r|^2\r]^{1/2}\ls\|h\|_{\ltw}.
\end{equation}

To this end, since $h\in\ltw$ and $\{\psi^s_\gz\}_{s\in\zz,\,\gz\in\scg_k}$
is an orthonormal basis of $\ltw$ (see Theorem \ref{tb.a}), we write
\begin{eqnarray*}
{\rm I}&&\le\lf\{\sum_{j\in\zz}\sum_{\bz\in\scg_j}
\lf[\sum_{s\in\zz}\sum_{\gz\in\scg_s}
\lf|\lf(h,\psi^s_\gz\r)\r|
\lf|\lf(\psi^s_\gz,\wz{\psi}^{k,\,i}_{j,\,\bz}\r)\r|\r]^2\r\}^{1/2}\\
&&\le\lf\{\sum_{j\in\zz}\sum_{\bz\in\scg_j}
\lf|\lf(h,\psi^j_\bz\r)\r|^2
\lf|\lf(\psi^j_\bz,\wz{\psi}^{k,\,i}_{j,\,\bz}\r)\r|^2\r\}^{1/2}\\
&&\hs+\lf\{\sum_{j\in\zz}\sum_{\bz\in\scg_j}
\lf[\sum_{\{\gz\in\scg_j:\ \gz\neq\bz\}}
\lf|\lf(h,\psi^j_\gz\r)\r|
\lf|\lf(\psi^j_\gz,\wz{\psi}^{k,\,i}_{j,\,\bz}\r)\r|\r]^2\r\}^{1/2}\\
&&\hs+\lf\{\sum_{j\in\zz}\sum_{\bz\in\scg_j}
\lf[\sum_{s=-\fz}^{j-1}\sum_{\gz\in\scg_s}
\lf|\lf(h,\psi^s_\gz\r)\r|
\lf|\lf(\psi^s_\gz,\wz{\psi}^{k,\,i}_{j,\,\bz}\r)\r|\r]^2\r\}^{1/2}\\
&&\hs+\lf\{\sum_{j\in\zz}\sum_{\bz\in\scg_j}
\lf[\sum_{s=j+1}^\fz\sum_{\gz\in\scg_s}
\cdots\r]^2\r\}^{1/2}
=:\sum_{t=1}^4{\rm I}_t.
\end{eqnarray*}

To deal with ${\rm I}_1$, we first estimate
$|(\psi^j_\bz, \wz{\psi}^{k,\,i}_{j,\,\bz})|$ for any
$(j,\bz)\in\scc$, with $\scc$ as in \eqref{4.4.x1}
and $\az\in\sca_j$ with $\sca_j$ as in \eqref{b.v}.

Observe that, by \eqref{a.b},
we see that, for any $r_0,\,\nu_0\in(0,\fz)$
and $x_0\in\cx$,
\begin{eqnarray}\label{3.9x1}
&&\int_{\cx}e^{-\nu_0 d(x,\,x_0)/r_0}\,d\mu(x)\\
&&\noz\hs\ls\int_{B(x_0,\,r_0)}e^{-\nu_0 d(x,\,x_0)/r_0}\,d\mu(x)
+\sum_{\ell=1}^\fz\int_{B(x_0,\,(\ell+1)r_0)
\bh B(x_0,\,\ell r_0)}\cdots\\
&&\noz\hs\ls V(x_0,r_0)+\sum_{\ell=1}^\fz e^{-\nu_0\ell}
V\lf(x_0,[\ell+1]r_0\r)\\
&&\noz\hs\ls V(x_0,r_0)+\sum_{\ell=1}^\fz e^{-\nu_0\ell}
(\ell+1)^nV\lf(x_0,r_0\r)\ls V(x_0,r_0).
\end{eqnarray}

From \eqref{4.4.x}, \eqref{b.x}, \eqref{b.a},
$d(x^j_\az,y_\bz^j)\ge2^k\dz^{j+1}$ and \eqref{3.9x1},
we deduce that
\begin{eqnarray*}
\lf|\lf(\psi^j_\bz, \wz{\psi}^{k,\,i}_{j,\,\bz}\r)\r|
&&\le\int_\cx \lf|\psi^j_\bz(x)\wz{\psi}^{k,\,i}_{j,\,\bz}(x)\r|\,d\mu(x)
\ls e^{\nu\dz2^{k-2}}\int_{B(x^j_\az,\,8\dz^j)}\lf|\psi^j_\bz(x)\r|^2
\,d\mu(x)\\
&&\ls e^{\nu\dz2^{k-2}}\frac1{V(y^j_\bz,\dz^j)}
\int_{B(x^j_\az,\,8\dz^j)}e^{-2\nu\dz^{-j}d(y_\bz^j,\,x)}\,d\mu(x)\\
&&\ls e^{\nu\dz2^{k-2}}e^{-\nu\dz^{-j}d(y_\bz^j,\,x^j_\az)}
\frac1{V(y^j_\bz,\dz^j)}
\int_{B(x^j_\az,\,8\dz^j)}e^{-\nu\dz^{-j}d(y_\bz^j,\,x)}\,d\mu(x)\\
&&\ls e^{\nu\dz2^{k-2}}e^{-\nu\dz2^{k}}\ls1.
\end{eqnarray*}
By this and Theorem \ref{tb.a}, we obtain
$$
{\rm I}_1\ls\lf\{\sum_{j\in\zz}\sum_{\bz\in\scg_j}
\lf|\lf(h,\psi^j_\bz\r)\r|^2\r\}^{1/2}
\sim\|h\|_{\ltw}.
$$

Now we turn to estimate ${\rm I}_2$. Let $G_0$ be the same
as in Remark \ref{rb.l}(ii).
Observe that, by \eqref{b.w} and \eqref{b.z}, we know that,
for given $j\in\zz$ and $\bz\in\scg_j$, $\gz\in\scg_j$
and $\gz\neq\bz$ if and only if $\gz\in\scg_j$
and $d(y^j_\gz,y^j_\bz)=d(x^{j+1}_\gz,x^{j+1}_\bz)\ge\dz^{j+1}$.
By this fact and the Minkowski inequality, we have
\begin{eqnarray}\label{x.u}
\qquad{\rm I}_2&&=\lf\{\sum_{j\in\zz}\sum_{\bz\in\scg_j}
\lf[\sum_{\{\gz\in\scg_j:\ d(y^j_\gz,\,y^j_\bz)\ge\dz^{j+1}\}}
\lf|\lf(h,\psi^j_\gz\r)\r|
\lf|\lf(\psi^j_\gz,\wz{\psi}^{k,\,i}_{j,\,\bz}\r)\r|\r]^2\r\}^{1/2}\\
&&\noz\le\sum_{s=0}^\fz\lf\{\sum_{j\in\zz}\sum_{\bz\in\scg_j}
\lf[\sum_{\{\gz\in\scg_j:\ 2^s\dz^{j+1}\le d(y^j_\gz,\,y^j_\bz)
<2^{s+1}\dz^{j+1}\}}\lf|\lf(h,\psi^j_\gz\r)\r|
\lf|\lf(\psi^j_\gz,\wz{\psi}^{k,\,i}_{j,\,\bz}\r)\r|\r]^2\r\}^{1/2}.
\end{eqnarray}

Moreover, from Remark \ref{rb.l}(ii), we deduce that
\begin{equation}\label{x.z}
\#\lf\{\bz\in\scg_j:\ 2^s\dz^{j+1}\le d\lf(y^j_\gz,\,y^j_\bz\r)
<2^{s+1}\dz^{j+1}\r\}\ls2^{(s+1)G_0},
\end{equation}
which, together with \eqref{x.u} and the H\"older inequality,
further implies that
\begin{eqnarray*}
{\rm I}_2&&\ls\sum_{s=0}^\fz2^{\frac{(s+1)}{2}G_0}
\lf\{\sum_{j\in\zz}\sum_{\bz\in\scg_j}
\sum_{\{\gz\in\scg_j:\ 2^s\dz^{j+1}\le d(y^j_\gz,\,y^j_\bz)
<2^{s+1}\dz^{j+1}\}}\lf|\lf(h,\psi^j_\gz\r)\r|^2
\lf|\lf(\psi^j_\gz,\wz{\psi}^{k,\,i}_{j,\,\bz}\r)\r|^2\r\}^{1/2}\\
&&\sim\sum_{s=0}^\fz2^{\frac{(s+1)}{2}G_0}
\lf\{\sum_{j\in\zz}\sum_{\gz\in\scg_j}
\lf|\lf(h,\psi^j_\gz\r)\r|^2
\sum_{\{\bz\in\scg_j:\ 2^s\dz^{j+1}\le d(y^j_\gz,\,y^j_\bz)
<2^{s+1}\dz^{j+1}\}}
\lf|\lf(\psi^j_\gz,\wz{\psi}^{k,\,i}_{j,\,\bz}\r)\r|^2\r\}^{1/2}.
\end{eqnarray*}

We then estimate $|(\psi^j_\gz, \wz{\psi}^{k,\,i}_{j,\,\bz})|$  for any
$(j,\gz)\in\scc$ with $\scc$ as in \eqref{4.4.x1},
$s\in\zz_+$ and $\bz$
satisfying
$
2^s\dz^{j+1}\le d(y^j_\gz,\,y^j_\bz)
<2^{s+1}\dz^{j+1}
$.
From \eqref{b.x}, \eqref{b.a}, $\az\in\sca^k_{j,\,\bz}$,
$2^s\dz^{j+1}\le d(y^j_\gz,\,y^j_\bz)$, the H\"older inequality
and \eqref{3.9x1}, it follows that
\begin{eqnarray*}
\lf|\lf(\psi^j_\gz, \wz{\psi}^{k,\,i}_{j,\,\bz}\r)\r|
&&\le\int_\cx \lf|\psi^j_\gz(x)\wz{\psi}^{k,\,i}_{j,\,\bz}(x)\r|\,d\mu(x)
\ls e^{\nu\dz2^{k-2}}\int_{B(x^j_\az,\,8\dz^j)}
\lf|\psi^j_\gz(x)\psi^j_\bz(x)\r|\,d\mu(x)\\
&&\ls e^{\nu\dz2^{k-2}}
\int_{B(x^j_\az,\,8\dz^j)}\frac{e^{-\nu\dz^{-j}d(y_\gz^j,\,x)}}
{\sqrt{V(y^j_\gz,\dz^j)}}\frac{e^{-\nu\dz^{-j}d(y_\bz^j,\,x)}}
{\sqrt{V(y^j_\bz,\dz^j)}}\,d\mu(x)\\
&&\ls e^{\nu\dz2^{k-2}}e^{-\frac{\nu}{4}\dz^{-j}
d(y_\gz^j,\,y^j_\bz)}
e^{-\frac{\nu}{4}\dz^{-j}d(y_\bz^j,\,x^j_\az)}\\
&&\hs\times\int_{B(x^j_\az,\,8\dz^j)}
\frac{e^{-\frac{3\nu}{4}\dz^{-j}d(y_\gz^j,\,x)}}
{\sqrt{V(y^j_\gz,\dz^j)}}\frac{e^{-\frac{\nu}{2}\dz^{-j}d(y_\bz^j,\,x)}}
{\sqrt{V(y^j_\bz,\dz^j)}}\,d\mu(x)\\
&&\ls e^{-\nu\dz2^{s-2}}
\lf\|\frac{e^{-\frac{3\nu}{4}\dz^{-j}d(y_\gz^j,\,\cdot)}}
{\sqrt{V(y^j_\gz,\dz^j)}}\r\|_{\ltw}
\lf\|\frac{e^{-\frac{\nu}{2}\dz^{-j}d(y_\bz^j,\,\cdot)}}
{\sqrt{V(y^j_\bz,\dz^j)}}\r\|_{\ltw}
\ls e^{-\nu\dz2^{s-2}}.
\end{eqnarray*}
Thus, by this, \eqref{x.z} and Theorem \ref{tb.a}, we obtain
$$
{\rm I}_2\ls\sum_{s=0}^\fz2^{(s+1)G_0}
e^{-\nu\dz2^{s-2}}
\lf\{\sum_{j\in\zz}\sum_{\gz\in\scg_j}
\lf|\lf(h,\psi^j_\gz\r)\r|^2\r\}^{1/2}
\ls\|h\|_{\ltw}.
$$

Now we turn to consider ${\rm I}_3$.
We first estimate $|(\psi^s_\gz, \wz{\psi}^{k,\,i}_{j,\,\bz})|$ for any
$(j,\gz),\,(s,\gz)\in\scc$ with
$d(y^s_\gz,\,y^j_\bz)\ge\dz^{j+1}$ and $s\in\zz\cap(-\fz,j-1]$.
From $s^j_\az\in V_j$, $\psi^j_\bz\in W_j$
and $V_j\bot W_j$ with $V_j$ and $W_j$ for all
$j\in\zz$ as in Lemma \ref{le.a}, it follows that
$$
\int_\cx\wz{\psi}^{k,\,i}_{j,\,\bz}(x)\,d\mu(x)
=e^{\nu\dz2^{k-2}}\int_\cx s^j_\az(x)\psi^j_\bz(x)\,d\mu(x)=0,
$$
which, combined with \eqref{b.x}, \eqref{b.a}, $\az\in\sca^k_{j,\,\bz}$,
$d(y^s_\gz,\,y^j_\bz)\ge\dz^{j+1}$, the H\"older inequality
and \eqref{3.9x1}, further implies that
\begin{eqnarray*}
\lf|\lf(\psi^s_\gz, \wz{\psi}^{k,\,i}_{j,\,\bz}\r)\r|
&&=\lf|\int_{\cx}\lf[\psi^s_\gz(x)-\psi^s_\gz\lf(y^j_\bz\r)\r]
\wz{\psi}^{k,\,i}_{j,\,\bz}(x)\,d\mu(x)\r|\\
&&\ls e^{\nu\dz2^{k-2}}\int_{B(x^j_\az,\,8\dz^j)}
\lf|\psi^s_\gz(x)-\psi^s_\gz\lf(y^j_\bz\r)\r|
\lf|\psi^j_\bz(x)\r|\,d\mu(x)\\
&&\ls e^{\nu\dz2^{k-2}}\int_{B(x^j_\az,\,8\dz^j)}
\lf|\psi^s_\gz(x)-\psi^s_\gz\lf(y^j_\bz\r)\r|
\frac{e^{-\nu\dz^{-j}d(y_\bz^j,\,x)}}
{\sqrt{V(y^j_\bz,\dz^j)}}\,d\mu(x)\\
&&\ls e^{\nu\dz2^{k-2}}e^{-\frac{\nu}{2}\dz^{-j}d(y_\bz^j,\,x^j_\az)}
\lf\|\frac{e^{-\frac{\nu}{4}\dz^{-j}d(y_\bz^j,\,\cdot)}}
{\sqrt{V(y^j_\bz,\dz^j)}}\r\|_{\ltw}\\
&&\hs\times\lf\{\int_{B(x^j_\az,\,8\dz^j)}
\lf|\psi^s_\gz(x)-\psi^s_\gz\lf(y^j_\bz\r)\r|^2
e^{-\frac{\nu}{2}\dz^{-j}d(y_\bz^j,\,x)}\,d\mu(x)\r\}^{1/2}\ls {\rm E}^{1/2},
\end{eqnarray*}
where
$$
{\rm E}:=\int_{\cx}\lf|\psi^s_\gz(x)-\psi^s_\gz\lf(y^j_\bz\r)\r|^2
e^{-\frac{\nu}{2}\dz^{-j}d(y_\bz^j,\,x)}\,d\mu(x).
$$

Now we estimate ${\rm E}$ by writing
\begin{eqnarray*}
{\rm E}&&=\int_{B(y^j_\bz,\,\dz^j)}
\lf|\psi^s_\gz(x)-\psi^s_\gz\lf(y^j_\bz\r)\r|^2
e^{-\frac{\nu}{2}\dz^{-j}d(y_\bz^j,\,x)}\,d\mu(x)\\
&&\hs+\int_{B(y^j_\bz,\,\dz^s)\bh B(y^j_\bz,\,\dz^j)}\cdots
+\int_{\cx\bh B(y^j_\bz,\,\dz^s)}\cdots
=:{\rm E}_1+{\rm E}_2+{\rm E}_3.
\end{eqnarray*}

For ${\rm E}_1$, by $\dz^j<\dz^s$, \eqref{b.b}
and \eqref{3.9x1}, we know that
\begin{eqnarray*}
{\rm E}_1&&\ls\int_{B(y^j_\bz,\,\dz^j)}
\frac{e^{-2\nu\dz^{-s}d(y_\gz^s,\,x)}}
{V(y^s_\gz,\dz^s)}
\lf[\frac{d(x,y^j_\bz)}{\dz^s}\r]^{2\eta}
e^{-\frac{\nu}{2}\dz^{-j}d(y_\bz^j,\,x)}\,d\mu(x)\\
&&\ls e^{-\nu\dz^{-s}d(y_\gz^s,\,y^j_\bz)}
\dz^{2(j-s)\eta}\frac1{V(y^s_\gz,\dz^s)}
\int_{B(y^j_\bz,\,\dz^j)}
e^{-\nu\dz^{-s}d(y_\gz^s,\,x)}\,d\mu(x)\\
&&\ls e^{-\nu\dz^{-s}d(y_\gz^s,\,y^j_\bz)}
\dz^{2(j-s)\eta}.
\end{eqnarray*}

For ${\rm E}_2$, by \eqref{b.b} and \eqref{3.9x1}, we have
\begin{eqnarray*}
{\rm E}_2&&\ls\sum_{t=0}^{j-s-1}
\int_{B(y^j_\bz,\,\dz^{s+t})\bh B(y^j_\bz,\,\dz^{s+t+1})}
\frac{e^{-2\nu\dz^{-s}d(y_\gz^s,\,x)}}
{V(y^s_\gz,\dz^s)}
\lf[\frac{d(x,y^j_\bz)}{\dz^s}\r]^{2\eta}
e^{-\frac{\nu}{2}\dz^{-j}d(y_\bz^j,\,x)}\,d\mu(x)\\
&&\ls\sum_{k=0}^{j-s-1}\dz^{2t\eta}
\frac{e^{-\nu\dz^{-s}d(y_\gz^s,\,y^j_\bz)}}
{V(y^s_\gz,\dz^s)}
\int_{B(y^j_\bz,\,\dz^{s+t})\bh B(y^j_\bz,\,\dz^{s+t+1})}
e^{-\nu\dz^{-s}d(y_\gz^s,\,x)}
e^{-\frac{\nu}{2}\dz^{-j}d(y^j_\bz,\,x)}\,d\mu(x)\\
&&\ls e^{-\nu\dz^{-s}d(y_\gz^s,\,y^j_\bz)}
\sum_{t=0}^{j-s-1}\dz^{2t\eta}
e^{-\frac{\nu}{2}\dz^{s+t+1-j}}
\int_{B(y^j_\bz,\,\dz^{s+t})\bh B(y^j_\bz,\,\dz^{s+t+1})}
\frac{e^{-\nu\dz^{-s}d(y_\gz^s,\,x)}}{V(y^s_\gz,\dz^s)}
\,d\mu(x)\\
&&\ls e^{-\nu\dz^{-s}d(y_\gz^s,\,y^j_\bz)}
\sum_{t=0}^{j-s-1}\dz^{2t\eta}
e^{-\frac{\nu}{2}\dz^{s+t+1-j}}.
\end{eqnarray*}

To show ${\rm E}_3$, by \eqref{b.a} and \eqref{3.9x1},
we conclude that
\begin{eqnarray*}
{\rm E}_3&&\ls\sum_{t=0}^\fz\int_{B(y^j_\bz,\,2^{t+1}\dz^s)
\bh B(y^j_\bz,\,2^t\dz^s)}
\lf[\lf|\psi^s_\gz(x)\r|^2+\lf|\psi^s_\gz\lf(y^j_\bz\r)\r|^2\r]
e^{-\frac{\nu}{2}\dz^{-j}d(y_\bz^j,\,x)}\,d\mu(x)\\
&&\ls\sum_{t=0}^\fz\int_{B(y^j_\bz,\,2^{t+1}\dz^s)
\bh B(y^j_\bz,\,2^t\dz^s)}
\lf[e^{-2\nu\dz^{-s}d(y_\gz^s,\,x)}
+e^{-2\nu\dz^{-s}d(y_\gz^s,\,y^j_\bz)}\r]
\frac{e^{-\frac{\nu}{2}\dz^{-j}d(y_\bz^j,\,x)}}
{V(y^s_\gz,\dz^s)}\,d\mu(x)\\
&&\ls e^{-\frac{\nu}{4}\dz^{-s}d(y_\gz^s,\,y^j_\bz)}
\sum_{t=0}^\fz e^{-\nu2^{t-2}\dz^{s-j}}
\frac1{V(y^s_\gz,\dz^s)}\lf\{\int_{B(y^j_\bz,\,2^{t+1}\dz^s)
\bh B(y^j_\bz,\,2^t\dz^s)}
e^{-\frac{7\nu}{4}\dz^{-s}d(y_\gz^s,\,x)}\,d\mu(x)\r.\\
&&\lf.\hs+e^{-\frac{3\nu}{2}\dz^{-s}d(y^j_\bz,\,y_\gz^s)}
\int_{B(y^j_\bz,\,2^{t+1}\dz^s)
\bh B(y^j_\bz,\,2^t\dz^s)}
e^{-\frac{\nu}{4}\dz^{-s}d(y_\gz^s,\,x)}\,d\mu(x)\r\}\\
&&\ls e^{-\frac{\nu}{4}\dz^{-s}d(y_\gz^s,\,y^j_\bz)}
\sum_{t=0}^\fz 2^{-2(t-2)\eta}\dz^{2(j-s)\eta}
\ls e^{-\frac{\nu}{4}\dz^{-s}d(y_\gz^s,\,y^j_\bz)}\dz^{2(j-s)\eta}.
\end{eqnarray*}

Combining the estimates for ${\rm E}_1$, ${\rm E}_2$ and ${\rm E}_3$, we have
\begin{eqnarray}\label{4.16x1}
\lf|\lf(\psi^s_\gz, \wz{\psi}^{k,\,i}_{j,\,\bz}\r)\r|
&&\ls{\rm E}^{1/2}\ls\lf\{e^{-\frac{\nu}{4}\dz^{-s}d(y_\gz^s,\,y^j_\bz)}
\sum_{t=0}^{j-s}\dz^{2t\eta}
e^{-\frac{\nu}{2}\dz^{s+t+1-j}}\r\}^{1/2}\\
&&\noz=:e^{-\frac{\nu}{8}\dz^{-s}d(y_\gz^s,\,y^j_\bz)}S_{j,\,s},
\end{eqnarray}
where $S_{j,\,s}:=\{\sum_{t=0}^{j-s}\dz^{2t\eta}
e^{-\frac{\nu}{2}\dz^{s+t+1-j}}\}^{1/2}$.

Moreover, Observe that, by \eqref{b.w}, \eqref{b.z},
$\scg_j\cap\scg_s=\emptyset$
and $\sca_{s+1}\st\sca_{j+1}$ for $s,\,j\in\zz$ with $s<j$, we know that,
for given $j\in\zz$, $\bz\in\scg_j$ and $s\in\zz\cap(-\fz,j)$,
$\gz\in\scg_s$ if and only if $\gz\in\scg_s$ and
$d(y^s_\gz,y^j_\bz)=d(x^{s+1}_\gz,x^{j+1}_\bz)\ge\dz^{j+1}$.
From this, together with \eqref{4.16x1}
and the H\"older inequality, we deduce that
\begin{eqnarray*}
{\rm I}_3&&=\lf\{\sum_{j\in\zz}\sum_{\bz\in\scg_j}
\lf[\sum_{s=-\fz}^{j-1}\sum_{\{\gz\in\scg_s:
\ d(y^s_\gz,\,y^j_\bz)\ge\dz^{j+1}\}}
\lf|\lf(h,\psi^s_\gz\r)\r|
\lf|\lf(\psi^s_\gz,\wz{\psi}^{k,\,i}_{j,\,\bz}\r)\r|\r]^2\r\}^{1/2}\\
&&\ls\lf\{\sum_{j\in\zz}\sum_{\bz\in\scg_j}
\lf[\sum_{s=-\fz}^{j-1}\sum_{\{\gz\in\scg_s:
\ d(y^s_\gz,\,y^j_\bz)\ge\dz^{j+1}\}}
\lf|\lf(h,\psi^s_\gz\r)\r|
e^{-\frac{\nu}{8}\dz^{-s}d(y_\gz^s,\,y^j_\bz)}S_{j,\,s}
\r]^2\r\}^{1/2}\\
&&\ls\lf\{\sum_{j\in\zz}\sum_{\bz\in\scg_j}
\lf[\sum_{s=-\fz}^{j-1}\sum_{\{\gz\in\scg_s:
\ d(y^s_\gz,\,y^j_\bz)\ge\dz^{j+1}\}}
\lf|\lf(h,\psi^s_\gz\r)\r|^2
e^{-\frac{\nu}{8}\dz^{-s}d(y_\gz^s,\,y^j_\bz)}S_{j,\,s}
\r]\r.\\
&&\hs\lf.\times\lf[\sum_{s=-\fz}^{j-1}
\sum_{\{\gz\in\scg_s:\ d(y^s_\gz,\,y^j_\bz)\ge\dz^{j+1}\}}
e^{-\frac{\nu}{8}\dz^{-s}d(y_\gz^s,\,y^j_\bz)}S_{j,\,s}
\r]\r\}^{1/2}.
\end{eqnarray*}
Observe that, by the elementary inequality
\begin{equation}\label{x.x}
\lf[\sum_{j=0}^\fz\lf|a_j\r|\r]^p\le\sum_{j=0}^\fz\lf|a_j\r|^p
\quad {\rm for\ all\ }\{a_j\}_{j=0}^{\fz}\st\cc\ {\rm and\ }
p\in(0,1],
\end{equation}
we have
\begin{eqnarray}\label{4.16x2}
\sum_{s=-\fz}^{j-1}S_{j,\,s}
&&\le\sum_{s=-\fz}^{j-1}\sum_{t=0}^{j-s}\dz^{t\eta}
e^{-\frac{\nu}{4}\dz^{s+t+1-j}}
=\sum_{i=1}^{\fz}\sum_{t=0}^i\dz^{t\eta}
e^{-\frac{\nu}{4}\dz^{t+1-i}}\\
&&\noz=\sum_{t=0}^{\fz}\dz^{t\eta}\sum_{i=t}^\fz
e^{-\frac{\nu}{4}\dz^{t+1-i}}
=\sum_{t=0}^{\fz}\dz^{t\eta}\sum_{j=0}^\fz
e^{-\frac{\nu}{4}\dz^{1-j}}\ls1
\end{eqnarray}
and, similarly,
\begin{eqnarray}\label{4.16x3}
\sum_{j=s+1}^{\fz}S_{j,\,s}
\le\sum_{j=s+1}^{\fz}\sum_{t=0}^{j-s}\dz^{t\eta}
e^{-\frac{\nu}{4}\dz^{s+t+1-j}}
=\sum_{i=1}^{\fz}\sum_{t=0}^i\dz^{t\eta}
e^{-\frac{\nu}{4}\dz^{t+1-i}}\ls1.
\end{eqnarray}
Thus, from Lemma \ref{lb.x} and \eqref{4.16x2}, we deduce that
\begin{eqnarray*}
\sum_{s=-\fz}^{j-1}\sum_{\{\gz\in\scg_s:
\ d(y^s_\gz,\,y^j_\bz)\ge\dz^{j+1}\}}
e^{-\frac{\nu}{8}\dz^{-s}d(y_\gz^s,\,y^j_\bz)}S_{j,\,s}
\ls\sum_{s=-\fz}^{j-1}S_{j,\,s}
e^{-\frac{\nu}{16}\dz^{-s}d(\scy^s,\,y^j_\bz)}
\ls\sum_{s=-\fz}^{j-1}S_{j,\,s}\ls1,
\end{eqnarray*}
where $\scy^s$ for $s\in\zz$ is as in \eqref{3.5x},
which, together with Lemma \ref{lb.x} again and \eqref{4.16x3},
further implies that
\begin{eqnarray*}
{\rm I}_3&&\ls\lf\{\sum_{j\in\zz}\sum_{\bz\in\scg_j}
\sum_{s=-\fz}^{j-1}\sum_{\{\gz\in\scg_s:\ d(y^s_\gz,\,y^j_\bz)\ge\dz^{j+1}\}}
\lf|\lf(h,\psi^s_\gz\r)\r|^2
e^{-\frac{\nu}{8}\dz^{-s}d(y_\gz^s,\,y^j_\bz)}S_{j,\,s}\r\}^{1/2}\\
&&\sim\lf\{\sum_{j\in\zz}\sum_{s=-\fz}^{j-1}S_{j,\,s}
\sum_{\gz\in\scg_s}\lf|\lf(h,\psi^s_\gz\r)\r|^2
\sum_{\{\bz\in\scg_j:\ d(y^s_\gz,\,y^j_\bz)\ge\dz^{j+1}\}}
e^{-\frac{\nu}{8}\dz^{-s}d(y_\gz^s,\,y^j_\bz)}\r\}^{1/2}\\
&&\ls\lf\{\sum_{j\in\zz}\sum_{s=-\fz}^{j-1}S_{j,\,s}
\sum_{\gz\in\scg_s}\lf|\lf(h,\psi^s_\gz\r)\r|^2
e^{-\frac{\nu}{16}\dz^{-s}d(y_\gz^s,\,\scy^j)}\r\}^{1/2}\\
&&\ls\lf\{\sum_{j\in\zz}\sum_{s=-\fz}^{j-1}S_{j,\,s}
\sum_{\gz\in\scg_s}\lf|\lf(h,\psi^s_\gz\r)\r|^2\r\}^{1/2}
\sim\lf\{\sum_{s\in\zz}\lf(\sum_{j=s+1}^{\fz}S_{j,\,s}\r)
\sum_{\gz\in\scg_s}\lf|\lf(h,\psi^s_\gz\r)\r|^2\r\}^{1/2}\\
&&\ls\lf\{\sum_{s\in\zz}
\sum_{\gz\in\scg_s}\lf|\lf(h,\psi^s_\gz\r)\r|^2\r\}^{1/2}
\sim\|h\|_{\ltw}.
\end{eqnarray*}

Finally, we deal with ${\rm I}_4$.
We first estimate $|(\psi^s_\gz, \wz{\psi}^{k,\,i}_{j,\,\bz})|$ for any
$(j,\gz),\,(s,\gz)\in\scc$ with
$d(y^s_\gz,\,y^j_\bz)\ge\dz^{s+1}$ and $s\in\zz\cap[j+1,\fz)$.
From $\psi^s_\gz\in W_s$, $s^j_\az\in V_j\st V_s$ and
$W_s\bot V_s$ with $V_k$ and $W_k$ for any $k\in\zz$
as in Lemma \ref{le.a}, it follows that
$$\int_\cx\psi^s_\gz(x)s^j_\az(x)\,d\mu(x)=0,$$
which, together with
\eqref{b.x}, \eqref{b.a}, the H\"older inequality
and \eqref{3.9x1}, further implies that
\begin{eqnarray*}
\lf|\lf(\psi^s_\gz, \wz{\psi}^{k,\,i}_{j,\,\bz}\r)\r|
&&=\lf|e^{\nu\dz2^{k-2}}\int_{\cx}
\psi^s_\gz(x)s^j_\az(x)\psi^j_\bz(x)\,d\mu(x)\r|\\
&&=\lf|e^{\nu\dz2^{k-2}}\int_{\cx}
\psi^s_\gz(x)s^j_\az(x)
\lf[\psi^j_\bz(x)-\psi^j_\bz\lf(y^s_\gz\r)\r]\,d\mu(x)\r|\\
&&\ls e^{\nu\dz2^{k-2}}\int_{B(x^j_\az,\,8\dz^j)}
\lf|\psi^s_\gz(x)\r|
\lf|\psi^j_\bz(x)-\psi^j_\bz\lf(y^s_\gz\r)\r|
\,d\mu(x)\\
&&\ls e^{\nu\dz2^{k-2}}\int_{B(x^j_\az,\,8\dz^j)}
\lf|\psi^j_\bz(x)-\psi^j_\bz\lf(y^s_\gz\r)\r|
\frac{e^{-\nu\dz^{-s}d(y_\gz^s,\,x)}}
{\sqrt{V(y^s_\gz,\dz^s)}}\,d\mu(x)\\
&&\ls e^{\nu\dz2^{k-2}}
\lf\|\frac{e^{-\frac{\nu}{2}\dz^{-s}d(y_\gz^s,\,\cdot)}}
{\sqrt{V(y^s_\gz,\dz^s)}}\r\|_{\ltw}\\
&&\hs\times\lf\{\int_{B(x^j_\az,\,8\dz^j)}
\lf|\psi^j_\bz(x)-\psi^j_\bz\lf(y^s_\gz\r)\r|^2
e^{-\nu\dz^{-s}d(y_\bz^j,\,x)}\,d\mu(x)\r\}^{1/2}\ls {\rm F}^{1/2},
\end{eqnarray*}
where
$$
{\rm F}:= e^{\nu\dz2^{k-1}}\int_{B(x^j_\az,\,8\dz^j)}
\lf|\psi^j_\bz(x)-\psi^j_\bz\lf(y^s_\gz\r)\r|^2
e^{-\nu\dz^{-s}d(y_\bz^j,\,x)}\,d\mu(x).
$$

Now we estimate ${\rm F}$ by writing
\begin{eqnarray*}
{\rm F}&&=\int_{B(y^s_\gz,\,\dz^s)}
e^{\nu\dz2^{k-1}}\chi_{B(x^j_\az,\,8\dz^j)}(x)
\lf|\psi^j_\bz(x)-\psi^j_\bz\lf(y^s_\gz\r)\r|^2
e^{-\nu\dz^{-s}d(y_\bz^j,\,x)}\,d\mu(x)\\
&&\hs+\int_{B(y^s_\gz,\,\dz^j)\bh B(y^s_\gz,\,\dz^s)}\cdots
+\int_{\cx\bh B(y^s_\gz,\,\dz^j)}\cdots
=:{\rm F}_1+{\rm F}_2+{\rm F}_3.
\end{eqnarray*}

For ${\rm F}_1$, by \eqref{b.x}, $\dz^s<\dz^j$, \eqref{b.b},
$\az\in\sca^k_{j,\,\bz}$ and \eqref{3.9x1}, we know that
\begin{eqnarray*}
{\rm F}_1&&\ls e^{\nu\dz2^{k-1}}\int_{B(y^s_\gz,\,\dz^s)
\cap B(x^j_\az,\,8\dz^j)}
\frac{e^{-2\nu\dz^{-j}d(y_\bz^j,\,x)}}
{V(y^j_\bz,\dz^j)}
\lf[\frac{d(x,y^s_\gz)}{\dz^j}\r]^{2\eta}
e^{-\nu\dz^{-s}d(y_\gz^s,\,x)}\,d\mu(x)\\
&&\ls e^{\nu\dz2^{k-1}}e^{-\frac{\nu}{2}\dz^{-j}d(y^s_\gz,\,y_\bz^j)}
\dz^{2(s-j)\eta}e^{-\frac{\nu}{2}\dz^{-j}d(y_\bz^j,\,x^j_\az)}
\int_{B(y^s_\gz,\,\dz^s)}
\frac{e^{-\nu\dz^{-j}d(y_\bz^j,\,x)}}
{V(y^j_\bz,\dz^j)}\,d\mu(x)\\
&&\ls e^{-\frac{\nu}{2}\dz^{-j}d(y_\gz^s,\,y^j_\bz)}
\dz^{2(s-j)\eta}.
\end{eqnarray*}

For ${\rm F}_2$, by \eqref{b.x}, $\dz^s<\dz^j$, \eqref{b.b},
$\az\in\sca^k_{j,\,\bz}$ and \eqref{3.9x1}, we have
\begin{eqnarray*}
{\rm F}_2&&\ls e^{\nu\dz2^{k-1}}\sum_{t=0}^{s-j-1}
\int_{[B(y^s_\gz,\,\dz^{j+t})\bh B(y^s_\gz,\,\dz^{j+t+1})]
\cap B(x^j_\az,\,8\dz^j)}
\frac{e^{-2\nu\dz^{-j}d(y_\bz^j,\,x)}}
{V(y^j_\bz,\dz^j)}\lf[\frac{d(x,y^s_\gz)}{\dz^j}\r]^{2\eta}\\
&&\hs\times e^{-\nu\dz^{-s}d(y_\gz^s,\,x)}\,d\mu(x)\\
&&\ls e^{\nu\dz2^{k-1}}
e^{-\frac{\nu}{2}\dz^{-j}d(y^j_\bz,\,x_\az^j)}
e^{-\frac{\nu}{2}\dz^{-j}d(y_\gz^s,\,y^j_\bz)}\\
&&\hs\times\sum_{t=0}^{s-j-1}\dz^{2t\eta}
e^{-\nu\dz^{j+t+1-s}}
\frac1{V(y^j_\bz,\dz^j)}
\int_{B(y^s_\gz,\,\dz^{j+t})\bh B(y^s_\gz,\,\dz^{j+t+1})}
e^{-\nu\dz^{-j}d(y_\bz^j,\,x)}\,d\mu(x)\\
&&\ls e^{-\frac{\nu}{2}\dz^{-j}d(y_\gz^s,\,y^j_\bz)}
\sum_{t=0}^{s-j-1}\dz^{2t\eta}e^{-\nu\dz^{j+t+1-s}}.
\end{eqnarray*}

To estimate ${\rm F}_3$, by \eqref{b.a} and \eqref{3.9x1}, we conclude that
\begin{eqnarray*}
{\rm F}_3&&\ls e^{\nu\dz2^{k-1}}\sum_{t=0}^\fz\int_{[B(y^s_\gz,\,2^{t+1}\dz^{j})
\bh B(y^s_\gz,\,2^t\dz^{j})]\cap B(x^j_\az,\,8\dz^j)}
\lf[\lf|\psi^j_\bz(x)\r|^2+\lf|\psi^j_\bz\lf(y^s_\gz\r)\r|^2\r]\\
&&\hs\times e^{-\nu\dz^{-s}d(y_\gz^s,\,x)}\,d\mu(x)\\
&&\ls e^{\nu\dz2^{k-1}}\sum_{t=0}^\fz\int_{[B(y^s_\gz,\,2^{t+1}\dz^{j})
\bh B(y^s_\gz,\,2^t\dz^{j})]\cap B(x^j_\az,\,8\dz^j)}
\lf[e^{-2\nu\dz^{-j}d(y_\bz^j,\,x)}
+e^{-2\nu\dz^{-j}d(y_\gz^s,\,y^j_\bz)}\r]\\
&&\hs\times\frac{e^{-\nu\dz^{-s}d(y_\gz^s,\,x)}}
{V(y^j_\bz,\dz^j)}\,d\mu(x)\\
&&\ls e^{\nu\dz2^{k-1}}
e^{-\frac{\nu}{2}\dz^{-j}d(y^j_\bz,\,x_\az^j)}
\sum_{t=0}^\fz\frac1{V(y^j_\bz,\dz^j)}\\
&&\hs\times\lf\{e^{-\nu 2^{t-1}\dz^{j-s}}
e^{-\frac{\nu}{2}\dz^{-j}d(y_\gz^s,\,y^j_\bz)}
\int_{B(y^s_\gz,\,2^{t+1}\dz^j)
\bh B(y^s_\gz,\,2^t\dz^j)}
e^{-\nu\dz^{-j}d(y_\bz^j,\,x)}\,d\mu(x)\r.\\
&&\lf.\hs\hs+e^{-\nu 2^{t-2}\dz^{j-s}}
e^{-\frac{5\nu}{4}\dz^{-j}d(y_\gz^s,\,y^j_\bz)}
\int_{B(y^s_\gz,\,2^{t+1}\dz^j)
\bh B(y^s_\gz,\,2^t\dz^j)}
e^{-\frac{\nu}{4}\dz^{-j}d(y_\bz^j,\,x)}\,d\mu(x)\r\}\\
&&\ls e^{-\frac{\nu}{2}\dz^{-j}d(y_\gz^s,\,y^j_\bz)}
\sum_{t=0}^\fz 2^{-2(t-2)\eta}\dz^{2(s-j)\eta}\\
&&\ls e^{-\frac{\nu}{2}\dz^{-j}d(y_\gz^s,\,y^j_\bz)}
\sum_{t=0}^\fz 2^{-2(t-2)\eta}\dz^{2(s-j)\eta}
\ls e^{-\frac{\nu}{2}\dz^{-s}d(y_\gz^s,\,y^j_\bz)}\dz^{2(s-j)\eta}.
\end{eqnarray*}

Combining the estimates for ${\rm F}_1$, ${\rm F}_2$ and ${\rm F}_3$, we have
\begin{eqnarray}\label{7.31x1}
\lf|\lf(\psi^s_\gz, \wz{\psi}^{k,\,i}_{j,\,\bz}\r)\r|
&&\ls{\rm F}^{1/2}\ls e^{-\frac{\nu}{4}\dz^{-j}d(y_\gz^s,\,y^j_\bz)}
\lf\{\sum_{t=0}^{s-j}\dz^{2t\eta}
e^{-\nu\dz^{j+t+1-s}}\r\}^{1/2}\\
&&\noz=:e^{-\frac{\nu}{4}\dz^{-j}d(y_\gz^s,\,y^j_\bz)}T_{j,\,s},
\end{eqnarray}
where $T_{j,\,s}:=\{\sum_{t=0}^{s-j}\dz^{2t\eta}
e^{-\nu\dz^{j+k+1-s}}\}^{1/2}$.

Moreover, Observe that, by \eqref{b.w}, \eqref{b.z},
$\scg_j\cap\scg_s=\emptyset$
and $\sca_j\st\sca_s$ for $s,\,j\in\zz$ with $j<s$, we know that,
for given $j\in\zz$, $\bz\in\scg_j$ and $s\in\zz\cap(j,\fz)$,
$\gz\in\scg_s$ if and only if $\gz\in\scg_s$ and
$d(y^s_\gz,y^j_\bz)=d(x^{s+1}_\gz,x^{j+1}_\bz)\ge\dz^{s+1}$,
which, combined with \eqref{7.31x1} and
the H\"older inequality, further implies that
\begin{eqnarray*}
{\rm I}_4&&=\lf\{\sum_{j\in\zz}\sum_{\bz\in\scg_j}
\lf[\sum_{s=j+1}^{\fz}\sum_{\{\gz\in\scg_s:\ d(y^s_\gz,\,y^j_\bz)\ge\dz^{s+1}\}}
\lf|\lf(h,\psi^s_\gz\r)\r|
\lf|\lf(\psi^s_\gz,\wz{\psi}^{k,\,i}_{j,\,\bz}\r)\r|\r]^2\r\}^{1/2}\\
&&\ls\lf\{\sum_{j\in\zz}\sum_{\bz\in\scg_j}
\lf[\sum_{s=j+1}^{\fz}\sum_{\{\gz\in\scg_s:\ d(y^s_\gz,\,y^j_\bz)\ge\dz^{s+1}\}}
\lf|\lf(h,\psi^s_\gz\r)\r|
e^{-\frac{\nu}{4}\dz^{-s}d(y_\gz^s,\,y^j_\bz)}T_{j,\,s}
\r]^2\r\}^{1/2}\\
&&\ls\lf\{\sum_{j\in\zz}\sum_{\bz\in\scg_j}
\lf[\sum_{s=j+1}^{\fz}\sum_{\{\gz\in\scg_s:\ d(y^s_\gz,\,y^j_\bz)\ge\dz^{s+1}\}}
\lf|\lf(h,\psi^s_\gz\r)\r|^2
e^{-\frac{\nu}{4}\dz^{-s}d(y_\gz^s,\,y^j_\bz)}T_{j,\,s}\r]\r.\\
&&\hs\lf.\times\lf[\sum_{s=j+1}^{\fz}
\sum_{\{\gz\in\scg_s:\ d(y^s_\gz,\,y^j_\bz)\ge\dz^{s+1}\}}
e^{-\frac{\nu}{4}\dz^{-s}d(y_\gz^s,\,y^j_\bz)}T_{j,\,s}
\r]\r\}^{1/2}.
\end{eqnarray*}
Observe that, by \eqref{x.x}, we have
\begin{eqnarray}\label{7.31x2}
\sum_{s=j+1}^{\fz}T_{j,\,s}
&&\le\sum_{s=j+1}^{\fz}\sum_{t=0}^{s-j}\dz^{t\eta}
e^{-\frac{\nu}{2}\dz^{j+t+1-s}}
=\sum_{i=1}^{\fz}\sum_{t=0}^i\dz^{t\eta}
e^{-\frac{\nu}{2}\dz^{t+1-i}}\\
&&\noz=\sum_{t=0}^{\fz}\dz^{t\eta}\sum_{i=t}^\fz
e^{-\frac{\nu}{2}\dz^{t+1-i}}
=\sum_{t=0}^{\fz}\dz^{t\eta}\sum_{j=0}^\fz
e^{-\frac{\nu}{2}\dz^{1-j}}\ls1
\end{eqnarray}
and, similarly,
\begin{equation}\label{7.31x3}
\sum_{j=-\fz}^{s-1}T_{j,\,s}
\le\sum_{j=-\fz}^{s-1}\sum_{t=0}^{s-j}\dz^{t\eta}
e^{-\frac{\nu}{2}\dz^{j+ t+1-s}}
=\sum_{i=1}^{\fz}\sum_{t=0}^i\dz^{t\eta}
e^{-\frac{\nu}{2}\dz^{t+1-i}}\ls1.
\end{equation}
Thus, from Lemma \ref{lb.x} and \eqref{7.31x2}, we deduce that
\begin{eqnarray*}
&&\sum_{j=s+1}^{\fz}\sum_{\{\gz\in\scg_s:
\ d(y^s_\gz,\,y^j_\bz)\ge\dz^{s+1}\}}
e^{-\frac{\nu}{4}\dz^{-j}d(y_\gz^s,\,y^j_\bz)}T_{j,\,s}\\
&&\hs\ls\sum_{s=j+1}^{\fz}T_{j,\,s}
e^{-\frac{\nu}{8}\dz^{-j}d(\scy^s,\,y^j_\bz)}
\ls\sum_{s=j+1}^{\fz}T_{j,\,s}\ls1,
\end{eqnarray*}
which, combined with Lemma \ref{lb.x} again and \eqref{7.31x3},
further implies that
\begin{eqnarray*}
{\rm I}_4&&\ls\lf\{\sum_{j\in\zz}\sum_{\bz\in\scg_j}
\sum_{s=j+1}^{\fz}\sum_{\{\gz\in\scg_s:\ d(y^s_\gz,\,y^j_\bz)\ge\dz^{s+1}\}}
\lf|\lf(h,\psi^s_\gz\r)\r|^2
e^{-\frac{\nu}{4}\dz^{-j}d(y_\gz^s,\,y^j_\bz)}T_{j,\,s}\r\}^{1/2}\\
&&\sim\lf\{\sum_{j\in\zz}\sum_{s=j+1}^{\fz}T_{j,\,s}
\sum_{\gz\in\scg_s}\lf|\lf(h,\psi^s_\gz\r)\r|^2
\sum_{\{\bz\in\scg_j:\ d(y^s_\gz,\,y^j_\bz)\ge\dz^{s+1}\}}
e^{-\frac{\nu}{4}\dz^{-j}d(y_\gz^s,\,y^j_\bz)}\r\}^{1/2}\\
&&\ls\lf\{\sum_{j\in\zz}\sum_{s=j+1}^{\fz}T_{j,\,s}
\sum_{\gz\in\scg_s}\lf|\lf(h,\psi^s_\gz\r)\r|^2
e^{-\frac{\nu}{8}\dz^{-j}d(y_\gz^s,\,\scy^j)}\r\}^{1/2}\\
&&\ls\lf\{\sum_{j\in\zz}\sum_{s=j+1}^{\fz}T_{j,\,s}
\sum_{\gz\in\scg_s}\lf|\lf(h,\psi^s_\gz\r)\r|^2\r\}^{1/2}
\sim\lf\{\sum_{s\in\zz}\lf(\sum_{j=-\fz}^{s-1}T_{j,\,s}\r)
\sum_{\gz\in\scg_s}\lf|\lf(h,\psi^s_\gz\r)\r|^2\r\}^{1/2}\\
&&\ls\lf\{\sum_{s\in\zz}
\sum_{\gz\in\scg_s}\lf|\lf(h,\psi^s_\gz\r)\r|^2\r\}^{1/2}
\sim\|h\|_{\ltw}.
\end{eqnarray*}
This, combined with the estimates for ${\rm I}_1$, ${\rm I}_2$
and ${\rm I}_3$, finishes the proof of \eqref{e.i} and hence
Proposition \ref{pe.f}.
\end{proof}

We also need to establish some estimates of integral kernels
defined as follows. Let $k\in\zz$
and $i\in\{1,\ldots,m_k\}$ with $m_k$ as in \eqref{4.5z}.
For $(x,y)\in(\cx\times\cx)\bh \{(x,x):\ x\in\cx\}$, let
\begin{equation}\label{K-1}
K_{k,\,i}(x,y):=\sum_{j\in\zz}\sum_{\bz\in\scg_j}
\wz{\psi}^{k,\,i}_{j,\,\bz}(x)\overline{\psi^j_\bz(y)},
\end{equation}
where $\scg_j$ for any $j\in\zz$ is as in \eqref{b.w}, and,
for each $N\in\nn$ and $x,\,y\in\cx$,
\begin{equation}\label{K-N}
K^N_{k,\,i}(x,y):=\sum_{(j,\,\bz)\in\scc_N}
\wz{\psi}^{k,\,i}_{j,\,\bz}(x)\overline{\psi^j_\bz(y)},
\end{equation}
where $\scc_N$ for any $N\in\nn$ is as in \eqref{4.5x}.

\begin{proposition}\label{pe.f}
Suppose that $(\cx,d,\mu)$ is a metric measure space of homogeneous type,
$N\in\nn$, $k\in\zz$ and $i\in\{1,\ldots,m_k\}$ with $m_k$ as in \eqref{4.5z}.
Let $K_{k,\,i}$, $K^N_{k,\,i}$ be defined as in \eqref{K-1} and \eqref{K-N}.
Then
$$
K_{k,\,i},\,K^N_{k,\,i}\in L^1_{\loc}(\{\cx\times\cx\}\bh \{(x,x):\ x\in\cx\})
$$
and satisfy \eqref{d.a}, \eqref{d.b} and \eqref{d.g}
with $s:=\eta/2$ and $\eta$ as in \eqref{b.1}.
\end{proposition}

\begin{proof}
Let $N\in\nn$, $k\in\zz$, $i\in\{1,\ldots,m_k\}$,
and the kernels $K_{k,\,i}$ and $K^N_{k,\,i}$ be defined
as in \eqref{K-1} and \eqref{K-N}, respectively.
We only prove the results
for $K_{k,\,i}$, since the proof of $K^N_{k,\,i}$ is similar.
Obviously, if $K_{k,\,i}$ satisfies \eqref{d.a},
then $K_{k,\,i}\in L^1_{\loc}(\{\cx\times\cx\}\bh \{(x,x):\ x\in\cx\})$.
Thus, it suffices to show that
$K_{k,\,i}$ satisfies \eqref{d.a}, \eqref{d.b} and \eqref{d.g}.

Now we prove that $K_{k,\,i}$ satisfies \eqref{d.a}.
By \eqref{b.x}, \eqref{b.a} and $\az\in\sca^k_{j,\,\bz}$
with $\sca^k_{j,\,\bz}$ as in \eqref{4.5y}, we see that,
for all $x,\,y\in\cx$ with $x\neq y$,
\begin{eqnarray}\label{x.v}
\qquad|K_{k,\,i}(x,y)|&&\ls e^{\nu\dz2^{k-2}}\sum_{j\in\zz}\sum_{\bz\in\scg_j}
\chi_{B(x^j_\az,\,8\dz^j)}(x)
e^{-\frac{\nu}{2}\dz^{-j}d(y^j_\bz,\,x)}
\lf|\psi^j_\bz(x)\r|^{1/2}
\frac{|\psi^j_\bz(y)|}{[V(y^j_\bz,\dz^j)]^{1/4}}\\
&&\noz\ls e^{\nu\dz2^{k-2}}e^{-\frac{\nu}{2}\dz^{-j}d(y^j_\bz,\,x^j_\az)}
\sum_{j\in\zz}\sum_{\bz\in\scg_j}
\lf|\psi^j_\bz(x)\r|^{1/2}
\frac{|\psi^j_\bz(y)|}{[V(y^j_\bz,\dz^j)]^{1/4}}\\
&&\noz\ls\sum_{j\in\zz}\sum_{\bz\in\scg_j}
\lf|\psi^j_\bz(x)\r|^{1/2}
\frac{|\psi^j_\bz(y)|}{[V(y^j_\bz,\dz^j)]^{1/4}}=:{\rm H}
\end{eqnarray}
with $\scg_j$ for $j\in\zz$ as in \eqref{b.w}.

Notice that, by \eqref{a.b}, for $j\in\zz$ and $\bz\in\scg_j$
\begin{equation}\label{x.t}
V\lf(y,\dz^j\r)\le V\lf(y^j_\bz,\dz^j+d\lf(y,y^j_\bz\r)\r)
\ls\lf[\frac{\dz^j+d(y,y^j_\bz)}{\dz^j}\r]^n
V\lf(y^j_\bz,\dz^j\r).
\end{equation}
By this and $y^j_\bz:=x^{j+1}_\bz$ for all $\bz\in\scg_j$,
together with \eqref{b.a} and Lemma \ref{lb.x}, we write
\begin{eqnarray*}
{\rm H}&&\ls\sum_{j\in\zz}\sum_{\bz\in\scg_j}
e^{-\frac{\nu}{2}\dz^{-j}d(x,\,y^j_\bz)}
e^{-\nu\dz^{-j}d(y,\,y^j_\bz)}
\frac1{V(y^j_\bz,\dz^j)}\\
&&\ls\sum_{j\in\zz}\sum_{\bz\in\scg_j}
e^{-\frac{\nu}{2}\dz^{-j}d(x,\,y)}
e^{-\frac{\nu}{2}\dz^{-j}d(y,\,y^j_\bz)}
\lf[\frac{\dz^j+d(y,y^j_\bz)}{\dz^j}\r]^n\frac1{V(y,\dz^j)}\\
&&\ls\sum_{j\in\zz}\frac1{V(y,\dz^j)}e^{-\frac{\nu}{2}\dz^{-j}d(x,\,y)}
\sum_{\bz\in\scg_j}e^{-\frac{\nu}{4}\dz^{-j}d(y,\,y^j_\bz)}
\ls\sum_{j\in\zz}\frac1{V(y,\dz^j)}e^{-\frac{\nu}{2}\dz^{-j}d(x,\,y)}
e^{-\frac{\nu}{8}\dz^{-j}d(y,\,\scy^j)}\\
&&\sim\sum_{\{j\in\zz:\ \dz^j\ge d(x,\,y)\}}
\frac1{V(y,\dz^j)}e^{-\frac{\nu}{2}\dz^{-j}d(x,\,y)}
e^{-\frac{\nu}{8}\dz^{-j}d(y,\,\scy^j)}
+\sum_{\{j\in\zz:\ \dz^j<d(x,\,y)\}}\cdots=:{\rm H}_1+{\rm H}_2,
\end{eqnarray*}
where, for any $j\in\zz$, $\scy^j$ is as in \eqref{3.5x}.

To estimate ${\rm H}_1$, from \cite[Lemma 8.3]{ah13}
(with $a:=1$, $\nu:=1$ and $\gz$ replaced by $\nu/8$)
and \eqref{a.e}, it follows that
$$
{\rm H}_1\ls\sum_{\{j\in\zz:\ \dz^j\ge d(x,\,y)\}}
\frac1{V(y,\dz^j)}e^{-\frac{\nu}{8}\dz^{-j}d(y,\,\scy^j)}
\ls\frac1{V(y,x)}\sim\frac1{V(x,y)}.
$$

Now we deal with ${\rm H}_2$. By \eqref{a.b} and \eqref{a.e},
we obtain
\begin{eqnarray*}
{\rm H}_2&&\ls\sum_{\{j\in\zz:\ \dz^j<d(x,\,y)\}}
\frac1{V(y,x)}\lf[\frac{d(x,y)}{\dz^j}\r]^n
e^{-\frac{\nu}{2}\dz^{-j}d(x,\,y)}\\
&&\ls\frac1{V(x,y)}\sum_{\{j\in\zz:\ \dz^j<d(x,\,y)\}}
\frac{\dz^j}{d(x,y)}\ls\frac1{V(x,y)}.
\end{eqnarray*}

Combining the estimates of ${\rm H}_1$ and ${\rm H}_2$,
we further know that ${\rm H}\ls{\rm H}_1+{\rm H}_2\ls\frac1{V(x,y)}$,
which, together with \eqref{x.v}, implies that
$K_{k,\,i}$ satisfies \eqref{d.a}.

Then we prove that $K_{k,\,i}$ satisfies \eqref{d.g}.
For all $x,\,y,\,\wz{y}\in\cx$ with $0<d(y,\wz{y})\le\frac12 d(x,y)$,
by \eqref{b.x}, \eqref{b.a} and $\az\in\sca^k_{j,\,\bz}$
with $\sca^k_{j,\,\bz}$ as in \eqref{4.5y},
together with $y^j_\bz:=x^{j+1}_\bz$ for all $\bz\in\scg_j$, we have
\begin{eqnarray}\label{x.s}
&&\lf|K_{k,\,i}(x,y)-K_{k,\,i}(x,\wz{y})\r|\\
&&\noz\hs\ls e^{\nu\dz2^{k-2}}\sum_{j\in\zz}\sum_{\bz\in\scg_j}
\chi_{B(x^j_\az,\,8\dz^j)}(x)
e^{-\frac{\nu}{2}\dz^{-j}d(y^j_\bz,\,x)}
\lf|\psi^j_\bz(y)-\psi^j_\bz\lf(\wz{y}\r)\r|
\frac{|\psi^j_\bz(x)|^{1/2}}{[V(y^j_\bz,\dz^j)]^{1/4}}\\
&&\noz\hs\ls e^{\nu\dz2^{k-2}}e^{-\frac{\nu}{2}\dz2^k}
\sum_{j\in\zz}\sum_{\bz\in\scg_j}
\lf|\psi^j_\bz(y)-\psi^j_\bz\lf(\wz{y}\r)\r|
\frac{|\psi^j_\bz(x)|^{1/2}}{[V(y^j_\bz,\dz^j)]^{1/4}}\\
&&\noz\hs\ls\sum_{j\in\zz}\sum_{\bz\in\scg_j}
\lf|\psi^j_\bz(y)-\psi^j_\bz\lf(\wz{y}\r)\r|
\frac{|\psi^j_\bz(x)|^{1/2}}{[V(y^j_\bz,\dz^j)]^{1/4}}=:{\rm J}.
\end{eqnarray}
To estimate $\rm J$, we consider the following two cases.

\textbf{Case (i)} $d(y,\wz{y})\ge\dz^j$.
In this case, from \eqref{b.a},
$d(x,\wz{y})\ge d(x,y)-d(y,\wz{y})\ge\frac12 d(x,y)$,
$d(y,\wz{y})\ge\dz^j$, \eqref{a.b} and some computations similar
to those used in the estimates of $\rm H$,
together with $y^j_\bz:=x^{j+1}_\bz$ for all $\bz\in\scg_j$, we deduce that
\begin{eqnarray*}
{\rm J}&&\ls\sum_{j\in\zz}\sum_{\bz\in\scg_j}
\frac{e^{-\nu\dz^{-j}d(y,\,y^j_\bz)}}{V(y^j_\bz,\dz^j)}
e^{-\frac{\nu}{2}\dz^{-j}d(x,\,y^j_\bz)}
+\sum_{j\in\zz}\sum_{\bz\in\scg_j}
\frac{e^{-\nu\dz^{-j}d(\wz{y},\,y^j_\bz)}}{V(y^j_\bz,\dz^j)}
e^{-\frac{\nu}{2}\dz^{-j}d(x,\,y^j_\bz)}\\
&&\ls\sum_{j\in\zz}\sum_{\bz\in\scg_j}
\frac{e^{-\frac{\nu}{2}\dz^{-j}d(y,\,y^j_\bz)}}{V(y^j_\bz,\dz^j)}
e^{-\frac{\nu}{2}\dz^{-j}d(x,\,y)}
+\sum_{j\in\zz}\sum_{\bz\in\scg_j}
\frac{e^{-\frac{\nu}{2}\dz^{-j}d(\wz{y},\,y^j_\bz)}}{V(y^j_\bz,\dz^j)}
e^{-\frac{\nu}{2}\dz^{-j}d(x,\,\wz{y})}\\
&&\ls\sum_{j\in\zz}\lf[\frac{\dz^j}{d(x,y)}\r]^{\eta}\sum_{\bz\in\scg_j}
\frac{e^{-\frac{\nu}{2}\dz^{-j}d(y,\,y^j_\bz)}}{V(y^j_\bz,\dz^j)}
e^{-\frac{\nu}{4}\dz^{-j}d(x,\,y)}\\
&&\hs+\sum_{j\in\zz}\lf[\frac{\dz^j}{d(x,\wz{y})}\r]^{\eta}\sum_{\bz\in\scg_j}
\frac{e^{-\frac{\nu}{2}\dz^{-j}d(\wz{y},\,y^j_\bz)}}{V(y^j_\bz,\dz^j)}
e^{-\frac{\nu}{4}\dz^{-j}d(x,\,\wz{y})}\\
&&\ls\lf[\frac{d(y,\wz{y})}{d(x,y)}\r]^{\eta}
\sum_{j\in\zz}\sum_{\bz\in\scg_j}
\frac{e^{-\frac{\nu}{2}\dz^{-j}d(y,\,y^j_\bz)}}{V(y^j_\bz,\dz^j)}
e^{-\frac{\nu}{4}\dz^{-j}d(x,\,y)}\\
&&\hs+\lf[\frac{d(y,\wz{y})}{d(x,\wz{y})}\r]^{\eta}\sum_{j\in\zz}
\sum_{\bz\in\scg_j}\frac{e^{-\frac{\nu}{2}
\dz^{-j}d(\wz{y},\,y^j_\bz)}}{V(y^j_\bz,\dz^j)}
e^{-\frac{\nu}{4}\dz^{-j}d(x,\,\wz{y})}\\
&&\ls\lf[\frac{d(y,\wz{y})}{d(x,y)}\r]^{\eta}
\frac1{V(x,y)}+\lf[\frac{d(y,\wz{y})}{d(x,\wz{y})}\r]^{\eta}
\frac1{V(x,\wz{y})}
\ls\lf[\frac{d(y,\wz{y})}{d(x,y)}\r]^{\eta}\frac1{V(x,y)}.
\end{eqnarray*}
This completes the proof of \textbf{Case (i)}.

\textbf{Case (ii)} $d(y,\wz{y})<\dz^j$.
In this case, by \eqref{b.a}, \eqref{b.b},
\eqref{x.t} and Lemma \ref{lb.x}, we further write
\begin{eqnarray*}
{\rm J}&&\ls\sum_{j\in\zz}\sum_{\bz\in\scg_j}
\lf[\frac{d(y,\wz{y})}{\dz^j}\r]^{\eta}
\frac{e^{-\nu\dz^{-j}d(y,\,y^j_\bz)}}{V(y^j_\bz,\dz^j)}
e^{-\frac{\nu}{2}\dz^{-j}d(x,\,y^j_\bz)}\\
&&\ls\sum_{j\in\zz}\sum_{\bz\in\scg_j}
\lf[\frac{d(y,\wz{y})}{\dz^j}\r]^{\eta}e^{-\frac{\nu}{2}\dz^{-j}d(x,\,y)}
\frac{e^{-\frac{\nu}{2}\dz^{-j}d(y,\,y^j_\bz)}}{V(y^j_\bz,\dz^j)}\\
&&\ls\sum_{j\in\zz}\sum_{\bz\in\scg_j}
\lf[\frac{d(y,\wz{y})}{\dz^j}\r]^{\eta}e^{-\frac{\nu}{2}\dz^{-j}d(x,\,y)}
\lf[\frac{\dz^j+d(y,y^j_\bz)}{\dz^j}\r]^n
\frac{e^{-\frac{\nu}{2}\dz^{-j}d(y,\,y^j_\bz)}}{V(y,\dz^j)}\\
&&\ls\sum_{j\in\zz}\lf[\frac{d(y,\wz{y})}{\dz^j}\r]^{\eta}
e^{-\frac{\nu}{2}\dz^{-j}d(x,\,y)}\frac1{V(y,\dz^j)}\sum_{\bz\in\scg_j}
e^{-\frac{\nu}{4}\dz^{-j}d(y,\,y^j_\bz)}\\
&&\ls\sum_{j\in\zz}\lf[\frac{d(y,\wz{y})}{\dz^j}\r]^{\eta}
e^{-\frac{\nu}{2}\dz^{-j}d(x,\,y)}\frac1{V(y,\dz^j)}
e^{-\frac{\nu}{8}\dz^{-j}d(y,\,\scy^j)}\\
&&\sim\sum_{\{j\in\zz:\ \dz^j\ge d(x,\,y)\}}
\lf[\frac{d(y,\wz{y})}{\dz^j}\r]^{\eta}
e^{-\frac{\nu}{2}\dz^{-j}d(x,\,y)}\frac1{V(y,\dz^j)}
e^{-\frac{\nu}{8}\dz^{-j}d(y,\,\scy^j)}
+\sum_{\{j\in\zz:\ \dz^j<d(x,\,y)\}}\cdots\\
&&=:{\rm J}_1+{\rm J}_2,
\end{eqnarray*}
where $\scy^j$ for any $j\in\zz$ is as in \eqref{3.5x}.

Similar to the estimates for ${\rm H}_1$, we conclude that
$$
{\rm J}_1\ls\lf[\frac{d(y,\wz{y})}{d(x,y)}\r]^{\eta}
\sum_{\{j\in\zz:\ \dz^j\ge d(x,\,y)\}}
e^{-\frac{\nu}{2}\dz^{-j}d(x,\,y)}\frac1{V(y,\dz^j)}
e^{-\frac{\nu}{8}\dz^{-j}d(y,\,\scy^j)}
\ls\lf[\frac{d(y,\wz{y})}{d(x,y)}\r]^{\eta}\frac1{V(x,y)}.
$$

Now we turn to estimate ${\rm J}_2$. By \eqref{a.b} and \eqref{a.e},
we see that
\begin{eqnarray*}
{\rm J}_2&&\ls\sum_{\{j\in\zz:\ \dz^j<d(x,\,y)\}}
\lf[\frac{d(y,\wz{y})}{\dz^j}\r]^{\eta}
\frac1{V(y,x)}\lf[\frac{d(x,y)}{\dz^j}\r]^n
e^{-\frac{\nu}{2}\dz^{-j}d(x,\,y)}\\
&&\ls\frac1{V(y,x)}\sum_{\{j\in\zz:\ \dz^j<d(x,\,y)\}}
\lf[\frac{d(y,\wz{y})}{\dz^j}\r]^{\eta}
\lf[\frac{d(x,y)}{\dz^j}\r]^n
\lf[\frac{\dz^j}{d(x,y)}\r]^{n+\eta+1}\\
&&\ls\frac1{V(y,x)}\lf[\frac{d(y,\wz{y})}{d(x,y)}\r]^{\eta}
\sum_{\{j\in\zz:\ \dz^j<d(x,\,y)\}}\frac{\dz^j}{d(x,y)}
\ls\frac1{V(x,y)}\lf[\frac{d(y,\wz{y})}{d(x,y)}\r]^{\eta}.
\end{eqnarray*}

Combining the estimates of ${\rm J}_1$ and ${\rm J}_2$,
we further know that
$$
{\rm J}\ls{\rm J}_1+{\rm J}_2\ls\frac1{V(x,y)}
\lf[\frac{d(y,\wz{y})}{d(x,y)}\r]^{\eta},
$$
which completes the proof \textbf{Case (ii)}.
This, together with \eqref{x.s} and \textbf{Case (i)},
further implies that $K$ satisfies \eqref{d.g}.

Finally, we show that $K$ satisfies \eqref{d.b}.
For all $x,\,\wz{x},\,y\in\cx$ with $0<d(x,\wz{x})\le\frac12 d(x,y)$,
by \eqref{b.x}, \eqref{b.a} and $\az\in\sca^k_{j,\,\bz}$, we find that
\begin{eqnarray*}
&&\lf|K_{k,\,i}(x,y)-K_{k,\,i}(\wz{x},y)\r|\\
&&\hs\le e^{\nu\dz2^{k-2}}\sum_{j\in\zz}\sum_{\bz\in\scg_j}
\lf|s^j_\az(x)\psi^j_\bz(x)-s^j_\az\lf(\wz{x}\r)\psi^j_\bz\lf(\wz{x}\r)\r|
\lf|\psi^j_\bz(y)\r|\\
&&\hs\le e^{\nu\dz2^{k-2}}\sum_{j\in\zz}\sum_{\bz\in\scg_j}
\lf[\lf|s^j_\az(x)\psi^j_\bz(x)\r|^{1/2}
+\lf|s^j_\az\lf(\wz{x}\r)\psi^j_\bz\lf(\wz{x}\r)\r|^{1/2}\r]\\
&&\hs\hs\times\lf|s^j_\az(x)\psi^j_\bz(x)-s^j_\az\lf(\wz{x}\r)
\psi^j_\bz\lf(\wz{x}\r)\r|^{1/2}\lf|\psi^j_\bz(y)\r|\\
&&\hs\ls e^{\nu\dz2^{k-2}}\sum_{j\in\zz}\sum_{\bz\in\scg_j}
\lf[\chi_{B(x^j_\az,\,8\dz^j)}(x)
\frac{e^{-\frac{\nu}{2}\dz^{-j}d(y^j_\bz,\,x)}}
{[V(y^j_\bz,\dz^j)]^{1/4}}
+\chi_{B(x^j_\az,\,8\dz^j)}\lf(\wz{x}\r)
\frac{e^{-\frac{\nu}{2}\dz^{-j}d(y^j_\bz,\,\wz{x})}}
{[V(y^j_\bz,\dz^j)]^{1/4}}\r]\\
&&\hs\hs\times\lf|s^j_\az(x)\psi^j_\bz(x)-s^j_\az\lf(\wz{x}\r)
\psi^j_\bz\lf(\wz{x}\r)\r|^{1/2}\lf|\psi^j_\bz(y)\r|\\
&&\hs\ls\sum_{j\in\zz}\sum_{\bz\in\scg_j}
\lf|s^j_\az(x)\psi^j_\bz(x)-s^j_\az\lf(\wz{x}\r)
\psi^j_\bz\lf(\wz{x}\r)\r|^{1/2}
\frac{|\psi^j_\bz(y)|}{[V(y^j_\bz,\dz^j)]^{1/4}}\\
&&\hs\ls\sum_{j\in\zz}\sum_{\bz\in\scg_j}
\lf|s^j_\az(x)\r|^{1/2}
\lf|\psi^j_\bz(x)-\psi^j_\bz\lf(\wz{x}\r)\r|^{1/2}
\frac{|\psi^j_\bz(y)|}{[V(y^j_\bz,\dz^j)]^{1/4}}\\
&&\hs\hs+\sum_{j\in\zz}\sum_{\bz\in\scg_j}
\lf|s^j_\az(x)-s^j_\az\lf(\wz{x}\r)\r|^{1/2}
\lf|\psi^j_\bz\lf(\wz{x}\r)\r|^{1/2}
\frac{|\psi^j_\bz(y)|}{[V(y^j_\bz,\dz^j)]^{1/4}}\hs=:{\rm A}+{\rm B}.
\end{eqnarray*}

By some arguments similar to those used
in the estimates for $\rm J$, we have
$$
{\rm A}\ls\sum_{j\in\zz}\sum_{\bz\in\scg_j}
\lf|\psi^j_\bz(x)-\psi^j_\bz\lf(\wz{x}\r)\r|^{1/2}
\frac{|\psi^j_\bz(y)|}{[V(y^j_\bz,\dz^j)]^{1/4}}
\ls\frac1{V(x,y)}\lf[\frac{d(x,\wz{x})}{d(x,y)}\r]^{\eta}.
$$

To estimate ${\rm B}$, by \eqref{b.1}, we further write
\begin{eqnarray*}
{\rm B}&&\ls\sum_{j\in\zz}\sum_{\bz\in\scg_j}
\lf[\frac{d(x,\wz{x})}{\dz^j}\r]^{\eta/2}
\lf|\psi^j_\bz(x)\r|^{1/2}
\frac{|\psi^j_\bz(y)|}{[V(y^j_\bz,\dz^j)]^{1/4}}\\
&&\sim\sum_{\{j\in\zz:\ \dz^j\ge d(x,\,y)\}}\sum_{\bz\in\scg_j}
\lf[\frac{d(x,\wz{x})}{\dz^j}\r]^{\eta/2}
\lf|\psi^j_\bz(x)\r|^{1/2}
\frac{|\psi^j_\bz(y)|}{[V(y^j_\bz,\dz^j)]^{1/4}}\\
&&\hs+\sum_{\{j\in\zz:\ d(x,\,\wz{x})\le\dz^j<d(x,\,y)\}}
\sum_{\bz\in\scg_j}\cdots
+\sum_{\{j\in\zz:\ d(x,\,\wz{x})>\dz^j\}}\sum_{\bz\in\scg_j}\cdots\\
&&=:{\rm B}_1+{\rm B}_2+{\rm B}_3.
\end{eqnarray*}

From some arguments similar to those used
in the estimates for $\rm H$, it follows that
\begin{eqnarray*}
{\rm B}_1\ls\lf[\frac{d(x,\wz{x})}{d(x,y)}\r]^{\eta/2}
\sum_{\{j\in\zz:\ \dz^j\ge d(x,\,y)\}}\sum_{\bz\in\scg_j}
\lf|\psi^j_\bz(x)\r|^{1/2}
\frac{|\psi^j_\bz(y)|}{[V(y^j_\bz,\dz^j)]^{1/4}}
\ls\frac1{V(x,y)}\lf[\frac{d(x,\wz{x})}{d(x,y)}\r]^{\eta/2}.
\end{eqnarray*}

Moreover, by some arguments similar to those used
in the estimates for $\rm H$,
\eqref{a.e} and \eqref{a.b}, we conclude that
\begin{eqnarray*}
{\rm B}_2&&\ls\sum_{\{j\in\zz:\ d(x,\,\wz{x})\le\dz^j<d(x,\,y)\}}
\lf[\frac{d(x,\wz{x})}{\dz^j}\r]^{\eta/2}
\frac1{V(\wz{x},\dz^j)}
e^{-\frac{\nu}{8}\dz^{-j} d(\scy^j,\,x)}
e^{-\frac{\nu}{2}\dz^{-j} d(x,\,y)}\\
&&\sim\sum_{\{j\in\zz:\ d(x,\,\wz{x})\le\dz^j<d(x,\,y)\}}
\lf[\frac{d(x,\wz{x})}{\dz^{j}}\r]^{\eta/2}
\frac1{V(x,\dz^j)}
e^{-\frac{\nu}{8}\dz^{-j} d(\scy^j,\,x)}
e^{-\frac{\nu}{2}\dz^{-j} d(x,\,y)}\\
&&\ls\sum_{\{j\in\zz:\ d(x,\,\wz{x})\le\dz^j<d(x,\,y)\}}
\lf[\frac{d(x,\wz{x})}{\dz^{j}}\r]^{\eta/2}
\frac1{V(x,y)}
\lf[\frac{d(x,y)}{\dz^{j}}\r]^{n}
\lf[\frac{\dz^{j}}{d(x,y)}\r]^{n+\frac{\eta}{2}+1}\\
&&\sim\lf[\frac{d(x,\wz{x})}{d(x,y)}\r]^{\eta/2}
\frac1{V(x,y)}\sum_{\{j\in\zz:\ d(x,\,\wz{x})\le\dz^j<d(x,\,y)\}}
\frac{\dz^{j}}{d(x,y)}\\
&&\sim\lf[\frac{d(x,\wz{x})}{d(x,y)}\r]^{\eta/2}
\frac1{V(x,y)}\sum_{j=\lceil\log_{1/\dz}1/d(x,\,y)
\rceil}^\fz\frac{\dz^{j}}{d(x,y)}
\ls\frac1{V(x,y)}\lf[\frac{d(x,\wz{x})}{d(x,y)}\r]^{\eta/2},
\end{eqnarray*}
where $\scy^j$ for any $j\in\zz$ is as in \eqref{3.5x}.

From $0<d(x,\wz{x})\le (1/2)d(x,y)$ and
some arguments similar to those used in the  estimate
of ${\rm B}_2$, we deduce that
\begin{eqnarray*}
{\rm B}_3&&\ls\sum_{\{j\in\zz:\ d(x,\,\wz{x})>\dz^j\}}
\lf[\frac{d(x,\wz{x})}{\dz^j}\r]^{\eta/2}
\frac1{V(\wz{x},\dz^j)}
e^{-\frac{\nu}{2}\dz^{-j} d(\scy^j,\,x)}
e^{-\nu\dz^{-j} d(x,\,y)}\\
&&\ls\sum_{\{j\in\zz:\ \dz^j<d(x,\,y)\}}
\lf[\frac{d(x,\wz{x})}{\dz^{j}}\r]^{\eta/2}
\frac1{V(\wz{x},d(x,y))}
\lf[\frac{d(x,y)}{\dz^{j}}\r]^{n}
\lf[\frac{\dz^{j}}{d(x,y)}\r]^{n+\frac{\eta}{2}+1}\\
&&\ls\lf[\frac{d(x,\wz{x})}{d(x,y)}\r]^{\eta/2}
\frac1{V(x,y)}\sum_{j=\lceil\log_{\frac1{\dz}}\frac1{d(x,\,y)}
\rceil}^\fz\frac{\dz^{j}}{d(x,y)}
\ls\frac1{V(x,y)}\lf[\frac{d(x,\wz{x})}{d(x,y)}\r]^{\eta/2}.
\end{eqnarray*}

Combining the estimates of ${\rm B}_1$, ${\rm B}_2$, ${\rm B}_3$ and ${\rm A}$,
we know that $K_{k,\,i}$ satisfies \eqref{d.b},
which completes the proof of Proposition \ref{pe.f}.
\end{proof}

\begin{remark}\label{r4.5x}
By a slight modification on the proof of Proposition \ref{pe.f},
we can show that, for any $N\in\nn$, $k\in\zz$ and $i\in\{1,\ldots,.m_k\}$,
$K_{k,\,i}$ and $K^N_{k,\,i}$ satisfy \eqref{d.b} and \eqref{d.g}
with any $s\in(0,\eta)$ and $\eta$ as in \eqref{b.1}.
\end{remark}

Now we are ready to establish the following boundedness result for
$\Pi_1$ in \eqref{e.z}, which is an extension of \cite[Lemma 4.2]{bgk}.

\begin{lemma}\label{le.c}
Let $(\cx,d,\mu)$ be a metric measure space of homogeneous type.
Then the bilinear operator $\Pi_1$ in \eqref{e.z},
originally defined for $f,\,g\in\ltw$
with finite wavelet decompositions as in \eqref{4.1x},
can be extended to a bounded bilinear
operator from $\ltw\times\ltw$ into $\hona$.
\end{lemma}

\begin{proof}
Let $f,\,g\in\ltw$ with finite wavelet decompositions as in
\eqref{4.1x}, ${\mathfrak s}^j_\az:=s^j_\az/\nu^j_\az$
for all $(j,\az)\in\sca$, with $\sca$ as in \eqref{b.s}, and
$\nu^j_\az:=\int_{\cx}s^j_\az\,d\mu\sim\mu^j_\az$ [see \eqref{b.x}],
where $\mu^j_\az:=\mu(B(x^j_\az,\dz^j))$ for all $(j,\az)\in\sca$.

We first observe that, for any given $j\in\zz\cap[-M_2,M_2]$,
with $M_2$ as in \eqref{4.1x}, and $\bz\in\scg_j$,
with $\scg_j$ as in \eqref{b.w}, by $\scg_j,\,\sca_j\st\sca_{j+1}$,
$\scg_j\cap\sca_j=\emptyset$ and \eqref{b.z},
we know that $\bz\in\scg_j$ if and
only if $\bz\in\scg_j$ and $d(x^j_{\az},y^j_\bz)\ge\dz^{j+1}$.
Moreover, by the finite wavelet decomposition of $g$, we
see that $Q_jg=0$ for all $j\not\in\zz\cap[-M_2,M_2]$.
From these facts, \eqref{e.z} and
Theorems \ref{tb.a} and \ref{tb.n}, it follows that
\begin{eqnarray}\label{e.c}
\Pi_1(f,g)&&=\sum_{j=-M_2}^{M_2}\lf(P_jf\r)\lf(Q_jg\r)\\
&&\noz=\sum_{j=-M_2}^{M_2}\lf[\sum_{\az\in\sca_j}
\lf(f,\frac{s^j_\az}{\sqrt{\nu^j_\az}}\r)
\frac{s^j_\az}{\sqrt{\nu^j_\az}}\r]
\lf[\sum_{\bz\in\scg_j}\lf(g,\psi^j_\bz\r)\psi^j_\bz\r]\\
&&\noz=\sum_{j=-M_2}^{M_2}\sum_{\az\in\sca_j}
\sum_{\{\bz\in\scg_j:\ d(x^j_{\az},y^j_\bz)\ge\dz^{j+1}\}}
\lf(f,{\mathfrak s}^j_\az\r)\lf(g,\psi^j_\bz\r) s^j_\az\psi^j_\bz
\end{eqnarray}
in $\lon$.

Now we show that
\begin{eqnarray}\label{x.o}
T&&:=\sum_{j=-M_2}^{M_2}\sum_{\bz\in\scg_j}
\sum_{\{\az\in\sca_j:\ d(x^j_{\az},y^j_\bz)\ge\dz^{j+1}\}}
\lf|\lf(f,{\mathfrak s}^j_\az\r)\r|\lf|\lf(g,\psi^j_\bz\r)\r|\\
&&\noz\hs\times\int_{\cx}s^j_\az(x)\lf|\psi^j_\bz(x)\r|\,d\mu(x)<\fz.
\end{eqnarray}
Indeed, from \eqref{b.x}, \eqref{b.a}, the H\"older inequality,
\eqref{3.9x1}, \eqref{a.b} and $\mu^j_\az\sim\nu^j_\az$,
we deduce that
\begin{eqnarray*}
T&&\ls\sum_{j=-M_2}^{M_2}\sum_{\bz\in\scg_j}
\sum_{\{\az\in\sca_j:\ d(x^j_{\az},y^j_\bz)\ge\dz^{j+1}\}}
\lf|\lf(f,{\mathfrak s}^j_\az\r)\r|\lf|\lf(g,\psi^j_\bz\r)\r|
\int_{B(x^j_\az,\,8\dz^j)}\frac{e^{-\nu\dz^{-j}d(x,\,y^j_\bz)}}
{\sqrt{V(y^j_\bz,\dz^j)}}\,d\mu(x)\\
&&\ls\sum_{j=-M_2}^{M_2}\sum_{\bz\in\scg_j}
\sum_{\{\az\in\sca_j:\ d(x^j_{\az},y^j_\bz)\ge\dz^{j+1}\}}
\lf|\lf(f,\frac{s^j_\az}{\sqrt{\nu^j_\az}}\r)\r|\lf|\lf(g,\psi^j_\bz\r)\r|
e^{-\frac{\nu}{2}\dz^{-j}d(x^j_\az,\,y^j_\bz)}\frac1{\sqrt{\nu^j_\az}}\\
&&\hs\times\lf\{\int_{B(x^j_\az,\,8\dz^j)}
\lf[\frac{e^{-\frac{\nu}{2}\dz^{-j}d(x,\,y^j_\bz)}}
{\sqrt{V(y^j_\bz,\dz^j)}}\r]^2\,d\mu(x)\r\}^{1/2}
\lf[V\lf(x^j_\az,8\dz^j\r)\r]^{1/2}\\
&&\ls\sum_{j=-M_2}^{M_2}\sum_{\bz\in\scg_j}
\sum_{\{\az\in\sca_j:\ d(x^j_{\az},y^j_\bz)\ge\dz^{j+1}\}}
\lf|\lf(f,\frac{s^j_\az}{\sqrt{\nu^j_\az}}\r)\r|\lf|\lf(g,\psi^j_\bz\r)\r|
e^{-\nu\dz^{-j}d(x^j_\az,\,y^j_\bz)/2},
\end{eqnarray*}
which, combined with the H\"older inequality again, Lemma \ref{lb.x},
Theorems \ref{tb.n} and \ref{tb.a}, further implies that
\begin{eqnarray*}
T&&\ls\lf\{\sum_{j=-M_2}^{M_2}\sum_{\bz\in\scg_j}
\sum_{\{\az\in\sca_j:\ d(x^j_{\az},y^j_\bz)\ge\dz^{j+1}\}}
\lf|\lf(f,\frac{s^j_\az}{\sqrt{\nu^j_\az}}\r)\r|^2
e^{-\nu\dz^{-j}d(x^j_\az,\,y^j_\bz)/2}\r\}^{1/2}\\
&&\hs\times\lf\{\sum_{j=-M_2}^{M_2}\sum_{\bz\in\scg_j}
\lf|\lf(g,\psi^j_\bz\r)\r|^2 \sum_{\az\in\sca_j}
e^{-\nu\dz^{-j}d(x^j_\az,\,y^j_\bz)/2}\r\}^{1/2}\\
&&\ls\lf\{\sum_{j=-M_2}^{M_2}\sum_{\az\in\sca_j}
\lf|\lf(f,\frac{s^j_\az}{\sqrt{\nu^j_\az}}\r)\r|^2
\sum_{\{\bz\in\scg_j:\ d(x^j_{\az},y^j_\bz)\ge\dz^{j+1}\}}
e^{-\nu\dz^{-j}d(x^j_\az,\,y^j_\bz)/2}\r\}^{1/2}\\
&&\hs\times\lf\{\sum_{j=-M_2}^{M_2}\sum_{\bz\in\scg_j}
\lf|\lf(g,\psi^j_\bz\r)\r|^2\r\}^{1/2}\\
&&\ls\lf\{\sum_{j=-M_2}^{M_2}\sum_{\az\in\sca_j}
\lf|\lf(f,\frac{s^j_\az}{\sqrt{\nu^j_\az}}\r)\r|^2\r\}^{1/2}\|g\|_{\ltw}
\ls M_2^{1/2}\|f\|_{\ltw}\|g\|_{\ltw}<\fz.
\end{eqnarray*}
This shows \eqref{x.o}.

Let $\sca^k_{j,\,\bz}$ and $m_k$ be as in \eqref{4.5y}
and \eqref{4.5z}, respectively, for any
$j\in\nn\cap[-M_2,M_2]$, $\bz\in\scg_j$, with $\scg_j$
as in \eqref{b.w}, and $k\in\zz_+$.
Then, by \eqref{e.c}, \eqref{x.o} and the Fubini theorem, we write
\begin{eqnarray}\label{e.g}
\Pi_1(f,g)&&=\sum_{j=-M_2}^{M_2}\sum_{\bz\in\scg_j}
\sum_{\{\az\in\sca_j:\ d(x^j_{\az},y^j_\bz)\ge\dz^{j+1}\}}
\lf(f,{\mathfrak s}^j_\az\r)\lf(g,\psi^j_\bz\r)  s^j_\az\psi^j_\bz\\
&&\noz=\sum_{k=0}^\fz\sum_{j=-M_2}^{M_2}
\sum_{\bz\in\scg_j}\sum_{\az\in\sca^k_{j,\,\bz}}
\lf(f,{\mathfrak s}^j_\az\r)\lf(g,\psi^j_\bz\r)  s^j_\az\psi^j_\bz\\
&&\noz=\sum_{k=0}^\fz\sum_{i=1}^{m_k}
e^{-\nu\dz 2^{k-2}}
\sum_{j=-M_2}^{M_2}\sum_{\bz\in\scg_j}
\lf(f,{\mathfrak s}^j_{\az^i_{j,\,\bz}}\r)\lf(g,\psi^j_\bz\r)
e^{\nu\dz 2^{k-2}}s^j_{\az^i_{j,\,\bz}}\psi^j_\bz
\end{eqnarray}
in $\lon$.

To further estimate $\Pi_1(f,g)$, we introduce the following operator
$U_{k,\,i}$ for any $k\in\zz_+$ and $i\in\{1,\ldots,m_k\}$.
For any $(j,\bz)\in\scc$, let
\begin{equation}\label{x.d}
U_{k,\,i}\lf(\psi^j_\bz\r):=\wz{\psi}^{k,\,i}_{j,\,\bz}
\end{equation}
with $\wz{\psi}^{k,\,i}_{j,\,\bz}$ as in \eqref{4.4.x}
and $\psi^j_\bz$ as in Theorem \ref{tb.a} with $k$ and $\az$
replaced by $j$ and $\bz$, respectively.
Now we first show that $U_{k,\,i}$ can be extended to a bounded linear
operator on $\ltw$. Indeed,
for any $g\in\ltw$, by Theorem \ref{tb.a}, we know that
$$
g=\sum_{j\in\zz}\sum_{\bz\in\scg_j}
\lf(g,\psi^j_\bz\r) \psi^j_\bz
\quad {\rm in}\quad \ltw.
$$
Fix a collection
$\lf\{\scc_N:\ N\in\nn,\ \scc_N\st\scc\ {\rm and\ \scc_N\ is\ finite}\r\}$
as in \eqref{4.5x}.
For any fixed $k\in\zz$ and $N\in\nn$,
$(j,\bz)\in\scc$ and $i\in\{1,\ldots,m_k\}$, let
$g_N:=\sum_{(j,\,\bz)\in\scc_N}(g,\psi^j_\bz) \psi^j_\bz$
and $U^N_{k,\,i}g$ be as in \eqref{4.4.x}.
Thus, by the finiteness of $\scg_N$,
it is obvious that $U^N_{k,\,i}g\in\ltw$.

By Proposition \ref{pe.e} and its proof,
we see that $\{U^N_{k,\,i}g\}_{N\in\nn}$
is a Cauchy sequence in $\ltw$,
which further implies that there exists $G\in\ltw$
such that
\begin{equation}\label{e.j}
G=\lim_{N\to\fz}U^N_{k,\,i}g\quad {\rm in}\quad \ltw.
\end{equation}
Obviously, for any $N\in\nn$,
$
U_{k,\,i}\lf(g_N\r):=\sum_{(j,\,\bz)\in
\wz{\scc}_N}(g,\psi^j_\bz)  \wz{\psi}^{k,\,i}_{j,\,\bz}
=U_{k,\,i}^Ng.
$
Thus, we define
$$
U_{k,\,i}g:=\lim_{N\to\fz}U_{k,\,i}\lf(g_N\r)
=\lim_{N\to\fz}U^N_{k,\,i}g=G.
$$
Now we show that $U_{k,\,i}g$ is well defined.
To this end, it suffices to prove that the definition of
$U_{k,\,i}g$ is independent of the choice of $\{\scc_N\}_{N\in\nn}$.
Indeed, if there exists another $\wz{\scc}_N$
such that $\wz{U}^N_{k,\,i}g:=\sum_{(j,\,\bz)\in
\wz{\scc}_N}(g,\psi^j_\bz)  \wz{\psi}^{k,\,i}_{j,\,\bz}$ also
satisfies \eqref{e.j}, then let $\mathscr{E}_N:=\scc_N\cup\wz{\scc}_N$ for
any $N\in\nn$. By $\{\mathscr{E}_N\}_{N\in\nn}$ is non-decreasing,
$\bigcup_{N\in\nn}\mathscr{E}_N=\scc$,
Theorem \ref{tb.a} and Proposition \ref{pe.e} and its proof, we see that
\begin{eqnarray*}
&&\lf\|\wz{U}^N_{k,\,i}g-U^N_{k,\,i}g\r\|_{\ltw}\\
&&\hs\le\lf\|\sum_{(j,\,\bz)\in\mathscr{E}_N\bh\scc_N}\lf(g,\psi^j_\bz\r)
\wz{\psi}^{k,\,i}_{j,\,\bz}\r\|_{\ltw}
+\lf\|\sum_{(j,\,\bz)\in\mathscr{E}_N\bh\wz{\scc}_N}\lf(g,\psi^j_\bz\r)
\wz{\psi}^{k,\,i}_{j,\,\bz}\r\|_{\ltw}\\
&&\hs\le\lf[\sum_{(j,\,\bz)\in
\mathscr{E}_N\bh\scc_N}\lf|\lf(g,\psi^j_\bz\r)  \r|^2\r]^{1/2}
+\lf[\sum_{(j,\,\bz)\in
\mathscr{E}_N\bh\wz{\scc}_N}\lf|\lf(g,\psi^j_\bz\r)  \r|^2\r]^{1/2}
\to0\quad {\rm as}\quad N\to\fz.
\end{eqnarray*}

Thus,  $\{U^N_{k,\,i}g\}_{N\in\nn}$ and $\{\wz{U}^N_{k,\,i}g\}_{N\in\nn}$
are equivalent Cauchy sequences in $\ltw$ and, therefore, they
have the same limit in $\ltw$. This implies that
$U_{k,\,i}$ is well defined.
Moreover, from Proposition \ref{pe.e}, it follows
easily that $U_{k,\,i}$ is bounded on $\ltw$.

Then, for each $N\in\nn$, we consider the integral kernels of $U_{k,\,i}^N$
and $U_{k,\,i}$. For each $N\in\nn$,
$k\in\zz$, $i\in\{1,\ldots,m_k\}$, with $m_k$ as in \eqref{4.5z},
and $(x,y)\in(\cx\times\cx)\bh \{(x,x):\ x\in\cx\}$, let
$$
K^N_{k,\,i}(x,y):=\sum_{(j,\,\bz)\in\scc_N}
\wz{\psi}^{k,\,i}_{j,\,\bz}(x)\overline{\psi^j_\bz(y)}
$$
and
$$
K_{k,\,i}(x,y):=\sum_{j\in\zz}\sum_{\bz\in\scg_j}
\wz{\psi}^{k,\,i}_{j,\,\bz}(x)\overline{\psi^j_\bz(y)}.
$$
Now we claim that, for every $N\in\nn$,
$K^N_{k,\,i}$, $K_{k,\,i}$ are the integral kernels of $U_{k,\,i}^N$
and $U_{k,\,i}$, respectively.

Indeed, by Proposition \ref{pe.f}, we conclude that
$K^N_{k,\,i},\,K_{k,\,i}\in L^1_{\loc}
(\{\cx\times\cx\}\bh\{(x,x):\ x\in\cx\})$ are
the Calder\'on-Zygmund kernels with $c_{(K^N_{k,\,i})}$ and $C_{(K^N_{k,\,i})}$
independent of $N$.
It is obvious that, for every $N\in\nn$,
$K^N_{k,\,i}$ is the integral kernel of $U_{k,\,i}^N$.
Furthermore, by these facts and the definition of
$G$, together with the Lebesgue
dominated convergence theorem and the Fubini theorem, we conclude that,
for all $g,\,h\in C_b^{\eta/2}(\cx)$ with
$\supp(g)\cap\supp(h)=\emptyset$,
\begin{eqnarray*}
\langle K_{k,\,i},g\otimes h\rangle&&=\lim_{N\to\fz}
\lf\langle K^N_{k,\,i},g\otimes h\r\rangle\\
&&=\lim_{N\to\fz}\int_\cx\int_\cx
K^N_{k,\,i}(x,y)g(y)h(x)\,d\mu(y)d\mu(x)\\
&&=\lim_{N\to\fz}\lf( U^N_{k,\,i}g,h\r)
=\lf( U_{k,\,i}g,h\r),
\end{eqnarray*}
where $g\otimes h$ denotes the tensor product of $g$ and $h$.
This shows the above claim.

In order to prove that $U_{k,\,i}$ is bounded on $\hona$,
in view of Theorem \ref{te.g}(ii), it remains to show
that, for each $(1,2)$-atom $a$,
\begin{equation}\label{4.13x}
\int_{\cx}U_{k,\,i}a(x)\,d\mu(x)=0,
\end{equation}
observing that, by Theorem \ref{te.g}(i), $U_{k,\,i}a\in\lon$.

To this end, we need to use some arguments
similar to those used in the proof of \cite[p.\,22, Lemma 3]{m97}
as follows.

Let $a$ support in the ball $B_0:=B(x_0,r_0)$
for some $x_0\in\cx$ and $r_0\in(0,\fz)$.
Then we write
$$
\lf\langle U_{k,\,i}a,1\r\rangle=\lf\langle U_{k,\,i}a,
\chi_{2B_0}\r\rangle
+\lf\langle U_{k,\,i}a,\,\chi_{\cx\bh 2B_0}\r\rangle.
$$
By \eqref{e.j}, we know that
\begin{equation}\label{e.l}
\lf\langle U_{k,\,i}a,\chi_{2B_0}\r\rangle
=\lim_{N\to\fz}\lf\langle U^N_{k,\,i}a,\chi_{2B_0}\r\rangle
\end{equation}
and we then show that
\begin{equation}\label{e.k}
\lf\langle U_{k,\,i}a,\chi_{\cx\bh 2B_0}\r\rangle
=\lim_{N\to\fz}\lf\langle U^N_{k,\,i}a,\chi_{\cx\bh 2B_0}\r\rangle.
\end{equation}
To this end, by Proposition \ref{pe.f} and \eqref{d.a},
we first observe that,
for all $x\in\cx\bh 2B_0$ and $y\in B_0$,
$$
\lf|K^N_{k,\,i}(x,y)\r|\ls\frac1{V(x,y)}\ls\frac1{V(x,r_0)}.
$$
By this, the Lebesgue dominated convergence theorem and
the fact that $K^N_{k,\,i}(x,y)$ converge to $K(x,y)$ for all
$x,\,y\in\cx$ with $x\neq y$ (see the proof of Proposition \ref{pe.f}),
we conclude that, for all $x\in\cx\bh 2B_0$,
$$
\lim_{N\to\fz}U^N_{k,\,i}a(x)=U_{k,\,i}a(x).
$$

Moreover, by $\int_\cx a\,d\mu=0$, Proposition \ref{pe.f}
and \eqref{d.g},
we see that, for all $N\in\nn$ and $x\in\cx\bh 2B_0$,
\begin{eqnarray*}
\lf|U^N_{k,\,i}a(x)\r|&&\le\int_{B_0}\lf|K^N_{k,\,i}(x,y)-K^N_{k,\,i}(x,x_0)\r|
|a(y)|\,d\mu(y)\\
&&\noz\ls\int_{B_0}\lf[\frac{d(y,x_0)}{d(x,x_0)}\r]^s
\frac1{V(x,x_0)}|a(y)|\,d\mu(y)\\
&&\noz\ls\lf[\frac{r_0}{d(x,x_0)}\r]^s
\frac1{V(x,x_0)}\|a\|_{\lon}
\ls\lf[\frac{r_0}{d(x,x_0)}\r]^s\frac1{V(x,x_0)},
\end{eqnarray*}
where $s=\eta/2$ and $\eta$ is as in \eqref{b.1}.
From this and the Lebesgue dominated convergence theorem,
we deduce that \eqref{e.k} holds true.

Moreover, for all $j\in\zz$, let $V_j$ and $W_j$
be as in Lemma \ref{le.a}. By
$s^j_{\az^i_{j,\,\bz}}\in V_j$, $\psi^j_\bz\in W_j$
and $V_j\bot W_j$, we see that $(s^j_{\az^i_{j,\,\bz}},\psi^j_\bz)=0$
and hence $\int_\cx\wz{\psi}^{k,\,i}_{j,\,\bz}\,d\mu=0$.
By this,  \eqref{e.l} and the Fubini theorem, we conclude that,
\begin{eqnarray*}
\int_\cx U_{k,\,i}a(x)\,d\mu(x)
&&=\lf\langle U_{k,\,i}a,1\r\rangle
=\lf\langle U_{k,\,i}a,\chi_{2B_0}\r\rangle
+\lf\langle U_{k,\,i}a,\chi_{\cx\bh 2B_0}\r\rangle\\
&&=\lim_{N\to\fz}\lf\langle U^N_{k,\,i}a,\chi_{2B_0}\r\rangle
+\lim_{N\to\fz}\lf\langle U^N_{k,\,i}a,\chi_{\cx\bh 2B_0}\r\rangle
=\lim_{N\to\fz}\lf\langle U^N_{k,\,i}a,1\r\rangle\\
&&=\lim_{N\to\fz}\int_\cx\int_\cx a(y)
\sum_{(j,\,\bz)\in\scc_N}\wz{\psi}^{k,\,i}_{j,\,\bz}(x)\psi^j_\bz(y)
\,d\mu(y)d\mu(x)\\
&&=\lim_{N\to\fz}\sum_{(j,\,\bz)\in\scc_N}\int_\cx
\wz{\psi}^{k,\,i}_{j,\,\bz}(x)\,d\mu(x)\int_\cx a(y)\psi^j_\bz(y)\,d\mu(y)=0.
\end{eqnarray*}
That is, \eqref{4.13x} holds true. Thus, by Theorem \ref{te.g}(ii),
$U_{k,\,i}$ is bounded on $\hona$.

Now we claim that, for any $k\in\zz_+$ and $i\in\{1,\ldots,m_k\}$,
\begin{equation}\label{4.37x}
\sum_{j=-M_2}^{M_2}
\sum_{\bz\in\scg_j}\lf(f,{\mathfrak s}^j_{\az^i_{j,\,\bz}}\r)
\lf(g,\psi^j_\bz\r)\psi^j_\bz\in\hona.
\end{equation}
Indeed, by $\az^i_{j,\,\bz}\in\sca^k_{j,\,\bz}$ with
$\sca^k_{j,\,\bz}$ as in \eqref{4.5y},
\eqref{a.b} and $\mu^j_{\az^i_{j,\,\bz}}\sim
\nu^j_{\az^i_{j,\,\bz}}$, we see that
\begin{equation}\label{4.18x}
V\lf(y^j_\bz,\dz^j\r)\le V\lf(x^j_{\az^i_{j,\,\bz}},2^{k+2}\dz^j\r)
\ls2^{nk}V\lf(x^j_{\az^i_{j,\,\bz}},\dz^j\r)
\sim2^{nk}\mu^j_{\az^i_{j,\,\bz}}\sim2^{nk}\nu^j_{\az^i_{j,\,\bz}}.
\end{equation}

Moreover, by the proof of \cite[Lemma 3.7]{fy1}, we know that,
for any $j\in\zz$ and $\bz\in\scg_j$,
$\frac{\psi^j_\bz}{\sqrt{V(y^j_\bz,\dz^j)}}$ is a
$(1,2,\eta)$-molecule multiplied by a positive constant
independent of $j$ and $\bz$.
Thus, from this, the completion of $\hona$, Theorem \ref{tc.y},
\eqref{4.18x}, the H\"older inequality,
Theorems \ref{tb.a} and \ref{tb.n},
and the fact that, for any $\bz\in\scg_j$,  there
are at most $m_k$ points ($\az^i_{j,\,\bz}$)
in $\sca^k_{j,\,\bz}\st\sca_j$ corresponding
to $\bz$, we conclude that
\begin{eqnarray*}
&&\lf\|\sum_{j=-M_2}^{M_2}
\sum_{\bz\in\scg_j}\lf(f,{\mathfrak s}^j_{\az^i_{j,\,\bz}}\r)
\lf(g,\psi^j_\bz\r)\psi^j_\bz\r\|_{\hona}\\
&&\hs\le\sum_{j=-M_2}^{M_2}
\sum_{\bz\in\scg_j}\lf|\lf(f,{\mathfrak s}^j_{\az^i_{j,\,\bz}}\r)\r|
\lf|\lf(g,\psi^j_\bz\r)\r|\lf\|\psi^j_\bz\r\|_{\hona}\\
&&\hs\ls\sum_{j=-M_2}^{M_2}
\sum_{\bz\in\scg_j}\lf|\lf(f,\frac{s^j_{\az^i_{j,\,\bz}}}
{\sqrt{\nu^j_{\az^i_{j,\,\bz}}}}\r)\r|
\lf|\lf(g,\psi^j_\bz\r)\r| \frac{\sqrt{V(y^j_\bz,\dz^j)}}
{\sqrt{\nu^j_{\az^i_{j,\,\bz}}}}\\
&&\hs\ls2^{nk/2}\sum_{j=-M_2}^{M_2}\lf[\sum_{\bz\in\scg_j}
\lf|\lf(f,\frac{s^j_{\az^i_{j,\,\bz}}}
{\sqrt{\nu^j_{\az^i_{j,\,\bz}}}}\r)\r|^2\r]^{1/2}
\lf[\sum_{\bz\in\scg_j}\lf|\lf(g,\psi^j_\bz\r)\r|^2\r]^{1/2}\\
&&\hs\ls2^{nk/2} m_k^{1/2}\sum_{j=-M_2}^{M_2}\lf[\sum_{\az\in\sca_j}
\lf|\lf(f,\frac{s^j_{\az}}
{\sqrt{\nu^j_{\az}}}\r)\r|^2\r]^{1/2}\|g\|_{\ltw}\\
&&\hs\ls2^{nk/2} m_k^{1/2}M_2\|f\|_{\ltw}\|g\|_{\ltw}<\fz,
\end{eqnarray*}
which completes the proof of the above claim \eqref{4.37x}.

From \eqref{e.g}, \eqref{x.d}, the above claim and the boundedness of
$U_{k,\,i}$ on $\hona$ uniformly in $k$ and $i$, we deduce that
\begin{eqnarray*}
\Pi_1(f,g)&&=\sum_{k=0}^\fz\sum_{i=1}^{m_k}e^{-\nu\dz 2^{k-2}}
\sum_{j=-M_2}^{M_2}\sum_{\bz\in\scg_j}
\lf(f,{\mathfrak s}^j_{\az^i_{j,\,\bz}}\r)\lf(g,\psi^j_\bz\r)
U_{k,\,i}\lf(\psi^j_\bz\r)\\
&&\noz=\sum_{k=0}^\fz\sum_{i=1}^{m_k}
e^{-\nu\dz 2^{k-2}}U_{k,\,i}\lf(\sum_{j=-M_2}^{M_2}
\sum_{\bz\in\scg_j}\lf(f,{\mathfrak s}^j_{\az^i_{j,\,\bz}}\r)
\lf(g,\psi^j_\bz\r)\psi^j_\bz\r)
\end{eqnarray*}
in $\lon$. By the above claim, \eqref{4.37x}, together with
the boundedness of $U_{k,\,i}$ on $\hona$ uniformly in $k$ and $i$,
and Theorem \ref{tc.d}, we conclude that
\begin{eqnarray}\label{d.x}
{\rm L}:=&&\sum_{k=0}^\fz\sum_{i=1}^{m_k}
e^{-\nu\dz 2^{k-2}}\lf\|U_{k,\,i}\lf(\sum_{j=-M_2}^{M_2}
\sum_{\bz\in\scg_j}\lf(f,{\mathfrak s}^j_{\az^i_{j,\,\bz}}\r)
\lf(g,\psi^j_\bz\r)
\psi^j_\bz\r)\r\|_{\hona}\\
\noz\ls&&\sum_{k=0}^\fz\sum_{i=1}^{m_k}
e^{-\nu\dz 2^{k-2}}\lf\|\sum_{j=-M_2}^{M_2}
\sum_{\bz\in\scg_j}\lf(f,{\mathfrak s}^j_{\az^i_{j,\,\bz}}\r)
\lf(g,\psi^j_\bz\r)
\psi^j_\bz\r\|_{\hona}\\
\noz\ls&&\sum_{k=0}^\fz\sum_{i=1}^{m_k}
e^{-\nu\dz 2^{k-2}}\lf\|\lf\{\sum_{(j,\,\gz,\,\bz)\in\sci}
\lf|\lf(f,{\mathfrak s}^j_{\az^i_{j,\,\bz}}\r)
\lf(g,\psi^j_{\gz,\,\bz}\r)  \r|^2\frac{\chi_{Q^j_\gz}}
{\mu(Q^j_\gz)}\r\}^{1/2}\r\|_{\lon}.
\end{eqnarray}

Furthermore, from $\az^i_{j,\,\bz}\in\sca^k_{j,\,\bz}$,
$(j+1,\bz)\le(j,\gz)$ and Remark \ref{rb.d}(i),
it follows that
$$
d\lf(x^j_{\az^i_{j,\,\bz}},x^j_{\gz}\r)
\le d\lf(x^j_{\az^i_{j,\,\bz}},y^j_{\bz}\r)+d\lf(y^j_{\bz},x^j_{\gz}\r)
<2^{k+1}\dz^{j+1}+2\dz^{j+1}\le2^{k+2}\dz^{j+1},
$$
which, combined with Theorem \ref{tb.c}(iv), implies that
$
Q^j_\gz\st B(x^j_\gz,4\dz^j)\st
B(x^j_{\az^i_{j,\,\bz}},2^{k+3}\dz^j)
$.
By these inclusion relations, \eqref{b.x} and \eqref{a.b},
we further conclude that, for all $x\in\cx$,
\begin{eqnarray*}
&&\lf|\lf(f, {\mathfrak s}^j_{\az^i_{j,\,\bz}}\r)\r|\frac{\chi_{Q^j_\gz}(x)}
{\mu(Q^j_\gz)}\\
&&\hs\le\lf|\lf(f, {\mathfrak s}^j_{\az^i_{j,\,\bz}}\r)\r|
\chi_{B(x^j_{\az^i_{j,\,\bz}},\,2^{k+3}\dz^{j})}(x)
\frac{\chi_{Q^j_\gz}(x)}{\mu(Q^j_\gz)}\\
&&\hs\ls\frac{[\wz{C}_{(\cx)}]^k}{V(x^j_{\az^i_{j,\,\bz}},2^{k+3}\dz^{j})}
\int_{B(x^j_{\az^i_{j,\,\bz}},\,2^{k+3}\dz^{j})}|f(y)|\,d\mu(y)
\chi_{B(x^j_{\az^i_{j,\,\bz}},\,2^{k+3}\dz^{j})}(x)
\frac{\chi_{Q^j_\gz}(x)}{\mu(Q^j_\gz)}\\
&&\hs\ls[\wz{C}_{(\cx)}]^k M(f)(x)\frac{\chi_{Q^j_\gz}(x)}{\mu(Q^j_\gz)},
\end{eqnarray*}
which, together with \eqref{d.x},
$m_k:=N_02^{(k+1)G_0}$, the H\"older inequality,
the boundedness of the Hardy-Littlewood maximal function
$M$ on $\ltw$ and Theorem \ref{tb.a}, further implies that
\begin{eqnarray*}
{\rm L}&&\ls\sum_{k=0}^\fz \lf[\wz{C}_{(\cx)}\r]^k m_k
e^{-\nu\dz 2^{k-2}}\lf\|M(f)\lf\{\sum_{(j,\,\gz,\,\bz)\in\sci}
\lf|\lf(g,\psi^j_{\gz,\,\bz}\r)\r|^2\frac{\chi_{Q^j_\gz}}
{\mu(Q^j_\gz)}\r\}^{1/2}\r\|_{\lon}\\
&&\ls\sum_{k=0}^\fz \lf[\wz{C}_{(\cx)}\r]^k m_k
e^{-\nu\dz 2^{k-2}}\|M(f)\|_{\ltw}\lf\{\sum_{(j,\,\gz,\,\bz)\in\sci}
\lf|\lf(g,\psi^j_{\gz,\,\bz}\r)\r|^2\r\}^{1/2}\\
&&\ls\sum_{k=0}^\fz \lf[\wz{C}_{(\cx)}\r]^k m_k
e^{-\nu\dz 2^{k-2}}\|f\|_{\ltw}\|g\|_{\ltw}
\ls\|f\|_{\ltw}\|g\|_{\ltw}.
\end{eqnarray*}
This, combined with the completion of $\hona$,
then implies that $\Pi_1(f,g)\in\hona$ and
$$
\lf\|\Pi_1(f,g)\r\|_{\hona}\ls{\rm L}\ls\|f\|_{\ltw}\|g\|_{\ltw},
$$
which, together with the fact that the functions in
$L^2(\cx)$ with finite wavelet decompositions
as in \eqref{4.1x} are dense in $L^2(\cx)$ as well
as a standard density argument, further
finishes the proof of Lemma \ref{le.c}.
\end{proof}

\begin{remark}\label{re.g}
Using $\Pi_2(f,g)=\Pi_1(g,f)$ for all $f,\,g\in\ltw$
and Lemma \ref{le.c}, we see that $\Pi_2$ as in \eqref{e.z} can also be
extended to a bounded bilinear operator from $\ltw\times\ltw$ into $\hona$.
\end{remark}

\section{Products of Functions in $\hona$ and
$\bmo$}\label{s5}

\hskip\parindent In this section, we prove Theorem \ref{ta.a}.
To this end, we first give the meanings of $P_jg$ and $Q_jg$ for all $j\in\zz$
and $g\in\bmo$.

By \cite[Corollary 11.2]{ah13}, we know that, for any $j\in\zz$,
$\bz\in\scg_j$ and $g\in\bmo$, $\langle g,\psi^j_\bz\rangle$ is well defined
and there exists a positive constant $C$ such that,
for all $g\in\bmo$,
\begin{equation}\label{f.a}
\lf|\lf\langle g,\psi^j_\bz\r\rangle
\r|\le C\|g\|_{\bmo}\sqrt{V(y^j_\bz,\dz^j)}.
\end{equation}
Moreover, let $j\in\zz$, $\az\in\sca_j$
and $g\in\bmo$. By $g\in L^1_{\loc}(\cx)$ and \eqref{b.x},
we know that $|\langle g,s^j_\az\rangle|$ is finite, where
$$
\lf\langle g,s^j_\az\r\rangle:=\int_{\cx}gs^j_\az\,d\mu.
$$
From \eqref{b.x}, the geometrically doubling condition and
Remark \ref{rb.l}(ii), it follows that, for any fixed $x\in\cx$,
only finite items in $\sum_{\az\in\sca_j}\langle g,s^j_\az\rangle s^j_\az$
are non-zero. Thus,
$$
P_jg:=\sum_{\az\in\sca_j}\lf\langle g,s^j_\az\r\rangle s^j_\az
$$
is pointwise well defined. Then we show that $Q_jg$ is
also pointwise well defined. Indeed, by \eqref{f.a},
\eqref{b.a} and Lemma \ref{lb.x}, we see that
there exists a positive constant $C$ such that,
for all $g\in\bmo$ and $x\in\cx$,
\begin{eqnarray*}
\sum_{\bz\in\scg_j}\lf|\lf\langle g,\psi^j_\bz\r\rangle  \r|\lf|\psi^j_\bz(x)\r|
&&\le C\|g\|_{\bmo}\sum_{\bz\in\scg_j}\sqrt{V(y^j_\bz,\dz^j)}
\frac{e^{-\nu\dz^{-j}d(x,\,y^j_\bz)}}{\sqrt{V(y^j_\bz,\dz^j)}}\\
&&\le C\|g\|_{\bmo}e^{-\nu\dz^{-j}d(x,\,\scy^j)/2}
\le C\|g\|_{\bmo}<\fz.
\end{eqnarray*}
Thus,  $Q_jg:=\sum_{\bz\in\scg_j}\langle g,\psi^j_\bz\rangle
\psi^j_\bz$ is pointwise well defined.

Now we recall the following wavelet characterization of BMO$(\cx)$ from
\cite[Theorem 11.4]{ah13}.
A sequence $\{b^j_\bz\}_{j\in\zz,\,\bz\in\scg_j}$,
with $\scg_j$ as in \eqref{b.w} for any $j\in\zz$,
is said to belong to the \emph{Carleson sequence space} Car$(\cx)$ if
$$
\lf\|\lf\{b^j_\bz\r\}_{j\in\zz,\,\bz\in\scg_j}\r\|_{\rm Car(\cx)}
:=\sup_{k\in\zz,\,\az\in\sca_k}\lf[\frac1{\mu(Q_\az^k)}
\sum_{\begin{subarray}{c}
j\in\zz,\,\bz\in\scg_j\\(j+1,\,\bz)\le(k,\,\az)
\end{subarray}}\lf|b^j_\bz\r|^2\r]^{1/2}<\fz.
$$

\begin{theorem}\label{tf.a}
Let $(\cx,d,\mu)$ be a metric measure space of homogeneous type.
Then the space $\bmo/\cc$ ($\rm BMO(\cx)$ functions modulo constants)
and Car$(\cx)$ are isomorphic. This isomorphism is realized via
$b\mapsto\{\langle b,\psi^j_\bz\rangle\}_{j\in\zz,\,\bz\in\scg_j}
=:\{b^j_\bz\}_{j\in\zz,\,\bz\in\scg_j}$ with the inverse given by
$$
\{b^j_\bz\}_{j\in\zz,\,\bz\in\scg_j}\mapsto\sum_{j\in\zz,\,\bz\in\scg_j}
b^j_\bz\lf[\psi^j_\bz-\chi_{\{k\in\zz:\ \dz^k>r_0\}}(j)\psi^j_\bz(x_0)\r]
=:\wz{b},$$
where $\scg_j$ with $j\in\zz$ is as in \eqref{b.w},
the series converges in $L^2_{\loc}(\cx)$ for every $x_0\in\cx$
and $r_0\in(0,\fz)$, and the choices of $x_0$ and $r_0$ only alter
the result by an additive constant.
\end{theorem}

\begin{remark}\label{rf.z}
From the proof of \cite[Theorem 11.4]{ah13}, we deduce that,
if $b\in\bmo$, then $\wz{b}-b=\rm constant$ and hence
$$
b=\wz{b}=\sum_{j\in\zz,\,\bz\in\scg_j}
b^j_\bz\lf[\psi^j_\bz-\chi_{\{k\in\zz:\ \dz^k>r_0\}}(j)\psi^j_\bz(x_0)\r]
$$
converges in $\bmo$ for every $x_0\in\cx$ and $r_0\in(0,\fz)$.
\end{remark}

In order to further investigate the boundedness of $\Pi_1$,
$\Pi_2$ and $\Pi_3$ in \eqref{e.z},
we need to first establish a criterion on the boundedness
of sublinear operators from $\hona$ into a quasi-Banach space.
We first recall the following
notion of the finite atomic Hardy space
(see, for example, \cite{gly1,mm11}).

\begin{definition}\label{da.m}
Let $q\in(1, \fz]$. A function $f\in L^1(\cx)$
is said to be in the {\it finite atomic Hardy space
$H^{1,\,q}_{\at,\,{\rm fin}}(\mathcal{X})$} if there exist $N\in\nn$,
$(1,q)$-atoms (as in Definition \ref{dc.k}) $\{a_j\}_{j=1}^N$ and
numbers $\{\lz_j\}_{j=1}^N\subset\cc$ such that
\begin{equation}\label{c.zx}
f=\sum_{j=1}^N\lz_j a_j.
\end{equation}
Moreover, the norm of $f$ in $H^{1,\,q}_{\at,\,{\rm fin}}(\mathcal{X})$
is defined by setting
$$
\|f\|_{H^{1,\,q}_{\at,\,{\rm fin}}(\mathcal{X})}
:=\inf\lf\{\sum_{j=1}^N|\lz_j|\r\},
$$
where the infimum is taken over all possible finite decompositions of
$f$ as in \eqref{c.zx}.
\end{definition}

\begin{remark}\label{ra.c}
It is obvious that $H^{1,\,q}_{\at,\,{\rm fin}}(\mathcal{X})$
is dense in $H^{1,\,q}_{\at}(\mathcal{X})$ for all $q\in(1,\fz]$.
In what follows, we denote $H^{1,\,2}_{\at,\,{\rm fin}}(\mathcal{X})$
simply by $\hfa$.
\end{remark}

We then recall the following very useful result
from \cite[Theorem 3.2(ii)]{mm11}.

\begin{theorem}\label{tf.z}
Let $(\cx,d,\mu)$ be a metric measure space of homogeneous type.
Then, for each $q\in(1,\fz)$,
$\|\cdot\|_{H^{1,\,q}_{\at,\,{\rm fin}}(\cx)}$
and $\|\cdot\|_{\hona}$ are equivalent norms on
$H^{1,\,q}_{\at,\,{\rm fin}}(\cx)$.
\end{theorem}

We also need to recall the following notions of quasi-Banach spaces and
sublinear operators; see also \cite{gly1, yz08, k14}.

\begin{definition}\label{dd.c}
(i) A \emph{quasi-Banach space} $\cb$ is a vector space
endowed with a \emph{quasi-norm}
$\|\cdot\|_{\cb}$ which is complete, non-negative,
non-degenerate (namely, $\|f\|_{\cb}=0$
if and only if $f=0$), homogeneous, and obeys the quasi-triangle inequality,
namely, there exists a constant $K\in[1,\fz)$ such that, for all $f,\,g\in\cb$,
$\|f+g\|_{\cb}\le K(\|f\|_{\cb}+\|g\|_{\cb})$. Notice that
a quasi-Banach space $\cb$ is called a \emph{Banach space} if $K=1$.

(ii) For any given quasi-Banach space $\cb$ and linear space $\cy$,
an operator $T$ from $\cy$ to $\cb$ is said to be
\emph{$\cb$-sublinear} if there exists a positive constant $C$
such that, for any $\lz,\,\nu\in\cc$ and $f,\,g\in\cy$,
$$
\|T(\lz f+\nu g)\|_{\cb}\le C\lf(|\lz|\|Tf\|_{\cb}
+|\nu|\|Tg\|_{\cb}\r)
$$
and
$$
\|Tf-Tg\|_{\cb}\le C\|T(f-g)\|_{\cb}.
$$
\end{definition}

Obviously, any linear operator from $\cy$ to $\cb$ is $\cb$-sublinear.

Now, using Theorem \ref{tf.z}, we establish a
criterion on the boundedness of sublinear
operators from $\hona$ into a quasi-Banach space $\cb$,
which is a variant of \cite[Theorem 5.9]{gly1};
see also \cite[Theorem 3.5]{k14} and \cite[Theorem 1.1]{yz08}.

\begin{theorem}\label{td.f}
Let $(\cx,d,\mu)$ be a metric measure space of homogeneous type,
$q\in(1,\fz)$ and $\cb$ be a quasi-Banach space.
Suppose that $T:\ H^{1,\,q}_{\at,\,{\rm fin}}(\cx)\to\cb$
is a $\cb$-sublinear operator satisfying that
there exists a positive constant $A$ such that,
for all $f\in H^{1,\,q}_{\at,\,{\rm fin}}(\mathcal{X})$,
\begin{equation}\label{y.i}
\lf\|Tf\r\|_{\cb}\le A\|f\|_{H^{1,\,q}_{\at,\,{\rm fin}}(\mathcal{X})}.
\end{equation}
Then $T$ uniquely extends to a bounded sublinear operator
from $\hona$ into $\cb$.
Moreover, there exists a positive constant $C$ such that, for all $f\in\hona$,
$$
\lf\|Tf\r\|_{\cb}\le CA\|f\|_{\hona}.
$$
\end{theorem}

\begin{proof}
Assume that \eqref{y.i} holds true. For the sake of simplicity, we
only prove Theorem \ref{td.f} for $q=2$,
since the general case for $q\in(1,\fz)$
can be shown similarly.

Let $f\in\hona$. By the density of $\hfa$
in $\hona$ (see Remark \ref{ra.c}), we know that there
exists a Cauchy sequence $\{f_N\}_{N\in\nn}\st\hfa$ such that
$\lim_{N\to\fz}\|f_N-f\|_{\hona}=0$,
which, combined with \eqref{y.i} and Theorem \ref{tf.z},
further implies that
\begin{eqnarray*}
\|T\lf(f_N\r)-T\lf(f_M\r)\|_{\cb}
&&\ls\|T\lf(f_N-f_M\r)\|_{\cb}\ls A\lf\|f_N-f_M\r\|_{\hfa}\\
&&\sim A\lf\|f_N-f_M\r\|_{\hona}\to0, \quad {\rm as}\quad
N,\,M\to\fz.
\end{eqnarray*}
Thus, $\{T(f_N)\}_{N\in\nn}$ is a Cauchy sequence in $\cb$,
which, together with the completion of $\cb$, implies that
there exists $F\in\cb$ such that
$F=\lim_{N\to\fz}T(f_N)$ in $\cb$.
Let $T(f):=F$. From Theorem \ref{tf.z} and \eqref{y.i},
we easily deduce that $T(f)$ is well defined and
\begin{eqnarray*}
\|Tf\|_{\cb}&&\ls\limsup_{N\to\fz}\lf[
\|T(f)-T\lf(f_N\r)\|_{\cb}+\|T\lf(f_N\r)\|_{\cb}\r]
\ls\limsup_{N\to\fz}\|T\lf(f_N\r)\|_{\cb}\\
&&\ls A\limsup_{N\to\fz}\lf\|f_N\r\|_{\hfa}
\sim A\lim_{N\to\fz}\lf\|f_N\r\|_{\hona}\sim A\|f\|_{\hona},
\end{eqnarray*}
which completes the proof of Theorem \ref{td.f}.
\end{proof}

\begin{remark}\label{r5.7x}
Assume that $(\cx,d,\mu)$ is locally compact.
By \cite[Theorem 3.2(i)]{mm11} and some arguments similar
to those used in the proof of Theorem \ref{td.f}, we can also obtain
a criterion on the boundedness of sublinear
operators from $\hona$ into a quasi-Banach space $\cb$
with $(1,q)$-atoms, $q\in(1,\fz)$, replaced by continuous $(1,\fz)$-atoms.
\end{remark}

We now consider the boundedness of $\Pi_3$.
Recall that $\Pi_3$ in \eqref{e.z} is bounded from
$\ltw\times\ltw$ into $\lon$ (see Lemma \ref{le.b}).
To extend $\Pi_3$ to a bounded bilinear operator
from $\hona\times\bmo$ into $\lon$, we first formally write
\begin{equation}\label{x.e}
\Pi_3(a,g):=\sum_{j\in\zz}\lf[\sum_{\bz\in\scg_j}\lf(a,\psi^j_\bz\r)
\psi^j_\bz\r]\lf[\sum_{\gz\in\scg_j}\lf\langle g,\psi^j_\gz\r\rangle
\psi^j_\gz\r]
\end{equation}
for any $(1,2)$-atom $a$ and $g\in\bmo$,
where $\scg_j$ for any $j\in\zz$ is as in \eqref{b.w}.
Observe that, if $a,\,g\in L^2(\cx)$,
then $\Pi_3(a,g)$ in \eqref{x.e} coincides
with $\Pi_3(a,g)$ in \eqref{e.z} with $f$ replaced by $a$ and, in this case,
it is known that $\Pi_3(a,g)\in L^1(\cx)$ (see Lemma \ref{le.b}).

\begin{theorem}\label{tf.b}
Let $(\cx,d,\mu)$ be a metric measure space of homogeneous type.
Then, for any $(1,2)$-atom $a$ and $g\in\bmo$,
$\Pi_3(a,g)$ in \eqref{x.e} belongs to $\lon$ and
can be extended to a bounded bilinear operator
from $\hona\times\bmo$ into $\lon$.
\end{theorem}

\begin{proof}
We first show that, for any $(1,2)$-atom $a$
supported in the ball $B_0:=B(x_0,r_0)$,
for some $x_0\in\cx$ and $r_0\in(0,\fz)$,
and $g\in\bmo$, $\Pi_3(a,g)$ belongs to $\lon$ and
\begin{equation}\label{f.c}
\|\Pi_3(a,g)\|_{\lon}\ls\|g\|_{\bmo},
\end{equation}
where the implicit positive constant is independent of $a$ and $g$.

To this end, let $k_0\in\zz$ satisfy $\dz^{k_0+1}\le r_0<\dz^{k_0}$
and $C_3$ be a sufficiently large positive constant which will
be determined later. We formally write
\begin{eqnarray*}
\Pi_3(a,g)&&=\sum_{j=k_0+1}^\fz\lf[\sum_{\bz\in\scg_j}\lf(a,\psi^j_\bz\r)
\psi^j_\bz\r]\lf[\sum_{\{\gz\in\scg_j:\ y^j_\gz\in C_3B_0\}}
\lf\langle g,\psi^j_\gz\r\rangle\psi^j_\gz\r]\\
&&\hs+\sum_{j=k_0+1}^\fz\lf[\sum_{\bz\in\scg_j}\lf(a,\psi^j_\bz\r)
\psi^j_\bz\r]\lf[\sum_{\{\gz\in\scg_j:\ y^j_\gz\not\in C_3B_0\}}
\lf\langle g,\psi^j_\gz\r\rangle\psi^j_\gz\r]\\
&&\hs+\sum_{j=-\fz}^{k_0}\lf[\sum_{\bz\in\scg_j}\lf(a,\psi^j_\bz\r)
\psi^j_\bz\r]\lf[\sum_{\gz\in\scg_j}
\lf\langle g,\psi^j_\gz\r\rangle\psi^j_\gz\r]\\
&&=:\Pi_3^{(1)}(a,g)+\Pi_3^{(2)}(a,g)
+\Pi_3^{(3)}(a,g).
\end{eqnarray*}

To estimate $\Pi_3^{(1)}(a,g)$, let
\begin{eqnarray*}
g^{(1)}:=\sum_{\{\ell\in\zz:\ \dz^{\ell}\le r_0\}}
\sum_{\{\tz\in\scg_{\ell}:\ y^{\ell}_\tz\in C_3 B_0\}}
\lf\langle g,\psi^{\ell}_\tz\r\rangle  \psi^{\ell}_\tz.
\end{eqnarray*}
We now claim that $g^{(1)}\in\ltw$. Indeed,
let $\sca_{B_0}:=\{\gz\in\sca_{k_0}:\ Q^{k_0}_{\gz}\cap
C_3 B_0\neq\emptyset\}$ with $\sca_{k_0}$ as in \eqref{b.v}.
From Theorem \ref{tb.c}(iii), it follows that
$C_3B_0\st\bigcup_{\gz\in\sca_{B_0}}Q^{k_0}_\gz$.
Thus, by Theorem \ref{tb.c}(iv), we have
$$
d\lf(y^{k_0}_\gz,x_0\r)<(4+C_3)\dz^{k_0},
$$
which, combined with the geometrically doubling condition
and Remark \ref{rb.l}(ii), further implies that $\#\sca_{B_0}$
is bounded uniformly with respect to $k_0$.

By this, Theorem \ref{tb.a}, the Minkowski inequality,
Theorem \ref{tf.a}, $Q^{k_0}_\gz\st B(x_0,(4+C_3)\dz^{k_0})$,
$\dz^{k_0+1}\le r_0$ and \eqref{a.b}, we conclude that
\begin{eqnarray}\label{f.2}
\lf\|g^{(1)}\r\|_{\ltw}&&=\lf\{\sum_{\{\ell\in\zz:\ \dz^{\ell}\le r_0\}}
\sum_{\{\tz\in\scg_{\ell}:\ y^{\ell}_\tz\in C_3 B_0\}}
\lf|\lf\langle g,\psi^{\ell}_\tz\r\rangle  \r|^2\r\}^{1/2}\\
&&\noz\le\sum_{\gz\in\sca_{B_0}}\lf\{\sum_{\{\ell\in\zz:\ \dz^{\ell}\le r_0\}}
\sum_{\{\tz\in\scg_{\ell}:\ y^{\ell}_\tz\in Q^{k_0}_\gz\}}
\lf|\lf\langle g,\psi^{\ell}_\tz\r\rangle  \r|^2\r\}^{1/2}\\
&&\noz=\sum_{\gz\in\sca_{B_0}}\lf\{\sum_{\{\ell\in\zz:\ \dz^{\ell}\le r_0\}}
\sum_{\{\tz\in\scg_{\ell}:\ (\ell+1,\,\tz)\le(k_0,\,\gz)\}}
\lf|\lf\langle g,\psi^{\ell}_\tz\r\rangle  \r|^2\r\}^{1/2}\\
&&\noz\ls\|g\|_{\bmo}\sum_{\gz\in\sca_{B_0}}\sqrt{\mu(Q^{k_0}_\gz)}
\ls\|g\|_{\bmo}\sqrt{\mu(B_0)},
\end{eqnarray}
which shows the above claim.

By this claim and $a\in\ltw$, together with Lemma \ref{le.b}
and \eqref{f.2}, we conclude that
$\Pi_3^{(1)}(a,g)=\Pi_3(a,g^{(1)})$ belongs to $\lon$ and
$$
\lf\|\Pi_3^{(1)}(a,g)\r\|_{\lon}=\lf\|\Pi_3\lf(a,g^{(1)}\r)\r\|_{\lon}
\ls\|a\|_{\ltw}\lf\|g^{(1)}\r\|_{\ltw}
\ls\|g\|_{\bmo}.
$$

To deal with $\Pi_3^{(2)}(a,g)$,
we first estimate $|(a,\psi^j_\bz)|$
for all $(j,\bz)\in\scc$ with $\scc$ as in \eqref{4.4.x1}. By \eqref{b.a},
the H\"older inequality, the size condition of $a$
and \eqref{3.9x1}, we obtain
\begin{eqnarray}\label{f.g}
\lf|\lf(a,\psi^j_\bz\r)\r|&&\le\int_{\cx}|a(x)|
\lf|\psi^j_\bz(x)\r|\,d\mu(x)
\ls\int_{B_0}|a(x)|\frac{e^{-\nu\dz^{-j}d(x,y^j_\bz)}}
{\sqrt{V(y^j_\bz,\dz^j)}}\,d\mu(x)\\
&&\noz\ls e^{-\frac{\nu}{2}\dz^{-j}d(x_0,\,y^j_\bz)}
e^{\frac{\nu}{2}\dz^{-j}r_0}
\int_{B_0}|a(x)|\frac{e^{-\frac{\nu}{2}\dz^{-j}d(x,y^j_\bz)}}
{\sqrt{V(y^j_\bz,\dz^j)}}\,d\mu(x)\\
&&\noz\ls e^{-\frac{\nu}{2}\dz^{-j}d(x_0,\,y^j_\bz)}
e^{\frac{\nu}{2}\dz^{-j}r_0}
\|a\|_{\ltw}\lf\|\frac{e^{-\frac{\nu}{2}\dz^{-j}d(\cdot,y^j_\bz)}}
{\sqrt{V(y^j_\bz,\dz^j)}}\r\|_{\ltw}\\
&&\noz\ls e^{-\frac{\nu}{2}\dz^{-j}d(x_0,\,y^j_\bz)}
e^{\frac{\nu}{2}\dz^{-j}r_0}[\mu(B_0)]^{-1/2}.
\end{eqnarray}

Then we estimate $\int_\cx|\psi^j_\bz(x)\psi^j_\gz(x)|\,d\mu(x)$
for all $j\in\zz\cap[k_0+1,\fz)$ and $\bz,\,\gz\in\scg_j$
with $y^j_\gz\not\in C_3B_0$.
By \eqref{b.a}, the H\"older inequality and \eqref{3.9x1}, we have
\begin{eqnarray}\label{f.d}
&&\int_\cx\lf|\psi^j_\bz(x)\psi^j_\gz(x)\r|\,d\mu(x)\\
&&\noz\hs\ls\int_{\cx}\frac{e^{-\nu\dz^{-j}d(x,\,y^j_\bz)}}
{\sqrt{V(y^j_\bz,\dz^j)}}\frac{e^{-\nu\dz^{-j}d(x,\,y^j_\gz)}}
{\sqrt{V(y^j_\gz,\dz^j)}}\,d\mu(x)\\
&&\noz\hs\ls e^{-\frac{\nu}{2}\dz^{-j}d(y^j_\bz,\,y^j_\gz)}
\int_{\cx}\frac{e^{-\frac{\nu}{2}\dz^{-j}d(x,\,y^j_\bz)}}
{\sqrt{V(y^j_\bz,\dz^j)}}\frac{e^{-\frac{\nu}{2}\dz^{-j}d(x,\,y^j_\gz)}}
{\sqrt{V(y^j_\gz,\dz^j)}}\,d\mu(x)\\
&&\noz\hs\ls e^{-\frac{\nu}{2}\dz^{-j}d(y^j_\bz,\,y^j_\gz)}
\lf\|\frac{e^{-\frac{\nu}{2}\dz^{-j}d(\cdot,\,y^j_\bz)}}
{\sqrt{V(y^j_\bz,\dz^j)}}\r\|_{\ltw}
\lf\|\frac{e^{-\frac{\nu}{2}\dz^{-j}d(\cdot,\,y^j_\gz)}}
{\sqrt{V(y^j_\bz,\gz^j)}}\r\|_{\ltw}
\ls e^{-\frac{\nu}{2}\dz^{-j}d(y^j_\bz,\,y^j_\gz)}.
\end{eqnarray}

Notice that, for any $s\in\zz_+$ and $(j,\gz)\in\scc$, by
Remark \ref{rb.l}(ii) and $r_0<\dz^{k_0}$, we have
$$
\#\lf\{\gz\in\scg_j:\ y^j_\gz\in 2^{s+1}C_3 B_0\bh
2^s C_3 B_0\r\}\ls2^{(s+1)G_0}\lf[\frac{r_0}{\dz^j}\r]^{G_0}
\ls2^{(s+1)G_0}\dz^{(k_0-j)G_0},
$$
which, combined with \eqref{f.g}, \eqref{f.d},
$\dz^{k_0+1}\le r_0<\dz^{k_0}$,
\eqref{f.a}, Lemma \ref{lb.x} and \eqref{a.b}, implies that
\begin{eqnarray*}
&&\sum_{j=k_0+1}^\fz\sum_{\bz\in\scg_j}
\sum_{\{\gz\in\scg_j:\ y^j_\gz\not\in C_3B_0\}}
\lf|\lf(a,\psi^j_\bz\r)\r|\lf|\lf\langle g,\psi^j_\gz\r\rangle\r|
\lf\|\psi^j_\bz\psi^j_\gz\r\|_{\lon}\\
&&\hs\ls[\mu(B_0)]^{-1/2}\sum_{j=k_0+1}^\fz\sum_{\bz\in\scg_j}
\sum_{\{\gz\in\scg_j:\ y^j_\gz\not\in C_3B_0\}}
e^{-\frac{\nu}{2}\dz^{-j}d(x_0,\,y^j_\bz)}
e^{\frac{\nu}{2}\dz^{k_0-j}}\lf|\lf\langle g,\psi^j_\gz\r\rangle\r|
e^{-\frac{\nu}{2}\dz^{-j}d(y^j_\bz,\,y^j_\gz)}\\
&&\hs\ls[\mu(B_0)]^{-1/2}\|g\|_{\bmo}\sum_{j=k_0+1}^\fz
e^{\frac{\nu}{2}\dz^{k_0-j}}
\sum_{\{\gz\in\scg_j:\ y^j_\gz\not\in C_3B_0\}}\sqrt{V(y^j_\gz,\dz^j)}
e^{-\frac{\nu}{4}\dz^{-j}d(x_0,\,y^j_\gz)}\\
&&\hs\hs\times\sum_{\bz\in\scg_j}e^{-\frac{\nu}{4}
\dz^{-j}d(y^j_\bz,\,y^j_\gz)}\\
&&\hs\ls\|g\|_{\bmo}\sum_{j=k_0+1}^\fz
e^{\frac{\nu}{2}\dz^{k_0-j}}
\sum_{\{\gz\in\scg_j:\ y^j_\gz\not\in C_3B_0\}}
e^{-\frac{\nu}{4}\dz^{-j}d(x_0,\,y^j_\gz)}
\lf[\frac{d(x_0,y^j_\gz)+r_0}{r_0}\r]^{n/2}\\
&&\hs\ls\|g\|_{\bmo}\sum_{j=k_0+1}^\fz
e^{\frac{\nu}{2}\dz^{k_0-j}}
\sum_{s=0}^\fz
\sum_{\{\gz\in\scg_j:\ y^j_\gz\in2^{s+1}C_3B_0\bh2^sC_3B_0\}}
e^{-\frac{\nu}{4}\dz^{k_0+1-j}C_32^s}2^{(s+1)n}\\
&&\hs\ls\|g\|_{\bmo}\sum_{j=k_0+1}^\fz
\sum_{s=0}^\fz\dz^{(k_0-j)G_0}2^{(s+1)G_0}
e^{-\frac{\nu}{2}\dz^{k_0-j}(2^{s+1}-1)}2^{(s+1)n}\ls\|g\|_{\bmo},
\end{eqnarray*}
provided that $C_3\dz\ge4$.
By this and the completion of $\lon$, we conclude that
$\Pi_3^{(2)}(a,g)$ belongs to $\lon$ and
$$
\lf\|\Pi_3^{(2)}(a,g)\r\|_{\lon}\ls\|g\|_{\bmo}.
$$

Finally, we consider $\Pi_3^{(3)}(a,g)$.
To this end, we first estimate $|(a,\psi^j_\bz)|$
for all $j\in\zz\cap(-\fz,k_0]$ and $\bz\in\scg_j$
with $\scg_j$ as in \eqref{b.w}.
By $\int_\cx a(x)\,d\mu(x)=0$, $r_0<\dz^{k_0}\le\dz^{j}$
for all $j\le k_0$, and \eqref{b.b}, we obtain
\begin{eqnarray}\label{f.h}
\lf|\lf(a,\psi^j_\bz\r)\r|&&\le\int_{\cx}|a(x)|
\lf|\psi^j_\bz(x)-\psi^j_\bz(x_0)\r|\,d\mu(x)\\
&&\noz\ls\int_{B_0}|a(x)|\lf[\frac{d(x,x_0)}{\dz^j}\r]^{\eta}
\frac{e^{-\nu\dz^{-j}d(x_0,\,y^j_\bz)}}
{\sqrt{V(y^j_\bz,\dz^j)}}\,d\mu(x)\\
&&\noz\ls e^{-\nu\dz^{-j}d(x_0,\,y^j_\bz)}
\dz^{(k_0-j)\eta}\|a\|_{\lon}
\lf[\sqrt{V\lf(y^j_\bz,\dz^j\r)}\r]^{-1/2}\\
&&\noz\ls e^{-\nu\dz^{-j}d(x_0,\,y^j_\bz)}
\dz^{(k_0-j)\eta}
\lf[\sqrt{V\lf(y^j_\bz,\dz^j\r)}\r]^{-1/2}.
\end{eqnarray}

Moreover, observe that, from \eqref{b.w} and \eqref{b.z}, we deduce that,
for given $j\in\zz$ and $\bz\in\scg_j$, $\gz\in\scg_j$
and $\gz\neq\bz$ if and only if $\gz\in\scg_j$
and $d(y^j_\gz,y^j_\bz)=d(x^{j+1}_\gz,x^{j+1}_\bz)\ge\dz^{j+1}$.
By this, \eqref{f.a}, \eqref{f.h} and \eqref{f.d}, we further obtain
\begin{eqnarray*}
{\rm J}&&:=\sum_{j=-\fz}^{k_0}\sum_{\bz\in\scg_j}\sum_{\gz\in\scg_j}
\lf|\lf(a,\psi^j_\bz\r)\r|\lf|\lf\langle g,\psi^j_\gz\r\rangle\r|
\lf\|\psi^j_\bz\psi^j_\gz\r\|_{\lon}\\
&&\ls\|g\|_{\bmo}\sum_{j=-\fz}^{k_0}\dz^{(k_0-j)\eta}
\sum_{\bz\in\scg_j}\sum_{\gz\in\scg_j}e^{-\nu\dz^{-j}d(x_0,\,y^j_\bz)}
\lf[\frac{V(y^j_\gz,\dz^j)}{V(y^j_\bz,\dz^j)}\r]^{1/2}
e^{-\frac{\nu}{2}\dz^{-j}d(y^j_\bz,\,y^j_\gz)}\\
&&\sim\|g\|_{\bmo}\sum_{j=-\fz}^{k_0}\dz^{(k_0-j)\eta}
\sum_{\bz\in\scg_j}e^{-\nu\dz^{-j}d(x_0,\,y^j_\bz)}
+\|g\|_{\bmo}\sum_{j=-\fz}^{k_0}\dz^{(k_0-j)\eta}\\
&&\hs\times\sum_{\bz\in\scg_j}e^{-\nu\dz^{-j}d(x_0,\,y^j_\bz)}
\sum_{\{\gz\in\scg_j:\ d(y^j_\bz,\,y^j_\gz)\ge\dz^{j+1}\}}
\lf[\frac{V(y^j_\gz,\dz^j)}{V(y^j_\bz,\dz^j)}\r]^{1/2}
e^{-\frac{\nu}{2}\dz^{-j}d(y^j_\bz,\,y^j_\gz)}\\
&&=:{\rm J}_1+{\rm J}_2.
\end{eqnarray*}

From Lemma \ref{lb.x}, it follows easily that
\begin{eqnarray*}
{\rm J}_1\ls\|g\|_{\bmo}\sum_{j=-\fz}^{k_0}\dz^{(k_0-j)\eta}
\ls\|g\|_{\bmo}.
\end{eqnarray*}

For ${\rm J}_2$, notice that, for any $s\in\zz_+$,
$(j,\bz)\in\scc$, by Remark \ref{rb.l}(ii), we have
$$
\#\{\gz\in\scg_j:\ 2^s\dz^{j+1}\le
d(y^j_\bz,\,y^j_\gz)<2^{s+1}\dz^{j+1}\}
\ls2^{(s+1)G_0},
$$
which, together with \eqref{a.b} and Lemma \ref{lb.x},
further implies that
\begin{eqnarray*}
{\rm J}_2&&\ls\|g\|_{\bmo}\sum_{j=-\fz}^{k_0}\dz^{(k_0-j)\eta}
\sum_{\bz\in\scg_j}e^{-\nu\dz^{-j}d(x_0,\,y^j_\bz)}
\sum_{\{\gz\in\scg_j:\ d(y^j_\bz,\,y^j_\gz)\ge\dz^{j+1}\}}
\lf[\frac{d(y^j_\bz,y^j_\gz)+\dz^j}{\dz^j}\r]^{n/2}\\
&&\hs\times e^{-\frac{\nu}{2}\dz^{-j}d(y^j_\bz,\,y^j_\gz)}\\
&&\ls\|g\|_{\bmo}\sum_{j=-\fz}^{k_0}\dz^{(k_0-j)\eta}
\sum_{\bz\in\scg_j}e^{-\nu\dz^{-j}d(x_0,\,y^j_\bz)}\\
&&\hs\times\sum_{s=0}^\fz\sum_{\{\gz\in\scg_j:\ 2^s\dz^{j+1}\le
d(y^j_\bz,\,y^j_\gz)<2^{s+1}\dz^{j+1}\}}
2^{(s+1)n} e^{-\nu\dz2^{s-1}}\\
&&\ls\|g\|_{\bmo}\sum_{j=-\fz}^{k_0}\dz^{(k_0-j)\eta}
\sum_{s=0}^\fz2^{(s+1)G_0}2^{(s+1)n} e^{-\nu\dz2^{s-1}}
\ls\|g\|_{\bmo}.
\end{eqnarray*}
This, combined with the estimate of ${\rm J}_1$, implies that
$
{\rm J}\ls{\rm J}_1+{\rm J}_2\ls\|g\|_{\bmo}
$.
Thus, by the completion of $\lon$, we see that
$\Pi_3^{(3)}(a,g)\in\lon$ and
$$
\lf\|\Pi_3^{(3)}(a,g)\r\|_{\lon}\le{\rm J}\ls\|g\|_{\bmo}.
$$
By this and the estimates of $\Pi_3^{(1)}(a,g)$
and $\Pi_3^{(2)}(a,g)$, we conclude that $\Pi_3(a,g)$
belongs to $\lon$ and \eqref{f.c} holds true.

Moreover, we claim that, for any $(f,g)\in\hfa\times\bmo$,
\begin{equation}\label{y.g}
\lf\|\Pi(f,g)\r\|_{\lon}\ls\|f\|_{\hfa}\|g\|_{\bmo}.
\end{equation}
Indeed, for any $f\in\hfa$, there exists
a finite sequence $\{a_j\}_{j=1}^N$ ($N\in\nn$) of
$(1,2)$-atoms and $\{\lz_j\}_{j=1}^N\st\cc$ such that
\begin{equation}\label{y.h}
f=\sum_{i=1}^N\lz_j a_i \quad{\rm  and}\quad
\sum_{j=1}^N|\lz_i|\ls\|f\|_{\hfa}.
\end{equation}
Thus, from \eqref{f.c} and \eqref{y.h}, it follows that
\begin{eqnarray*}
\|\Pi(f,g)\|_{\lon}&&=\lf\|\sum_{j=1}^N\lz_j\Pi\lf(a_j,g\r)\r\|_{\lon}
\le\sum_{j=1}^N\lf|\lz_j\r|\lf\|\Pi\lf(a_j,g\r)\r\|_{\lon}\\
&&\ls\sum_{j=1}^N\lf|\lz_j\r|\|g\|_{\bmo}\ls\|f\|_{\hfa}\|g\|_{\bmo},
\end{eqnarray*}
which shows the claim \eqref{y.g}.
By the above claim and Theorem \ref{td.f} with $T(\cdot):=\Pi_3(\cdot,g)$
and $A\sim\|g\|_{\bmo}$ for any fixed $g\in\bmo$,
we see that $\Pi_3$ can be extended to a bounded bilinear operator
from $\hona\times\bmo$ into $\lon$, which completes
the proof of Theorem \ref{tf.b}.
\end{proof}

We then consider the boundedness of $\Pi_1$.
Recall that $\Pi_1$ in \eqref{e.z} is bounded from
$\ltw\times\ltw$ into $\hona$ (see Lemma \ref{le.c}).
To extend $\Pi_1$ to a bounded bilinear operator
from $\hona\times\bmo$ into $\hona$, we first formally write
\begin{equation}\label{x.f}
\Pi_1(a,g):=\sum_{j\in\zz}\lf[\sum_{\az\in\sca_j}\lf(a,{\mathfrak s}^j_\az\r)
s^j_\az\r]\lf[\sum_{\bz\in\scg_j}\lf\langle g,\psi^j_\bz\r\rangle
\psi^j_\bz\r]
\end{equation}
for any $(1,2)$-atom $a$ and $g\in\bmo$, where $\sca_j$ and $\scg_j$
for any $j\in\zz$ are as in, respectively, \eqref{b.v}
and \eqref{b.w}, ${\mathfrak s}^j_\az:=s^j_\az/\nu^j_\az$
and $\nu^j_\az:=\int_{\cx}s^j_\az\,d\mu$.
Notice that, if $a,\,g\in L^2(\cx)$,
then $\Pi_1(a,g)$ in \eqref{x.f} coincides
with $\Pi_1(a,g)$ in \eqref{e.z} with $f$ replaced by $a$ and, in this case,
it is known that $\Pi_1(a,g)\in\hona$ (see Lemma \ref{le.c}).

\begin{theorem}\label{tf.c}
Let $(\cx,d,\mu)$ be a metric measure space of homogeneous type. Then,
for any $(1,2)$-atom $a$ and $g\in\bmo$,
$\Pi_1(a,g)$ in \eqref{x.f} belongs to $\hona$ and
can be extended to a bounded bilinear operator
from $\hona\times\bmo$ into $\hona$.
\end{theorem}

\begin{proof}
We first prove that, for any $(1,2)$-atom $a$ supported
in the ball $B_0:=B(x_0,r_0)$,
for some $x_0\in\cx$ and $r_0\in(0,\fz)$, and $g\in\bmo$,
$\Pi_1(a,g)$ belongs to $\hona$ and
\begin{equation}\label{f.e}
\|\Pi_1(a,g)\|_{\hona}\ls\|g\|_{\bmo},
\end{equation}
where the implicit positive constant is independent of $a$ and $g$.

To this end, let $k_0\in\zz$ satisfy $\dz^{k_0+1}\le r_0<\dz^{k_0}$
and $C_4$ be a sufficiently large positive constant which will
be determined later. We then formally write
\begin{eqnarray*}
\Pi_1(a,g)&&=\sum_{j=k_0+1}^\fz\lf[\sum_{\az\in\sca_j}
\lf(a,{\mathfrak s}^j_\az\r)
s^j_\az\r]\lf[\sum_{\{\bz\in\scg_j:\ y^j_\bz\in C_4B_0\}}
\lf\langle g,\psi^j_\bz\r\rangle\psi^j_\bz\r]\\
&&\hs+\sum_{j=k_0+1}^\fz\lf[\sum_{\az\in\sca_j}\lf(a,{\mathfrak s}^j_\az\r)
s^j_\az\r]\lf[\sum_{\{\bz\in\scg_j:\ y^j_\bz\not\in C_4B_0\}}
\lf\langle g,\psi^j_\bz\r\rangle\psi^j_\bz\r]\\
&&\hs+\sum_{j=-\fz}^{k_0}\lf[\sum_{\az\in\sca_j}\lf(a,{\mathfrak s}^j_\az\r)
s^j_\az\r]\lf[\sum_{\bz\in\scg_j}
\lf\langle g,\psi^j_\bz\r\rangle\psi^j_\bz\r]\\
&&=:\Pi_1^{(1)}(a,g)+\Pi_1^{(2)}(a,g)
+\Pi_1^{(3)}(a,g).
\end{eqnarray*}

Let
\begin{eqnarray*}
g_1:=\sum_{\{\ell\in\zz:\ \dz^{\ell}\le r_0\}}
\sum_{\{\tz\in\scg_{\ell}:\ y^{\ell}_\tz\in C_4 B_0\}}
\lf\langle g,\psi^{\ell}_\tz\r\rangle\psi^{\ell}_\tz.
\end{eqnarray*}
By \eqref{f.2} with $C_3$ replaced by $C_4$, we know that
$g_1\in\ltw$ and
$$
\lf\|g_1\r\|_{\ltw}\ls\|g\|_{\bmo}\sqrt{\mu(B_0)}.
$$
By this and $a\in\ltw$, together with Lemma \ref{le.c},
we conclude that $\Pi_1^{(1)}(a,g)=\Pi_1(a,g_1)$ belongs to $\hona$, which,
together with Lemma \ref{le.c} and an argument similar to that used in the
estimate for $\Pi_3^{(1)}(a,g)$ in the proof of
Theorem \ref{tf.b}, implies that
$$
\lf\|\Pi_1^{(1)}(a,g)\r\|_{\hona}
=\lf\|\Pi_1\lf(a,g_1\r)\r\|_{\hona}\ls\|g\|_{\bmo}.
$$

Observe that, for any $j\in\zz\cap[k_0+1,\fz)$ and $\az\in\sca_j$
with $\sca_j$ as in \eqref{b.v},
if $B(x^j_\az,8\dz^j)\cap B(x_0,r_0)\neq\emptyset$, then
$
x^j_\az\in B(x_0,r_0+8\dz^j)\st 9B_0
$.
Thus, we have
$$
\Pi_1^{(2)}(a,g)=\sum_{j=k_0+1}^\fz
\lf[\sum_{\{\az\in\sca_j:\ x^j_\az\in9B_0\}}
\lf(a,{\mathfrak s}^j_\az\r)s^j_\az\r]
\lf[\sum_{\{\bz\in\scg_j:\ y^j_\bz\not\in C_4B_0\}}
\lf\langle g,\psi^j_\bz\r\rangle \psi^j_\bz\r].
$$

Now we claim that, for all $j\in\zz$, $\az\in\sca_j$, and $\bz\in\scg_j$
with $\scg_j$ as in \eqref{b.w},
\begin{equation}\label{5.7z}
a^j_{\az,\,\bz}:=e^{\frac{\nu}{2}
\dz^{-j}d(x^j_\az,\,y^j_\bz)}\frac{s^j_\az\psi^j_\bz}
{\sqrt{V(x^j_\az,10\dz^j)}}\quad {\rm is\ a}\ (1,2)-{\rm atom},
\end{equation}
 multiplied by a positive constant, supported in
$B(x^j_\az,10\dz^j)$.
Indeed, from $s^j_\az\in V_j$, $\psi^j_\bz\in W_j$ and
$W_j\bot V_j$ with $V_j$ and $W_j$ for any $j\in\zz$
as in Lemma \ref{le.a}, it follows that
$\int_\cx s^j_\az(x)\psi^j_\bz(x)\,d\mu(x)=0$. Hence,
$\int_\cx a^j_{\az,\,\bz}\,d\mu=0$. Meanwhile, by \eqref{b.x},
we see that $\supp(s^j_\az\psi^j_\bz)\st B(x^j_\az,10\dz^j)$.
Moreover, from \eqref{b.x}, \eqref{b.a} and \eqref{3.9x1}, we deduce that
$$
\lf\|s^j_\az\psi^j_\bz\r\|^2_{\ltw}
\ls\int_{B(x^j_\az,\,\dz^j)}\frac{e^{-2\nu\dz^{-j}d(y^j_\bz,\,x)}}
{V(y^j_\bz,\dz^j)}\,d\mu(x)
\ls e^{-\nu\dz^{-j}d(y^j_\bz,\,x^j_\az)}.
$$
Thus,
$$
\|a^j_{\az,\,\bz}\|_{\ltw}\ls\lf[V\lf(x^j_\az,10\dz^j\r)\r]^{-1/2},
$$
which shows the above claim.

Observe that, for any $j\in\zz\cap[k_0+1,\fz)$, $\az\in\sca_j$,
$\bz\in\scg_j$, $x^j_\az\in 9B_0$ and $y^j_\bz\not\in C_4B_0$, we have
$d(y_\bz^j,x_0)\ge C_4r_0\ge2d(x^j_\az,x_0)$ provided that
$C_4\ge18$. Moreover, from Remark \ref{rb.l}(ii)
and $r_0<\dz^{k_0}$, it follows that, for any $j\in\zz$ and
$t\in\zz_+$,
\begin{equation}\label{x.a}
\#\lf\{\az\in\sca_j:\ x^j_\az\in9B_0\r\}
\ls\lf[\frac{r_0}{\dz^j}\r]^{G_0}\ls\dz^{(k_0-j)G_0}
\end{equation}
and
\begin{equation}\label{x.b}
\#\{\bz\in\scg_j:\ y^j_\bz\in2^{t+1}C_4B_0\bh2^tC_4B_0\}
\ls2^{tG_0}\lf[\frac{r_0}{\dz^j}\r]^{G_0}\ls2^{tG_0}\dz^{(k_0-j)G_0}.
\end{equation}
By the above claim \eqref{5.7z}, \eqref{f.a},
$\nu^j_\az\sim\mu^j_\az:=V(x^j_\az,\dz^j)$,
\eqref{b.x} and \eqref{a.b}, we conclude that
\begin{eqnarray*}
{\rm A}&&:=\sum_{j=k_0+1}^\fz\sum_{\{\az\in\sca_j:\ x^j_\az\in9B_0\}}
\sum_{\{\bz\in\scg_j:\ y^j_\bz\not\in C_4B_0\}}
\lf|\lf(a,{\mathfrak s}^j_\az\r)\r|\lf|\lf\langle g,\psi^j_\bz\r\rangle\r|
\lf\|s^j_\az\psi^j_\bz\r\|_{\hona}\\
&&\ls\sum_{j=k_0+1}^\fz\sum_{\{\az\in\sca_j:\ x^j_\az\in9B_0\}}
\sum_{\{\bz\in\scg_j:\ y^j_\bz\not\in C_4B_0\}}
\lf|\lf(a,{\mathfrak s}^j_\az\r)\r|\lf|\lf\langle g,\psi^j_\bz\r\rangle\r|
e^{-\frac{\nu}{2}\dz^{-j}d(x^j_\az,\,y^j_\bz)}
\sqrt{V(x^j_\az,10\dz^j)}\\
&&\ls\|g\|_{\bmo}\sum_{j=k_0+1}^\fz
\sum_{\{\az\in\sca_j:\ x^j_\az\in9B_0\}}
\sum_{\{\bz\in\scg_j:\ y^j_\bz\not\in C_4B_0\}}
\|a\|_{\lon}\lf[\frac{V(y^j_\bz,\dz^j)}{V(x^j_\az,\dz^j)}\r]^{1/2}\\
&&\hs\times e^{-\frac{\nu}{2}\dz^{-j}d(y^j_\bz,\,x^j_\az)}\\
&&\ls\|g\|_{\bmo}\sum_{j=k_0+1}^\fz
\sum_{\{\az\in\sca_j:\ x^j_\az\in9B_0\}}
\sum_{\{\bz\in\scg_j:\ y^j_\bz\not\in C_4B_0\}}
\lf[\frac{d(x^j_\az,y^j_\bz)+\dz^j}{\dz^j}\r]^{n/2}
e^{-\frac{\nu}{2}\dz^{-j}d(y^j_\bz,\,x^j_\az)},
\end{eqnarray*}
which, together with $d(y_\bz^j,x_0)\ge2d(x^j_\az,x_0)$,
\eqref{x.a}, \eqref{x.b} and $\dz^{k_0+1}\le r_0<\dz^{k_0}$,
further implies that
\begin{eqnarray*}
{\rm A}&&\ls\|g\|_{\bmo}\sum_{j=k_0+1}^\fz
\sum_{\{\az\in\sca_j:\ x^j_\az\in9B_0\}}
\sum_{\{\bz\in\scg_j:\ y^j_\bz\not\in C_4B_0\}}
e^{-\frac{\nu}{4}\dz^{-j}d(y^j_\bz,\,x^j_\az)}\\
&&\ls\|g\|_{\bmo}\sum_{j=k_0+1}^\fz
\sum_{\{\bz\in\scg_j:\ y^j_\bz\not\in C_4B_0\}}
e^{-\frac{\nu}{8}\dz^{-j}d(y^j_\bz,\,x_0)}
\dz^{(k_0-j)G_0}\\
&&\ls\|g\|_{\bmo}\sum_{j=k_0+1}^\fz
\sum_{t=0}^\fz\sum_{\{\bz\in\scg_j:\ y^j_\bz\in2^{t+1}C_4B_0\bh2^tC_4B_0\}}
e^{-\frac{\nu}{8}C_4\dz^{k_0-j+1}2^t}\dz^{(k_0-j)G_0}\\
&&\ls\|g\|_{\bmo}\sum_{j=k_0+1}^\fz\sum_{t=0}^\fz
2^{-tM_0}\dz^{(j-k_0)M_0}\dz^{(k_0-j)G_0}2^{tG_0}\dz^{(k_0-j)G_0}\ls\|g\|_{\bmo},
\end{eqnarray*}
where $M_0$ and $C_4$ are sufficiently large positive constants such that
$M_0>2G_0$, with $G_0$ as in Remark \ref{rb.l}(ii),
and $C_4\ge18$, respectively. Thus, by the completion of $\hona$,
we conclude that $\Pi_1^{(2)}(a,g)\in\hona$ and
$$
\lf\|\Pi_1^{(2)}(a,g)\r\|_{\hona}
\le{\rm A}\ls\|g\|_{\bmo}.
$$

Finally, we consider $\Pi_1^{(3)}(a,g)$.
Observe that, for any $j\in\zz\cap(-\fz,k_0]$ and $\az\in\sca_j$,
if $B(x_0,r_0)\cap B(x^j_\az,8\dz^j)\neq\emptyset$,
then $x^j_\az\in B(x_0,9\dz^j)$.
By this, we further formally write
$$
\Pi_1^{(3)}(a,g)=\sum_{j=-\fz}^{k_0}
\lf[\sum_{\{\az\in\sca_j:\ x^j_\az\in B(x_0,\,9\dz^j)\}}
\lf(a,{\mathfrak s}^j_\az\r)s^j_\az\r]\lf[\sum_{\bz\in\scg_j}
\lf\langle g,\psi^j_\bz\r\rangle\psi^j_\bz\r].
$$

We first estimate $|(a,s^j_\az)|$ for all $j\in\zz\cap(-\fz,k_0]$
and $\az\in\sca_j$ with $x^j_\az\in B(x_0,\,9\dz^j)$.
By $\int_\cx a(x)\,d\mu(x)=0$, $r_0<\dz^{k_0}\le\dz^{j}$
for all $j\le k_0$, and \eqref{b.1}, we obtain
\begin{eqnarray}\label{x.q}
\lf|\lf(a,s^j_\az\r)\r|&&\le\int_{B_0}|a(x)|
\lf|s^j_\az(x)-s^j_\az(x_0)\r|\,d\mu(x)\\
&&\noz\ls\int_{B_0}|a(x)|\lf[\frac{d(x,x_0)}{\dz^j}\r]^{\eta}\,d\mu(x)
\ls\dz^{(k_0-j)\eta}\|a\|_{\lon}
\ls\dz^{(k_0-j)\eta}.
\end{eqnarray}

Moreover, for any $j\in\zz$, by
Remark \ref{rb.l}(ii), we have
$$
\#\{\az\in\sca_j:\ x^j_\az\in B(x_0,\,9\dz^j)\}\ls1.
$$
From this, \eqref{x.q}, the above claim \eqref{5.7z}, \eqref{a.b},
\eqref{f.a}, $\nu^j_\az\sim\mu^j_\az:=V(x^j_\az,\dz^j)$ and
Lemma \ref{lb.x}, it follows that
\begin{eqnarray*}
&&\sum_{j=-\fz}^{k_0}\sum_{\{\az\in\sca_j:\ x^j_\az\in B(x_0,\,9\dz^j)\}}
\sum_{\bz\in\scg_j}\lf|\lf(a,{\mathfrak s}^j_\az\r)\r|
\lf|\lf\langle g,\psi^j_\bz\r\rangle  \r|
\lf\|s^j_\az\psi^j_\bz\r\|_{\hona}\\
&&\hs\ls\sum_{j=-\fz}^{k_0}\sum_{\{\az\in\sca_j:\ x^j_\az\in B(x_0,\,9\dz^j)\}}
\sum_{\bz\in\scg_j}
\lf|\lf(a,{\mathfrak s}^j_\az\r)\r|\lf|\lf\langle g,\psi^j_\bz\r\rangle  \r|
e^{-\frac{\nu}{2}\dz^{-j}d(y^j_\bz,\,y^j_\gz)}
\sqrt{V(x^j_\az,10\dz^j)}\\
&&\hs\ls\|g\|_{\bmo}\sum_{j=-\fz}^{k_0}
\dz^{(k_0-j)\eta}\sum_{\{\az\in\sca_j:\ x^j_\az\in B(x_0,\,9\dz^j)\}}
\sum_{\{\bz\in\scg_j:\ d(x^j_\az,\,y^j_\bz)\ge\dz^{j+1}\}}
e^{-\frac{\nu}{2}\dz^{-j}d(x^j_\az,\,y^j_\bz)}\\
&&\hs\hs\times\lf[\frac{V(y^j_\bz,\dz^j)}{V(x^j_\az,\dz^j)}\r]^{1/2}\\
&&\hs\ls\|g\|_{\bmo}\sum_{j=-\fz}^{k_0}
\dz^{(k_0-j)\eta}\sum_{\{\az\in\sca_j:\ x^j_\az\in B(x_0,\,9\dz^j)\}}
\sum_{\{\bz\in\scg_j:\ d(x^j_\az,\,y^j_\bz)\ge\dz^{j+1}\}}
e^{-\frac{\nu}{2}\dz^{-j}d(x^j_\az,\,y^j_\bz)}\\
&&\hs\hs\times\lf[\frac{d(y^j_\bz,x^j_\az)+\dz^j}{\dz^j}\r]^{1/2}\\
&&\hs\ls\|g\|_{\bmo}\sum_{j=-\fz}^{k_0}
\dz^{(k_0-j)\eta}\sum_{\{\az\in\sca_j:\ x^j_\az\in B(x_0,\,9\dz^j)\}}
\sum_{\{\bz\in\scg_j:\ d(x^j_\az,\,y^j_\bz)\ge\dz^{j+1}\}}
e^{-\frac{\nu}{4}\dz^{-j}d(x^j_\az,\,y^j_\bz)}\\
&&\hs\ls\|g\|_{\bmo}\sum_{j=-\fz}^{k_0}
\dz^{(k_0-j)\eta}\ls\|g\|_{\bmo},
\end{eqnarray*}
which, combined with the completion of $\hona$, further
implies that $\Pi_1^{(3)}(a,g)\in\hona$ and
$$
\lf\|\Pi_1^{(3)}(a,g)\r\|_{\hona}\ls\|g\|_{\bmo}.
$$

By this and the estimates of $\Pi_1^{(1)}(a,g)$
and $\Pi_1^{(2)}(a,g)$, we conclude that
$\Pi_1(a,g)$ in \eqref{x.f} belongs to $\hona$ and \eqref{f.e} holds true,
which, together with Theorem \ref{td.f} and an argument
similar to that used in the proof of Theorem \ref{tf.b},
further implies that $\Pi_1$ can be extended to a bounded bilinear
operator from $\hona\times\bmo$ into $\hona$.
This finishes the proof of Theorem \ref{tf.c}.
\end{proof}

Before considering the boundedness of $\Pi_2$,
we first recall some useful results from \cite{ky},
which are valid without resorting to the reverse doubling condition
after some careful examinations, the details being omitted.

\begin{lemma}\label{lf.t}(\cite[Proposition 3.1]{ky})
Let $(\cx,d,\mu)$ be a metric measure space of homogeneous type,
$\bz\in(0,1]$ and $\gz\in(0,\fz]$. Then, for all $h\in\cg(\bz,\gz)$,
there exists a positive constant $C$, independent of $h$,
such that, for any $g\in\bmo$,
$$
\lf\|hg\r\|_{\bmo}\le C\frac1{V_1(x_1)}\|h\|_{\cg(\bz,\gz)}
\|g\|_{\bmop},
$$
here and hereafter, for a fixed $x_1\in\cx$ and all $g\in\bmo$,
$$
\|g\|_{\bmop}:=\|f\|_{\bmo}+\frac1{V_1(x_1)}\int_{B(x_1,\,1)}
|f(x)|\,d\mu(x).
$$
\end{lemma}

\begin{lemma}\label{lf.v}(\cite[Lemma 3.2]{ky})
Let $(\cx,d,\mu)$ be a metric measure space of homogeneous type
and $q\in(1,\fz]$. Then there exists a positive constant $C$
such that, for any $(1,q)$-atom $a$ supported in some ball $B$
and $g\in\bmo$,
$$
\lf\|\lf[g-m_B(g)\r]\cm(a)\r\|_{\lon}\le C\|g\|_{\bmo},
$$
where $\cm$ is as in \eqref{a.u} and
$m_B(g):=\frac1{\mu(B)}\int_B g\,d\mu$.
\end{lemma}

\begin{lemma}\label{lf.u}(\cite[Proposition 3.2(ii)]{ky})
Let $(\cx,d,\mu)$ be a metric measure space of homogeneous type.
Then there exists a positive constant $C$
such that, for any $f\in\lon$ and $g\in\bmo$,
$$
\lf\|fg\r\|_{\llo}\le C\|f\|_{\lon}\|g\|_{\bmop}.
$$
\end{lemma}

The following conclusion is an easy consequence of \cite[Lemma 4.4]{gly1},
which is valid without resorting to the reverse doubling condition.
We present some details here for the sake of clarity,
which are similar to those used in the proof of
\cite[Lemma (2.3)]{ms2}.

\begin{theorem}\label{tf.y}
Let $(\cx,d,\mu)$ be a metric measure space of homogeneous type. Then
there exists a positive constant $C$ such that, for all $f\in\hona$,
$$
\|\cm(f)\|_{\lon}\le C\|f\|_{\hona},
$$
where $\cm$ is as in \eqref{a.u}.
\end{theorem}

\begin{proof}
We first prove that,
for any $(1,2)$-atoms $a$, supported in the ball $B_0:=B(x_0,r_0)$
for some $x_0\in\cx$ and $r_0\in(0,\fz)$,
\begin{equation}\label{5.7x}
\|\cm(a)\|_{\lon}\ls1.
\end{equation}

Indeed, we first write
$$
\int_\cx\cm(a)\,d\mu(x)=\int_{B(x_0,\,2r_0)}\cm(a)\,d\mu(x)
+\int_{\cx\bh B(x_0,\,2r_0)}\cdots=:{\rm I}+{\rm II}.
$$

To estimate I, let $\ez\in(0,1)$ and
$\bz,\,\gz\in(0,\ez)$ be as in Definition \ref{da.d}.
By (T1) of Definition \ref{da.d}, we know that,
for all $x\in\cx$, $r\in(0,\fz)$ and $h\in\cg(x,r,\bz,\gz)$
with $\|h\|_{\cg(x,r,\bz,\gz)}\le1$,
\begin{eqnarray*}
|\langle a,h\rangle|&&\le\int_{B(x_0,\,r_0)}|a(y)h(y)|\,d\mu(y)\\
&&\ls\frac1{V(x,r)}\int_{B(x_0,\,r_0)\cap B(x,\,r)}|a(y)|\,d\mu(y)\\
&&\hs+\sum_{t=0}^\fz\int_{B(x,\,2^{t+1}r)\bh B(x,\,2^t r)}
\frac{|a(y)|}{V(x,y)}\lf[\frac{r}{r+d(x,y)}\r]^\gz\,d\mu(y)\\
&&\ls Ma(x)+\sum_{t=0}^\fz2^{-t\gz}\frac1{V(x,2^t r)}
\int_{B(x,\,2^{t+1}r)}|a(y)|\,d\mu(y)\\
&&\ls Ma(x)+\sum_{t=0}^\fz2^{-t\gz}Ma(x)\ls Ma(x).
\end{eqnarray*}
Thus, for all $x\in\cx$, $\cm a(x)\ls Ma(x)$,
which, combined with the H\"older inequality, \eqref{a.b}
and the boundedness of $M$ on $\ltw$ (see \cite[(3.6)]{cw77}),
further implies that
 \begin{eqnarray*}
{\rm I}&&\ls\int_{B(x_0,\,2r_0)}Ma(x)\,d\mu(x)
\ls\|Ma\|_{\ltw}[V(x_0,r_0)]^{1/2}
\ls\|a\|_{\ltw}[V(x_0,r_0)]^{1/2}\ls1.
\end{eqnarray*}

Now we turn to estimate II. Observe that,
for all $x\not\in B(x_0,2r_0)$ and $y\in B(x_0,r_0)\cap B(x,r)$,
we see that $r>r_0$ and hence $$
d(y,x_0)<r_0=\frac{2r_0}2\le\frac{r_0+d(x,x_0)}{2}
<\frac{r+d(x,x_0)}2.
$$
Thus, from this, $\int_\cx a\,d\mu=0$, (T2) of Definition \ref{da.d}
and $r>r_0$, we deduce that, for $h\in\cg(x,r,\bz,\gz)$
satisfying that $\|h\|_{\cg(x,r,\bz,\gz)}\le1$,
\begin{eqnarray*}
|\langle a,h\rangle|&&\le\int_{B(x_0,\,r_0)}|a(y)||h(y)-h(x_0)|\,d\mu(y)\\
&&\ls\int_{B(x_0,\,r_0)\cap B(x,\,r)}
\lf[\frac{d(y,x_0)}{r+d(x,x_0)}\r]^{\bz}
\frac{|a(y)|}{V(x,r)+V(x,x_0)}\lf[\frac{r}{r+d(x,x_0)}\r]^{\gz}
\,d\mu(y)\\
&&\ls\lf[\frac{r_0}{r_0+d(x,x_0)}\r]^{\bz}
\frac1{V(x,x_0)}\|a\|_{\lon}\ls\lf[\frac{r_0}{r_0+d(x,x_0)}\r]^{\bz}
\frac1{V(x,x_0)}.
\end{eqnarray*}
This, together with \eqref{a.b}, further implies that
\begin{eqnarray*}
{\rm II}&&\ls\int_{\cx\bh B(x_0,\,2r_0)}
\lf[\frac{r_0}{r_0+d(x,x_0)}\r]^{\bz}\frac1{V(x,x_0)}\,d\mu(x)\\
&&\ls\sum_{t=1}^\fz\frac1{V(x_0,2^tr_0)}
\int_{B(x_0,\,2^{t+1}r_0)\bh B(x_0,\,2^t r_0)}
\lf[\frac{r_0}{r_0+2^t r_0}\r]^{\bz}\,d\mu(x)
\ls\sum_{t=1}^\fz 2^{-t\bz}\ls1,
\end{eqnarray*}
Thus, \eqref{5.7x} holds true.

Moreover, for all $f\in\hfa$, there exist $N\in\nn$, a sequence
$\{a_j\}_{j=1}^N$ of $(1,2)$-atoms and $\{\lz_j\}_{j=1}^N\st\cc$
such that $f=\sum_{j=1}^N\lz_ja_j$ and
$\sum_{j=1}^N|\lz_j|\ls\|f\|_{\hfa}$.
From this, the fact that $\cm$ is sublinear and \eqref{5.7x}, we deduce that
$$
\|\cm(f)\|_{\lon}\le\sum_{j=1}^N\lf|\lz_j\r|\lf\|\cm\lf(a_j\r)\r\|_{\lon}
\ls\sum_{j=1}^N\lf|\lz_j\r|\ls\|f\|_{\hfa},
$$
which, together with Theorem \ref{td.f}, implies that
$\cm$ can be extended to a bounded sublinear operator from $\hona$
into $\lon$.
This finishes the proof of Theorem \ref{tf.y}.
\end{proof}

By Remark \ref{rf.x} and Theorem \ref{tf.y}, we easily
obtain the following result, the details being omitted.

\begin{corollary}\label{cf.o}
Let $(\cx,d,\mu)$ be a metric measure space of homogeneous type.
Then $\hona\st\hlo$ and there exists a positive constant $C$
such that, for all $f\in\hona$,
$$
\|f\|_{\hlo}\le C\|f\|_{\hona}.
$$
\end{corollary}

Now we deal with the boundedness of $\Pi_2$.
Recall that $\Pi_2$ in \eqref{e.z} is bounded from
$\ltw\times\ltw$ into $\hona$ (see Remark \ref{re.g}).
To extend $\Pi_2$ into a bounded bilinear operator
from $\hona\times\bmo$ into $\hlo$, we first formally write
\begin{equation}\label{x.g}
\Pi_2(a,g):=\sum_{j\in\zz}\lf[\sum_{\bz\in\scb_j}\lf(a,\psi^j_\bz\r)
\psi^j_\bz\r]\lf[\sum_{\az\in\sca_j}\lf\langle g,{\mathfrak s}^j_\az\r\rangle
s^j_\az\r]
\end{equation}
for any $(1,2)$-atom $a$ and $g\in\bmo$, where $\sca_j$ and $\scg_j$
for any $j\in\zz$ are as in, respectively, \eqref{b.v}
and \eqref{b.w}, ${\mathfrak s}^j_\az:=s^j_\az/\nu^j_\az$
and $\nu^j_\az:=\int_{\cx}s^j_\az\,d\mu$.
We point out that, if $a,\,g\in L^2(\cx)$,
then $\Pi_2(a,g)$ in \eqref{x.g} coincides
with $\Pi_2(a,g)$ in \eqref{e.z} with $f$ replaced by $a$ and, in this case,
it is known that $\Pi_2(a,g)\in\hona$ (see Remark \ref{re.g}).

\begin{theorem}\label{tf.d}
Let $(\cx,d,\mu)$ be a metric measure space of homogeneous type. Then,
for any $(1,2)$-atom $a$ and $g\in\bmo$,
$\Pi_2(a,g)$ as in \eqref{x.g} belongs to $\hlo$ and
can be extended to a bounded bilinear operator
from $\hona\times\bmo$ into $\hlo$.
\end{theorem}

\begin{proof}
We first prove that, for any $(1,2)$-atom $a$
supported in the ball $B_0:=B(x_0,r_0)$,
with $x_0\in\cx$ and $r_0\in(0,\fz)$,
and $g\in\bmo$, $\Pi_2(a,g)$ belongs to $\hlo$ and
\begin{equation}\label{f.i}
\|\Pi_2(a,g)\|_{\hlo}\ls\|g\|_{\bmop},
\end{equation}
where the  implicit positive constant is independent of $a$ and $g$.

Let $k_0\in\zz$ satisfy $\dz^{k_0+1}\le r_0<\dz^{k_0}$
and $C_5$ be a sufficiently large positive constant which will
be determined later. We formally write
\begin{eqnarray*}
\Pi_2(a,g)&&=\sum_{j\in\zz}\lf\{\sum_{\bz\in\scg_j}
\lf(a,\psi^j_\bz\r)\psi^j_\bz\r\}
\lf\{\sum_{\az\in\sca_j}\lf(\lf[g-m_{B_0}(g)\r]
\chi_{C_5B_0},{\mathfrak s}^j_\az\r)s^j_\az\r\}\\
&&\hs+\sum_{j\in\zz}\lf\{\sum_{\bz\in\scg_j}\lf(a,\psi^j_\bz\r)\psi^j_\bz\r\}
\lf\{\sum_{\az\in\sca_j}\lf\langle\lf[g-m_{B_0}(g)\r]
\chi_{\cx\bh C_5B_0},{\mathfrak s}^j_\az\r\rangle s^j_\az\r\}\\
&&\hs+\sum_{j\in\zz}\lf\{\sum_{\bz\in\scg_j}\lf(a,\psi^j_\bz\r)\psi^j_\bz\r\}
\lf\{\sum_{\az\in\sca_j}\lf\langle m_{B_0}(g),{\mathfrak s}^j_\az\r
\rangle s^j_\az\r\}\\
&&=:\Pi_2\lf(a,\lf[g-m_{B_0}(g)\r]\chi_{C_5B_0}\r)
+\Pi_2\lf(a,\lf[g-m_{B_0}(g)\r]\chi_{\cx\bh C_5B_0}\r)
+\Pi_2\lf(a,m_{B_0}(g)\r)\\
&&=:\Pi_2^{(1)}(a,g)+\Pi_2^{(2)}(a,g)+\Pi_2^{(3)}(a,g),
\end{eqnarray*}
where $m_{B_0}(g):=[\mu(B_0)]^{-1}\int_{B_0}g\,d\mu$.

By the John-Nirenberg inequality (see \cite{cw77})
and \eqref{a.b}, we have
\begin{eqnarray*}
&&\lf\|\lf[g-m_{B_0}(g)\r]\chi_{C_5B_0}\r\|_{\ltw}\\
&&\hs\le\lf[\int_{C_5B_0}
\lf|g(x)-m_{C_5B_0}(g)\r|^2\,d\mu(x)\r]^{1/2}
+\lf[\mu(C_5B_0)\r]^{1/2}\lf|m_{B_0}(g)-m_{C_5B_0}(g)\r|\\
&&\hs\ls\lf[\mu(C_5B_0)\r]^{1/2}\|g\|_{\bmo},
\end{eqnarray*}
which, combined with $a\in\ltw$ and Remark \ref{re.g},
implies that $\Pi_2^{(1)}(a,g)$ belongs to $\hona$ and
\begin{eqnarray*}
\lf\|\Pi_2^{(1)}(a,g)\r\|_{\hona}
&&=\lf\|\Pi_2\lf(a,\lf[g-m_{B_0}(g)\r]\chi_{C_5B_0}\r)\r\|_{\hona}\\
&&\ls\|a\|_{\ltw}\lf\|\lf[g-m_{B_0}(g)\r]\chi_{C_5B_0}\r\|_{\ltw}\\
&&\ls\lf[\mu(B_0)\r]^{-1/2}\lf[\mu(C_5B_0)\r]^{1/2}\|g\|_{\bmo}
\ls\|g\|_{\bmo}.
\end{eqnarray*}

From this and Corollary \ref{cf.o}, it follows that
$\Pi_2^{(1)}(a,g)\in\hlo$ and
$$
\lf\|\Pi_2^{(1)}(a,g)\r\|_{\hlo}\ls\lf\|\Pi_2^{(1)}(a,g)\r\|_{\hona}
\ls\|g\|_{\bmo}.
$$

To deal with $\Pi_2^{(2)}(a,g)$, we first estimate
$|([g-m_{B_0}(g)]\chi_{\cx\bh C_5B_0},{\mathfrak s}^j_\az)|$
for all $(j,\az)\in\sca$ with $\sca$ as in \eqref{b.s}.
Indeed, from \eqref{b.x}, \cite[Lemma 11.1]{ah13}
and \eqref{a.b}, we deduce that
\begin{eqnarray}\label{f.f}
&&\lf|\lf(\lf[g-m_{B_0}(g)\r]\chi_{\cx\bh C_5B_0},{\mathfrak s}^j_\az\r)\r|\\
&&\noz\hs\le\frac1{V(x^j_\az,\dz^j)}
\int_{B(x^j_\az,\,8\dz^j)}\lf|g(x)-m_{B_0}(g)\r|\,d\mu(x)\\
&&\noz\hs\le\frac1{V(x^j_\az,\dz^j)}\int_{B(x^j_\az,\,8\dz^j)}\lf|g(x)
-m_{B(x^j_\az,\,8\dz^j)}(g)\r|\,d\mu(x)\\
&&\noz\hs\hs+\lf|m_{B_0}(g)-m_{B(x^j_\az,\,8\dz^j)}(g)\r|\\
&&\noz\hs\ls\|g_{\bmo}\|\lf[1+\log\frac{8\dz^j+r_0+d(x^j_\az,x_0)}
{\min\{8\dz^j,r_0\}}\r]\\
&&\noz\hs\ls\|g\|_{\bmo}\lf[1+\log\frac{\dz^j+r_0
+d(x^j_\az,x_0)}{\min\{\dz^j,r_0\}}\r].
\end{eqnarray}

By the claim \eqref{5.7z} in the proof of Theorem \ref{tf.c} that,
for any $j\in\zz$, $\az\in\sca_j$ with $\sca_j$ as in \eqref{b.v},
and $\bz\in\scg_j$ with $\scg_j$ as in \eqref{b.w},
$a^j_{\az,\,\bz}$ is a $(1,2)$-atom,
multiplied by a positive
constant, supported in $B(x^j_\az,10\dz^j)$, we find that
\begin{eqnarray*}
{\rm I}:=&&\sum_{j\in\zz}\sum_{\az\in\sca_j}\sum_{\bz\in\scg_j}
\lf|\lf(\lf[g-m_{B_0}(g)\r]\chi_{\cx\bh C_5B_0},{\mathfrak s}^j_\az\r)\r|
\lf|\lf(a,\psi^j_\bz\r)\r|\lf\|s^j_\az\psi^j_\bz\r\|_{\hona}\\
\ls&&\sum_{j\in\zz}\sum_{\az\in\sca_j}\sum_{\bz\in\scg_j}
\lf|\lf(\lf[g-m_{B_0}(g)\r]\chi_{\cx\bh C_5B_0},{\mathfrak s}^j_\az\r)\r|
\lf|\lf(a,\psi^j_\bz\r)\r|\\
&&\times\sqrt{V\lf(x^j_\az,10\dz^j\r)}
e^{-\frac{\nu}{2}\dz^{-j}d(x^j_\az,\,y^j_\bz)}\\
\sim&&\sum_{j=k_0+1}^\fz\sum_{\az\in\sca_j}\sum_{\bz\in\scg_j}
\lf|\lf(\lf[g-m_{B_0}(g)\r]\chi_{\cx\bh C_5B_0},{\mathfrak s}^j_\az\r)\r|
\lf|\lf(a,\psi^j_\bz\r)\r|\\
&&\times\sqrt{V\lf(x^j_\az,10\dz^j\r)}
e^{-\frac{\nu}{2}\dz^{-j}d(x^j_\az,\,y^j_\bz)}
+\sum_{j=-\fz}^{k_0}\cdots=:{\rm I}_1+{\rm I}_2.
\end{eqnarray*}

Observe that, for any $j\in\zz\cap[k_0+1,\fz)$ and $\az\in\sca_j$,
if $B(x^j_\az,8\dz^j)\cap(\cx\bh C_5B_0)\neq\emptyset$, then,
by $\dz^{k_0+1}<r_0$, there exists $y\in
B(x^j_\az,8\dz^j)\cap(\cx\bh C_5B_0)$ such that
$$
d(y,x_0)\ge C_5r_0>C_5\dz^{k_0+1}\ge C_5\dz^{j}\ge18\dz^j
>2d\lf(y,x^j_\az\r)
$$
provided that $C_5\ge18$, and hence
\begin{equation}\label{y.u}
d\lf(x^j_\az,x_0\r)\ge d(y,x_0)-d\lf(x^j_\az,y\r)
>\frac12d\lf(y,x_0\r)\ge \frac{C_5}{2}r_0.
\end{equation}

Meanwhile, from an argument similar to that used in the proof of
\eqref{f.g}, we deduce that, for all $j\in\zz\cap[k_0+1,\fz)$ and
$\bz\in\scg_j$,
\begin{eqnarray}\label{x.p}
\qquad\lf|\lf(a,\psi^j_\bz\r)\r|&&
\ls\int_{B_0}|a(x)|\frac{e^{-\nu\dz^{-j}d(x,y^j_\bz)}}
{\sqrt{V(y^j_\bz,\dz^j)}}\,d\mu(x)
\ls e^{-\nu\dz^{-j}d(x_0,\,y^j_\bz)}
\frac{e^{\nu\dz^{-j}r_0}}{\sqrt{V(y^j_\bz,\dz^j)}}\|a\|_{\lon}\\
&&\noz\ls e^{-\nu\dz^{-j}d(x_0,\,y^j_\bz)}
\frac{e^{\nu\dz^{-j}r_0}}{\sqrt{V(y^j_\bz,\dz^j)}}.
\end{eqnarray}

Moreover, from Remark \ref{rb.l}(ii), it follows that,
for any $j\in\zz$ and $s\in\zz_+$,
\begin{equation}\label{x.c}
\#\lf\{\az\in\sca_j:\ x^j_\az\in 2^{s+1}\frac{C_5}2B_0
\bh 2^s\frac{C_5}2B_0\r\}\ls 2^{sG_0}\lf[\frac{r_0}{\dz^j}\r]^{G_0}.
\end{equation}
By \eqref{y.u}, \eqref{x.p}, \eqref{f.f}, $\dz^j\le\dz^{k_0+1}<r_0$
for each $j\in\zz\cap[k_0+1,\fz)$,
\eqref{a.b} and $C_5\ge18$, we conclude that
\begin{eqnarray*}
{\rm I}_1&&\ls\|g\|_{\bmo}
\sum_{j=k_0+1}^\fz\sum_{\{\az\in\sca_j:\ x^j_\az\not\in \frac{C_5}2B_0\}}
\sum_{\bz\in\scg_j}
\lf[1+\log\frac{r_0+d(x^j_\az,x_0)}{\dz^j}\r]\\
&&\hs\times e^{-\nu\dz^{-j}d(x_0,\,y^j_\bz)}e^{\nu\dz^{-j}r_0}
\lf[\frac{V(x^j_\az,\dz^j)}{V(y^j_\bz,\dz^j)}\r]^{1/2}
e^{-\frac{\nu}{2}\dz^{-j}d(x^j_\az,\,y^j_\bz)}\\
&&\ls\|g\|_{\bmo}\sum_{j=k_0+1}^\fz
\sum_{\{\az\in\sca_j:\ x^j_\az\not\in \frac{C_5}2B_0\}}e^{\nu\dz^{-j}r_0}
e^{-\frac{\nu}{4}\dz^{-j}d(x_0,\,x^j_\az)}
\lf[1+\log\frac{r_0+d(x^j_\az,x_0)}{\dz^j}\r]\\
&&\hs\times \sum_{\{\bz\in\scg_j:\ d(x^j_\az,\,y^j_\bz)\ge\dz^{j+1}\}}
e^{-\frac{3\nu}{4}\dz^{-j}d(x_0,\,y^j_\bz)}
e^{-\frac{\nu}{4}\dz^{-j}d(x^j_\az,\,y^j_\bz)}
\lf[\frac{d(x^j_\az,y^j_\bz)+\dz^j}{\dz^j}\r]^{n/2}\\
&&\ls\|g\|_{\bmo}\sum_{j=k_0+1}^\fz
\sum_{\{\az\in\sca_j:\ x^j_\az\not\in \frac{C_5}2B_0\}}e^{\nu\dz^{-j}r_0}
e^{-\frac{\nu}{4}\dz^{-j}d(x_0,\,x^j_\az)}
\lf[1+\log\frac{r_0+d(x^j_\az,x_0)}{\dz^j}\r]\\
&&\hs\times \sum_{\{\bz\in\scg_j:\ d(x^j_\az,\,y^j_\bz)\ge\dz^{j+1}\}}
e^{-\frac{3\nu}{4}\dz^{-j}d(x_0,\,y^j_\bz)},
\end{eqnarray*}
which, combined with Lemma \ref{lb.x} and \eqref{x.c}, further implies that
\begin{eqnarray*}
{\rm I}_1&&\ls\|g\|_{\bmo}\sum_{j=k_0+1}^\fz
\sum_{\{\az\in\sca_j:\ x^j_\az\not\in \frac{C_5}2B_0\}}e^{\nu\dz^{-j}r_0}
e^{-\frac{\nu}{4}\dz^{-j}d(x_0,\,x^j_\az)}
\lf[1+\log\frac{r_0+d(x^j_\az,x_0)}{\dz^j}\r]\\
&&\ls\|g\|_{\bmo}\sum_{j=k_0+1}^\fz\sum_{s=0}^\fz
\sum_{\{\az\in\sca_j:\ x^j_\az\in 2^{s+1}\frac{C_5}2B_0
\bh 2^s\frac{C_5}2B_0\}}e^{\nu\dz^{-j}r_0}
e^{-\frac{\nu}{4}\dz^{-j}2^{s-1}C_5r_0}\\
&&\hs\times\lf[1+\log\frac{r_0+2^sC_5}{\dz^j}\r]\\
&&\ls\|g\|_{\bmo}\sum_{j=k_0+1}^\fz\sum_{s=0}^\fz
e^{-(2^{s+1}-1)\nu\dz^{-j}r_0}2^{sG_0}\lf(\frac{r_0}{\dz^j}\r)^{G_0}
\lf(1+s+\log\frac{r_0}{\dz^j}\r)\\
&&\ls\|g\|_{\bmo}.
\end{eqnarray*}

Furthermore, we observe that, for any given $j\in\zz\cap[-M_2,M_2]$,
with $M_2$ as in \eqref{4.1x}, and $\bz\in\scg_j$,
with $\scg_j$ as in \eqref{b.w}, by $\scg_j,\,\sca_j\st\sca_{j+1}$,
$\scg_j\cap\sca_j=\emptyset$ and \eqref{b.z},
we know that $\az\in\sca_j$ if and
only if $\az\in\sca_j$ and $d(x^j_{\az},y^j_\bz)\ge\dz^{j+1}$.
By this, \eqref{f.f}, \eqref{f.h}, $\dz^j\ge\dz^{k_0}>r_0$
for any $j\in\zz\cap(-\fz,k_0]$, \eqref{a.b} and Lemma \ref{lb.x},
we obtain
\begin{eqnarray*}
{\rm I}_2&&\ls\|g\|_{\bmo}
\sum_{j=-\fz}^{k_0}\lf[\frac{r_0}{\dz^j}\r]^{\eta}
\sum_{\az\in\sca_j}\sum_{\bz\in\scg_j}
\lf[1+\log\frac{\dz^j+d(x^j_\az,x_0)}{r_0}\r]\\
&&\hs\times e^{-\nu\dz^{-j}d(x_0,\,y^j_\bz)}
e^{-\frac{\nu}{2}\dz^{-j}d(x^j_\az,\,y^j_\bz)}
\lf[\frac{V(x^j_\az,\dz^j)}{V(y^j_\bz,\dz^j)}\r]^{1/2}\\
&&\ls\|g\|_{\bmo}
\sum_{j=-\fz}^{k_0}\lf[\frac{r_0}{\dz^j}\r]^{\eta}
\sum_{\bz\in\scg_j}e^{-\frac{3\nu}{4}\dz^{-j}d(x_0,\,y^j_\bz)}
\sum_{\{\az\in\sca_j:\ d(x^j_\az,\,y^j_\bz)\ge\dz^{j+1}\}}
e^{-\frac{\nu}{4}\dz^{-j}d(x^j_\az,\,x_0)}\\
&&\hs\times\lf[1+\log\frac{\dz^j+d(x^j_\az,x_0)}{r_0}\r]
e^{-\frac{\nu}{4}\dz^{-j}d(x^j_\az,\,y^j_\bz)}
\lf[\frac{d(x^j_\az,y^j_\bz)+\dz^j}{\dz^j}\r]^{n/2}\\
&&\ls\|g\|_{\bmo}
\sum_{j=-\fz}^{k_0}\lf[\frac{r_0}{\dz^j}\r]^{\eta}
\sum_{\bz\in\scg_j}e^{-\frac{3\nu}{4}\dz^{-j}d(x_0,\,y^j_\bz)}
\sum_{\{\az\in\sca_j:\ d(x^j_\az,\,y^j_\bz)\ge\dz^{j+1}\}}
e^{-\frac{\nu}{4}\dz^{-j}d(x^j_\az,\,x_0)}\\
&&\hs\times\lf[1+\log\frac{\dz^j+d(x^j_\az,x_0)}{r_0}\r]\\
&&\ls\|g\|_{\bmo}
\sum_{j=-\fz}^{k_0}\lf[\frac{r_0}{\dz^j}\r]^{\eta}
\sum_{\az\in\sca_j}
e^{-\frac{\nu}{4}\dz^{-j}d(x^j_\az,\,x_0)}
\lf[1+\log\frac{\dz^j+d(x^j_\az,x_0)}{r_0}\r].
\end{eqnarray*}

Moreover, notice that, for any $j\in\zz$ and $t\in\zz_+$,
from Remark \ref{rb.l}(ii), we deduce that
$
\#\{\az\in\sca_j:\ 2^t\dz^j\le d(x^j_\az,\,x_0)<2^{t+1}\dz^j\}
\ls2^{tG_0}
$,
which implies that
\begin{eqnarray*}
&&\sum_{\az\in\sca_j}
e^{-\frac{\nu}{4}\dz^{-j}d(x^j_\az,\,x_0)}
\lf[1+\log\frac{\dz^j+d(x^j_\az,x_0)}{r_0}\r]\\
&&\hs\ls\sum_{\{\az\in\sca_j:\ d(x^j_\az,\,x_0)<\dz^j\}}
\lf[1+\log\frac{\dz^j}{r_0}\r]\\
&&\hs\hs+\sum_{t=0}^\fz\sum_{\{\az\in\sca_j:\ 2^t\dz^j\le
d(x^j_\az,\,x_0)<2^{t+1}\dz^j\}}
e^{-\frac{\nu}{4}2^t}\lf[1+\log\frac{\dz^j+2^t\dz^j}{r_0}\r]\\
&&\hs\ls1+\log\frac{\dz^j}{r_0}
+\sum_{t=0}^\fz2^{(t+1)G_0}
e^{-\frac{\nu}{4}2^t}(1+t)\lf[1+\log\frac{\dz^j}{r_0}\r]
\ls1+\log\frac{\dz^j}{r_0}.
\end{eqnarray*}
By this, we further conclude that
$$
{\rm I}_2\ls\|g\|_{\bmo}
\sum_{j=-\fz}^{k_0}\lf[\frac{r_0}{\dz^j}\r]^{\eta}
\lf[1+\log\frac{\dz^j}{r_0}\r]\ls\|g\|_{\bmo},
$$
which, combined with Corollary \ref{cf.o}, the completion of $\hona$
and the estimate of $\rm I_1$,
further implies that $\Pi_2^{(2)}(a,g)\in\hlo$ and
$$
\lf\|\Pi_2^{(2)}(a,g)\r\|_{\hlo}\ls\lf\|\Pi_2^{(2)}(a,g)\r\|_{\hona}
\ls{\rm I}\ls{\rm I}_1+{\rm I}_2\ls\|g\|_{\bmo}.
$$

Finally, we deal with $\Pi_2^{(3)}(a,g)$.
By \eqref{b.y}, $a\in\ltw$ and Theorem \ref{tb.a}, we have
\begin{equation}\label{5.10y}
\Pi_2(a,1)=a.
\end{equation}
From this, Remark \ref{rf.x}, Lemmas \ref{lf.u} and \ref{lf.v},
it follows that
\begin{eqnarray*}
\lf\|\Pi_2^{(3)}(a,g)\r\|_{\llo}
&&\ls\lf\|\lf|m_{B_0}(g)-g\r|\cm(a)\r\|_{\llo}
+\lf\||g|\cm(a)\r\|_{\llo}\\
&&\ls\lf\|\lf|m_{B_0}(g)-g\r|\cm(a)\r\|_{\lon}
+\|a\|_{\lon}\||g|\|_{\bmop}\\
&&\ls\|g\|_{\bmo}+\|g\|_{\bmop}\ls\|g\|_{\bmop},
\end{eqnarray*}
which, combined with the estimates of $\Pi_2^{(1)}(a,g)$
and $\Pi_2^{(2)}(a,g)$, implies that $\Pi_2(a,g)$
belongs to $\hlo$ and \eqref{f.i} holds true.

From the above proof of \eqref{f.i},
we deduce that there exists $h:=\Pi^{(1)}_2(a,g)
+\Pi^{(2)}_2(a,g)\in\hona$ satisfying that
$\|h\|_{\hona}\ls\|g\|_{\bmo}$ and
$$\Pi_2(a,g)=h+am_{B_0}(g),$$
which, together with Lemmas \ref{lf.v} and \ref{lf.u}, and some
arguments similar to those used in the proof of
\cite[(5.6)]{bgk}, further implies that, for all
$f\in\hfa$ and $g\in\bmo$,
\begin{equation*}
\lf\|\Pi_2(f,g)\r\|_{\hlo}=
\lf\|\cm\lf(\Pi_2(f,g)\r)\r\|_{\llo}\ls\|f\|_{\hfa}\|g\|_{\bmo}.
\end{equation*}
By this, Theorem \ref{td.f} with $T(\cdot):=\Pi_2(\cdot,g)$
and $A\sim\|g\|_{\bmo}$ for any fixed $g\in\bmo$
and the fact that $\hlo$ is a quasi-Banach space
(see, for example, \cite[Section 2.4]{k14}),
we know that $\Pi_2$ can be extended to a bounded bilinear
operator from $\hona\times\bmo$ into $\hlo$, which
completes the proof of Theorem \ref{tf.d}.
\end{proof}

Now we are ready to prove Theorem \ref{ta.a}.

\begin{proof}[Proof of Theorem \ref{ta.a}]
We first claim that, to show Theorem \ref{ta.a}, it suffices to prove that,
for any $(1,2)$-atom $a$, supported in a ball $B_0:=B(x_0,r_0)$
for some $x_0\in\cx$ and $r_0\in(0,\fz)$,
and $g\in\bmo$,
\begin{equation}\label{f.b}
a\times g=\Pi_1(a,g)+\Pi_2(a,g)+\Pi_3(a,g) \quad {\rm in}\quad
\lf(\cg^\ez_0(\bz,\gz)\r)^*
\end{equation}
with $\ez$, $\bz$ and $\gz$ as in Theorem \ref{ta.a}.

Assuming that \eqref{f.b} holds true, we now show the conclusion
of Theorem \ref{ta.a}. Indeed, for any $f\in\hona$, by Definition \ref{dc.k},
we know that there exist a sequence $\{a_j\}_{j\in\nn}$
of $(1,2)$-atoms and $\{\lz_j\}_{j\in\nn}\st\cc$ such that
$f=\sum_{j\in\nn}\lz_j a_j$ in $\hona$.
For each $N\in\nn$, let $f_N:=\sum_{j=1}^N\lz_j a_j$.
We then have
\begin{equation}\label{y.f}
\lim_{N\to\fz}f_N=f\quad{\rm in}\quad \hona.
\end{equation}

From \eqref{f.b}, it follows easily that
\begin{equation}\label{y.e}
f_N\times g=\Pi_1\lf(f_N,g\r)+\Pi_2\lf(f_N,g\r)
+\Pi_3\lf(f_N,g\r) \quad {\rm in}\quad\lf(\cg^\ez_0(\bz,\gz)\r)^*.
\end{equation}
We now show that
\begin{equation}\label{y.d}
\lim_{N\to\fz}f_N\times g=f\times g\quad {\rm in}\quad\lf(\cg^\ez_0(\bz,\gz)\r)^*.
\end{equation}

To this end, for any $h\in\cg^\ez_0(\bz,\gz)$, by Lemma \ref{lf.t}
and \eqref{y.f}, we obtain
\begin{eqnarray*}
&&\lf|\langle f_N\times g,h\rangle-\langle f\times g,h\rangle\r|\\
&&\noz\hs=\lf|\langle gh,f_N-f\rangle\r|
\le\lf\|gh\r\|_{\bmo}\lf\|f_N-f\r\|_{\hona}\\
&&\noz\hs\ls\frac1{V_1(x_1)}\|h\|_{\cg(\bz,\,\gz)}\|g\|_{\bmop}
\lf\|f_N-f\r\|_{\hona}\to 0, \quad {\rm as} \quad N\to\fz,
\end{eqnarray*}
which shows \eqref{y.d}.

Moreover, from \eqref{y.f}, Theorems \ref{tf.b}, \ref{tf.c}
and \ref{tf.d}, we deduce that
$$\lim_{N\to\fz}\Pi_3(f_N,g)=\Pi_3(f,g)\quad{\rm in}\quad
\lon,$$
$\lim_{N\to\fz}\Pi_1(f_N,g)=\Pi_1(f,g)$ in $\hona$, and
$\lim_{N\to\fz}\Pi_2(f_N,g)=\Pi_2(f,g)$ in $\hlo$,
which immediately imply that they all also hold true in
$(\cg^\ez_0(\bz,\gz))^*$.
By these facts, \eqref{y.f}, \eqref{y.e} and \eqref{y.d},
we conclude that
\begin{eqnarray*}
f\times g&&=\lim_{N\to\fz}f_N\times g
=\lim_{N\to\fz}\lf[\Pi_1\lf(f_N,g\r)+\Pi_2\lf(f_N,g\r)+\Pi_3\lf(f_N,g\r)\r]\\
&&=\Pi_1(f,g)+\Pi_2(f,g)+\Pi_3(f,g)\quad {\rm in}\quad
\lf(\cg^\ez_0(\bz,\gz)\r)^*,
\end{eqnarray*}
which, combined with Theorems \ref{tf.b}, \ref{tf.c}
and \ref{tf.d}, then completes the proof of Theorem \ref{ta.a}
with $\scl:=\Pi_3$ and $\sch:=\Pi_1+\Pi_2$.

Now we show \eqref{f.b}.
By Theorem \ref{tf.a} and Remark \ref{rf.z}, we know that
\begin{equation}\label{f.1}
\wz{g}:=\sum_{j\in\zz}\sum_{\bz\in\scg_j}\lf\langle g,\psi^j_\bz\r\rangle
\lf[\psi^j_\bz-\chi_{\{k\in\zz:\ \dz^k>r_0\}}(j)\psi^j_\bz(x_0)\r]
\end{equation}
converges in both $L^2_{\loc}(\cx)$ and $\bmo$.

Now we choose a fixed collection
$\lf\{\scc_N:\ N\in\nn,\ \scc_N\st\scc\ {\rm and\ \scc_N\ is\ finite}\r\}$
as in \eqref{4.5x} and let
$$\wz{g}_N:=\sum_{(j,\,\bz)\in\scc_N}\lf\langle g,\psi^j_\bz\r\rangle
\lf[\psi^j_\bz-\chi_{\{k\in\zz:\ \dz^k>r_0\}}(j)\psi^j_\bz(x_0)\r]
=\sum_{(j,\,\bz)\in\scc_N}\lf\langle g,\psi^j_\bz\r\rangle  \psi^j_\bz=:g_N$$
in $\bmo$.

By the finiteness of $\scc_N$, we know that $g_N\in\ltw$,
which, together with Lemmas \ref{le.a}, \ref{le.b} and \ref{le.c},
and Remark \ref{re.g}, further implies that, for any $N\in\nn$,
\begin{equation}\label{f.z}
ag_N=\Pi_1\lf(a,g_N\r)+\Pi_2\lf(a,g_N\r)
+\Pi_3\lf(a,g_N\r) \quad {\rm in}\quad \lon.
\end{equation}

Then we claim that, for all $h\in\cg^{\ez}_0(\bz,\gz)$,
$
\lim_{N\to\fz}\langle a\times\wz{g}_N,h\rangle
=\langle a\times\wz{g},h\rangle.
$
Indeed, by the definition of the distribution,
the duality between $\hona$ and $\bmo$,
Lemma \ref{lf.t} and \eqref{f.1},
we conclude that
\begin{eqnarray}\label{f.y}
&&\lf|\langle a\times\wz{g}_N,h\rangle-\langle a\times\wz{g},h\rangle\r|\\
&&\noz\hs=\lf|\langle\lf(\wz{g}_N-\wz{g}\r)h,a\rangle\r|
\le\lf\|\lf(\wz{g}_N-\wz{g}\r)h\r\|_{\bmo}\|a\|_{\hona}\\
&&\noz\hs\ls\frac1{V_1(x_1)}\|h\|_{\cg(\bz,\,\gz)}
\lf\|\wz{g}_N-\wz{g}\r\|_{\bmop}\\
&&\noz\hs\ls\frac1{V_1(x_1)}\|h\|_{\cg(\bz,\,\gz)}
\lf[\lf\|\wz{g}_N-\wz{g}\r\|_{\bmo}
+\frac1{\sqrt{V_1(x_1)}}\lf\|\lf(\wz{g}_N-\wz{g}\r)
\chi_{B(x_1,\,1)}\r\|_{\ltw}\r]\\
&&\noz\hs\to0,\quad N\to\fz.
\end{eqnarray}
This proves the above claim.

By Remark \ref{rf.z}, we know that $g-\wz{g}=:c_4$ is
a constant. Let $c_{(N)}:=\wz{g}_N-g_N$ for any $N\in\nn$.
It is easy to see that $c_{(N)}$ is a constant, depending on $N$,
for each $N\in\nn$.
From this, \eqref{f.y}, \eqref{f.1}, \eqref{f.z},
Theorems \ref{tf.b}, \ref{tf.c}
and \ref{tf.d}, $\Pi_2(a,1)=a$ and \eqref{b.c}, we deduce that
\begin{eqnarray*}
a\times g&&=a\times\wz{g}+c_4 a=\lim_{N\to\fz}a\times\wz{g}_N+c_4a
=\lim_{N\to\fz}\lf[ag_N+c_{(N)}a\r]+c_4a\\
&&=\lim_{N\to\fz}\lf[\Pi_1\lf(a,g_N\r)+\Pi_2\lf(a,g_N\r)
+\Pi_3\lf(a,g_N\r)+c_{(N)}\Pi_2(a,1)\r]+c_4 a\\
&&=\lim_{N\to\fz}\lf[\Pi_1\lf(a,\wz{g}_N\r)+\Pi_2\lf(a,\wz{g}_N\r)
+\Pi_3\lf(a,\wz{g}_N\r)\r]+c_4 a\\
&&=\Pi_1\lf(a,\wz{g}\r)+\Pi_3\lf(a,\wz{g}\r)
+\lf[\Pi_2\lf(a,\wz{g}\r)+c_4\Pi_2(a,1)\r]\\
&&=\Pi_1\lf(a,g\r)+\Pi_2\lf(a,g\r)
+\Pi_3(a,g)\quad {\rm in}\quad (\cg^{\ez}_0(\bz,\gz))^*,
\end{eqnarray*}
which completes the
proof of \eqref{f.b} and hence Theorem \ref{ta.a}.
\end{proof}

\bigskip

Xing Fu, Dachun Yang (Corresponding author) and Yiyu Liang

\medskip

School of Mathematical Sciences, Beijing Normal University,
Laboratory of Mathematics and Complex Systems, Ministry of
Education, Beijing 100875, People's Republic of China

\smallskip

{\it E-mails}: \texttt{xingfu@mail.bnu.edu.cn} (X. Fu)

\hspace{1.55cm}\texttt{dcyang@bnu.edu.cn} (D. Yang)

\hspace{1.55cm}\texttt{liangyiyu@mail.bnu.edu.cn} (Y. Liang)

\end{document}